%% file: CosmicFields_Apr1_ArXiv.tex
 \documentclass{article} 
 
\pdfoutput=1  

\usepackage{fncylab,enumerate,wrapfig,placeins}

\usepackage{amsmath}

\usepackage{amsthm,amsfonts,amssymb, graphicx}  
\usepackage{relsize} 
\usepackage{accents}  

\usepackage{caption}
 \usepackage{textgreek,upgreek,bm}

 \usepackage{titlesec}

\usepackage{geometry} 

\input{./bookmacrosRL_Journal.tex}

\def\bdd#1{b^{\text{\rm\tiny\ref{#1}}}}
\def\bdde#1{\varrho^{\text{\rm\tiny\ref{#1}}}}

  \def\ctr{\text{\tiny{\sf ctr}}}
\def\haqsaprobe{\hat{\qsaprobe}}
\def\cqsaprobe{\check{\qsaprobe}}
\def\haclYn{\widehat{\mathcal{Y}}^{\text{\tiny\sf n}}}

\def\Df{D^f}
\def\hacqsaprobe{\widehat{\cqsaprobe}}
\def\rF{\text{\rm F}}
\def\rG{\text{\rm G}}
\def\rH{\text{\rm H}}
\def\rJ{\text{\rm J}}
\def\hahaqsaprobe{\doublehat{\qsaprobe}}
\def\hahaSigma{\widehat{\widehat{\Sigma}}}

\def\bias{\beta}

\def\haUpupsilon{\widehat{\Upupsilon}}
\def\whamrm#1{\smallbreak\pagebreak[3]%
	\noindent{\rm#1}\ \ \gobblepars}

\def\haqsaprobe{\hat{\qsaprobe}}
\def\cqsaprobe{\check{\qsaprobe}}
\def\haqsaprobe{\hat{\qsaprobe}}
\def\tilh{\tilde{h}}

\def\clT{\mathcal{T}}

\def\clYn{\mathcal{Y}^{\text{\footnotesize\sf n}}}
\def\chclYn{\check{\mathcal{Y}}^{\text{\footnotesize\sf n}}}

\def\Sigmaqsa{\Sigma_{\text{\tiny$\qsaprobe$}}}

 \def\tilodestate{\tilde\odestate}

\def\barftwo{{\bar{f}}}

\def\thetaPR{\theta^{\text{\tiny\sf PR}}}

\def\ODEstateF{\ODEstate^{\text{\tiny\sf F}}}

 \def\XF{X^{\text{\tiny\sf F}}}

 \def\WF#1{W^{\text{\tiny{$#1$\sf F}}}}

\def\ODEstatePR{\ODEstate^{\text{\tiny\sf PR}}}

\makeatletter
\newcommand\gobblepars{%
    \@ifnextchar\par%
        {\expandafter\gobblepars\@gobble}%
{}}
\makeatother

\def\ODEstatePR{\ODEstate^{\text{\tiny\sf PR}}}

\def\clD{\mathcal{D}}

\def\opt{\text{\tiny\sf opt}}

\graphicspath{{./Figures/}}

\newtheorem{theorem}{Theorem}[section]
\newtheorem{lemma}[theorem]{Lemma}
\newtheorem{proposition}[theorem]{Proposition}
\newtheorem{corollary}[theorem]{Corollary}

\usepackage[usenames,dvipsnames]{color}
\usepackage{color}
 \definecolor{programcode}{gray}{0.9}
 
 \definecolor{lightgray}{gray}{0.7}

\definecolor{MyDarkBlue}{cmyk}{0.5,0.1,0,0.9}

\usepackage[colorlinks,%
linkcolor=BrickRed,%
filecolor =BrickRed,%
citecolor=RoyalPurple,%
]{hyperref}

\newlength{\noteWidth}
\long\def\notes#1{\ifinner
{\footnotesize #1}
\else 
\marginpar{\parbox[t]{\noteWidth}{\raggedright\footnotesize#1}}
\fi\typeout{#1}}

 \def\notes#1{\typeout{check notes!!!}}   

\def\bl#1{{\color{blue}#1}}

\def\sfb#1{}

\usepackage{cleveref}

\Crefname{corollary}{Corollary}{Corollaries}
\Crefname{eqnarray}{eq.}{eqs.}
\Crefname{equation}{eq.}{eqs.}

\Crefname{figure}{Fig.}{Figs.}
\Crefname{tabular}{Tab.}{Tabs.}
\Crefname{table}{Tab.}{Tabs.}
\Crefname{proposition}{Prop.}{Propositions}
\Crefname{theorem}{Thm.}{Thms.}
\Crefname{definition}{Def.}{Defs.} 
\Crefname{section}{Section}{Sections}
\Crefname{lemma}{Lemma}{Lemmas}
\Crefname{assumption}{Assumption}{Assumptions}

\def\wham#1{\smallbreak\pagebreak[3]%
	\noindent\textbf{#1}\ \ \gobblepars}

\def\witem{\wham{\small$\triangle$}}

 \def\bqed{{\color{blue} \qedb} \bigskip}

\title{Markovian Foundations for Quasi-Stochastic Approximation
\\
with Applications to Extremum Seeking Control}

\author{Caio Kalil Lauand
	\and Sean Meyn
	\thanks{CKL and SPM are with the University of Florida, Gainesville, FL 32611 		Financial support from ARO award W911NF2010055
		and National Science Foundation award  EPCN 1935389
		is gratefully acknowledged.}%
}

\def\lilb{b^\bullet}
\def\bigb{b^\circ}

\begin{document}
 
\maketitle

 \begin{abstract}

This paper concerns  quasi-stochastic approximation (QSA) to solve root finding problems commonly found in applications to optimization and reinforcement learning.     The general constant gain algorithm may be expressed as the time-inhomogeneous ODE $ \ddt \ODEstate_t = \alpha f_t (\ODEstate_t)$,  with state process $\ODEstate$ evolving on $\Re^d$.    Theory is based on an almost periodic vector field, so that in particular the time average of $f_t(\theta)$ defines the time-homogeneous mean vector field $\barf\colon\Re^d\to\Re^d$ with $\barf(\theta^*) = \Zero$.    Under smoothness assumptions on the functions involved, the following exact representation is obtained:
\[
	\ddt \ODEstate_t    =   \alpha [\barf (\ODEstate_t)   -\alpha \barUpupsilon_t +    \alpha^2 \clW_t^0  +   \alpha \frac{d}{dt}   \clW_t^1 +\frac{d^2}{dt^2}   \clW_t^2]
\]
along with formulae for the smooth signals $\{ \barUpupsilon_t, \clW_t^i :  i=0, 1, 2\}$.    This representation is based on the application of techniques from Markov processes,  for which Poisson's equation plays a central role.    This new representation, combined with new conditions for ultimate boundedness, has many applications for furthering the theory of QSA and its applications, including the following implications that are developed in this paper:  
\begin{alphanum}
\item 
A proof that the estimation error $\| \ODEstate_t - \theta^* \|$ is of order $O(\alpha)$, but can be reduced to $O(\alpha^2)$  using a second order linear filter.

\item
In application to extremum seeking control  
	[an approach to gradient free optimization], 
it is found that the results do not apply because the standard algorithms are not Lipschitz continuous.   A new approach is presented to ensure that the required Lipschitz bounds hold, and from this we obtain stability, transient bounds,   asymptotic bias of order  $O(\alpha^2)$,
 and asymptotic variance of order  $O(\alpha^4)$.    
\item It is in general possible to obtain better than 
$O(\alpha)$  bounds on error in traditional stochastic approximation when there is Markovian noise.

\end{alphanum}

\noindent
AMS-MSC Codes: 62L20, 34C29
\end{abstract}

\clearpage

\tableofcontents

\clearpage
\section{Introduction}
\label{s:intro}

The basic problem of interest in this paper is root finding:  find or approximate a solution $\theta^\ocp\in\Re^d$  to $\barf(\theta^\ocp) = 0$,  where  $\barf\colon\Re^d \to\Re^d$ may not be available in closed form, but noisy measurements of $\barf(\theta)$ are available for any $\theta$.      There is a large library of solution techniques.

\subsection{Root finding with noisy observations}

The field of stochastic approximation (SA) was born in the work of Robbins and Monro \cite{robmon51a}. The goal is to solve the root finding problem in which the function
 $\barf\colon\Re^d\to\Re^d$ is expressed as the expectation  
	\begin{equation}
		\barf(\theta) \eqdef \Expect[f(\theta,\qsaprobe)]
		\label{e:barf}
\end{equation}
with $\qsaprobe$ a random vector taking values in $\Re^m$,   and $f\colon\Re^d\times\Re^m\to\Re^d$.

	The basic SA algorithm is expressed as the $d$-dimensional recursion,
\begin{equation}
	\theta_{n+1} = \theta_{n}+ \alpha_{n+1} f(\theta_n,\qsaprobe_{n+1}) \, , \qquad n\ge 0,
	\label{e:SA_recur}
\end{equation}
in which $\{\alpha_n\}$ is the non-negative  step-size sequence, and $\{\qsaprobe_{n+1}\}$ converges to $\qsaprobe$ in distribution as $n \to \infty$.  
A common choice in theory is to opt for a step-size sequence that vanishes in $n$;   a typical choice  is $\alpha_n = n^{-\rho}$ with $\rho \in (1/2,1]$ \cite{CSRL,bor20a}.  

Writing $\tilXi_n =  f(\theta_n,\qsaprobe_{n+1})  -  \barf(\theta_n)$, the SA recursion can be expressed in the suggestive form,
\begin{equation}
	\theta_{n+1} = \theta_{n}+ \alpha_{n+1} \{ \barf(\theta_n) +\tilXi_n  \}\, , \qquad n\ge 0  
	\label{e:SA_recurDelta}
\end{equation}
This is interpreted as a noisy Euler approximation of the \textit{mean flow}:
\begin{equation}
	\ddt \odestate_t  = \barf  ( \odestate_t  )
	\label{e:barf_ODE_gen}
\end{equation}
Convergence   theory for SA is based on the recognition that the identity $\barf(\theta^\ocp) =0$ means that $\theta^\ocp$ is the stationary point for this ODE.   Convergence of the SA recursion is established under the assumption that  	\eqref{e:barf_ODE_gen} is globally asymptotically stable,  along with some regularity conditions on the noise,  and the standard assumptions ensuring success of noiseless Euler approximations \cite{bor20a}.    
Note that in early work it was assumed 
that the sequence $\bfqsaprobe \eqdef \{ \qsaprobe_n  : n\ge 1\}$ is i.i.d.\  (independent and identically distributed).   It is made clear in  \cite{bor20a} that convergence does not require such strong assumptions;  it is only when we turn to rates of convergence that we require strong assumptions and far more elaborate analysis.  

While there have been exciting advances in SA theory over recent decades,    these advances  do not allow us to escape an unfortunate truth:    in the majority of applications of SA, the mean square error decays no faster than  $1/n$---\textit{it would appear that there is  no way to avoid the Central Limit Theorem}.
This curse of variance is true in broad generality. Fortunately  there is a remedy made possible by the flexibility we have in many  application domains, such as optimization and reinforcement learning:  in these applications it is  \textit{we who design the noise} for purposes of ``exploration''.    
Much more efficient algorithms are obtained by abandoning randomness in the algorithm.

It is also convenient to perform analysis in continuous time, so that the recursion \eqref{e:SA_recur} is replaced with an ordinary differential equation.

\wham{Quasi-Stochastic Approximation} 
	is a deterministic analog of stochastic approximation.   The  \textit{QSA ODE} is defined by the ordinary differential equation
\begin{equation}
		\ddt\ODEstate_t = a_t f(\ODEstate_t,\qsaprobe_t) 
\label{e:QSA_ODE_gen}
\end{equation}
where $f\colon\Re^d\times \Re^m \to \Re^d$ is continuous and the gain process $\{a_t\}$ is analogous to $\{\alpha_n\}$.  The $m$-dimensional deterministic continuous time process $\bfqsaprobe$ is called the \textit{probing signal}.  

Subject to mild assumptions,   it was shown in \cite{laumey22,laumey22e} that the algorithm can be designed so that the  squared error decays at rate $1/t^{4\rho}$ for the vanishing gain algorithm, using
\begin{equation}
a_t = \alpha/(1+t/t_e)^\rho  
\label{e:arho}
\end{equation}
with  $\rho \in (\half, 1)$,   and $\alpha>0$,  $t_e\ge 1$ arbitrary constants.  

The present article focuses on the case   $\rho=0$,  meaning the gain is constant:  $a_t \equiv \alpha$.

It is assumed throughout the paper that the  probing signal is a nonlinear function of sinusoids:  
$\qsaprobe_t  = G_0(\qsaprobe _t^0  )$,  with $G_0\colon\Re^K \to\Re^m$ continuous,  and
\begin{equation}
	\qsaprobe _t^0  =  [  \cos (2\pi [\,  \omega_1 t  +  \phi_1 ] ) ;\,  \dots  ;\,    \cos (2\pi [\,  \omega_K t  +  \phi_K ] )  ]
\label{e:qsaprobe0}
\end{equation}
Motivation for a nonlinearity may be to create rich probing signals from simple ones.   For example,  with $K=1$ we can obtain 
		 $\qsaprobe_t  =    [  \cos (2\pi \omega_1 t ) ; \, \cos (4\pi \omega_1 t ) ;\, \dots  ; \,  \cos (2m\pi \omega_1 t   )  ]$ for any $m\ge 1$,  with  $G_0\colon\Re \to\Re^m$ polynomial.    
On choosing $G_0$ linear we obtain, for vectors $\{v^i\}\subset \Re^m$,
\begin{equation}
		\qsaprobe _t  = \sum_{i=1}^K  v^i   \cos (2\pi [  \omega_i t  +  \phi_i  ] )
\label{e:qSGD_probe}
\end{equation}

For analysis it is crucial   to abandon $\qsaprobe _t^0  $ in favor of the   $K$-dimensional clock-process   $\bfPhi$ with entries $\Phi_t^i =   \exp(2\pi j[\omega_i t + \phi_i])$.
On defining 
 $ G(z) = G_0( (z+1/z)/2  )$, where $1/z \eqdef (1/z_1,\dots,1/z_K)$ for $z\in \{\Co\setminus \{0\}\}^K$, we obtain 
\begin{equation}
\qsaprobe_t  = G( \Phi_t) 
\label{e:qsaprobeG}
\end{equation}
The function $G$ will be assumed continuous on the restricted domain $ \{\Co\setminus \{0\}\}^K$.  	

The clock-process   $\bfPhi$ is Markovian.    It can also be expressed as the solution to the linear system,
\begin{equation}
		\ddt \Phi_t = W \Phi_t, \quad \text{with } W =2\pi j \diag(\omega_i) 
		\label{e:LinearMarkovQSA}
\end{equation}
It evolves on a compact subset of Euclidean space denoted $\prstate \subset\Co^K$,  and it is evident that the uniform distribution on $\prstate$ is the unique invariant measure for $\bfPhi$;   this is denoted $ \uppi$.

The \textit{mean vector field} associated with \eqref{e:QSA_ODE_gen} is defined as the expectation,
\begin{equation}
		\barf(\theta) =   \Expect[ f(\theta, G(\Phi))  ]\,,\qquad \theta\in\Re^d\,,  \qquad \textit{in which  $ \Phi\sim \uppi$. 
}
\label{e:barfQSA}
\end{equation}

	The Law of Large Numbers holds:
$\lim_{T \to \infty } \frac{1}{T}    \int^T_0 g(\Phi_t)\, dt  = \barg$  for each initial condition $\Phi_0$, where $\barg \eqdef \int g(z) \uppi(dz)$ \cite[Ch. 8]{kha02}. This is commonly used in this paper with $g(z) = h(G(z))$, so that $g(\Phi_t) = h(\qsaprobe_t)$.


\subsection{Perturbative mean flow}

The central conclusion on which most of the concepts in this paper are based upon is the following representation for the QSA ODE, which holds under smoothness conditions on $f$ and $G_0$.

\wham{Perturbative Mean Flow:}  The  solution to the QSA ODE admits the exact description
 \begin{equation} 
		\begin{aligned}
			\ddt \ODEstate_t   & =  \alpha[ \barf (\ODEstate_t)  
						  -  
						\alpha \barUpupsilon_t
						  +   \clW_t]
			\\
			& \qquad    \clW_t =   \alpha^2 \clW_t^0   +  \alpha \frac{d}{dt} \clW_t^1 +    \frac{d^2}{dt^2} \clW_t^2
		\end{aligned} 
\label{e:BigGlobalODE}
\end{equation}	
The details: 
\begin{romannum}
\item
The deterministic processes  $\{\clW_t^i         : i=0,1,2\}$   have explicit representations,  
	 given in \eqref{e:AllTheNoise0}--\eqref{e:AllTheNoise2},
	 as smooth functions of the larger state process $\bfPsi =  (\bfODEstate,\bfPhi)$.
		
		\item The function $\barUpupsilon_t = \barUpupsilon(\ODEstate_t)  $ appears when there is multiplicative noise in the QSA ODE.    It can contribute significantly to the estimation error $\|\ODEstate_t - \theta^*\|$, resulting in large bias and variance. Fortunately, it can be eliminated with careful design.
		 
\end{romannum}
The implications of the perturbative mean flow (or P-mean flow) representation to algorithm design are focuses of this paper.   

However,  there is one catch:  while the P-mean flow representation holds in broad generality, we cannot use it to establish stability (in the sense of ultimate boundedness) since to-date we have not found global Lipschitz bounds on the functions  
$\{ \barUpupsilon,  \clW^i        : i=0,1,2\}$.

\wham{Algorithm Performance:} The pair $\bfPsi =  (\bfODEstate,\bfPhi)$ is a Feller Markov process.   Existence of an invariant measure $\upvarpi \sim \bfPsi  $ is guaranteed for QSA whenever the sample path $\bfODEstate$ is bounded from at least one initial condition. 
Conditions under which $\upvarpi$ is unique are given in \Cref{s:coupling} so that bias and variance are defined as follows

\begin{equation}
	\begin{aligned}
		\bias_\ODEstate  &= \lim_{T\to\infty} \Bigl\| \frac{1}{T}\int_0^T [\ODEstate_t - \theta^*] \, dt  \Bigr\| \,,
		\\
		\sigma^2_\ODEstate  &=   \Bigl(  \lim_{T\to\infty}    \frac{1}{T}\int_0^T \| \ODEstate_t -\theta^\ocp \|^2 \, dt  \Bigr) 
		-   \|\bias_\ODEstate\|^2  \,.
	\end{aligned} 
	\label{e:biasdef}
\end{equation}

The standard $L_p$ norms are also considered in their sample path forms: 
\begin{equation}
	\begin{aligned}
		\| \ODEstate -\theta^*\|_{L_1} = \lim_{T\to\infty}    \frac{1}{T}\int_0^T \| \ODEstate_t -\theta^*\| \, dt    \, ,
		\\
		\| \ODEstate -\theta^*\|_{L_2} = \sqrt{\lim_{T\to\infty}    \frac{1}{T}\int_0^T \| \ODEstate_t -\theta^*\|^2 \, dt }   \, .
	\end{aligned}
	\label{e:norms}
\end{equation}
The $L_1$ norm is also referred to as the absolute average error (AAD). These quantities are related via 
\begin{equation}
	\| \ODEstate -\theta^*\|_{L_1}   \le  \| \ODEstate -\theta^*\|_{L_2}  = \sqrt{   \sigma^2_\theta  +   \bias_\ODEstate^2  } \, .
	\label{e:cov_from_bias}
\end{equation}
We say that the  AAD is  $O(\eta_\alpha)$ if there is a finite constant $B_0$ such that
 $\| \ODEstate -\theta^*\|_{L_1} \le B_0 \eta_\alpha$   for all $\alpha>0$ in a neighborhood of the origin; $\eta_\alpha$ will be polynomial in $\alpha$.     

We find that the AAD of QSA may be order $O(\alpha)$ even when $\barUpupsilon\equiv 0$.    \Cref{fig:QSArep} illustrates how linear filtering techniques can reduce this AAD to $O(\alpha^2)$.  The concept is simple:  if $\ddt u_t$ is passed through a low pass filter with bandwidth $\gamma>0$,   then the output $y_t$  satisfies $|y_t|   \le \gamma \|\bfmu\|_\infty + o(1)$.
In general we require a second order filter to attenuate the term   $\frac{d^2}{dt^2} \clW_t^2$
 appearing in \eqref{e:BigGlobalODE}.

 We are also interested in the  \textit{target bias},  defined as the sample average
\begin{equation}
	b_\barftwo \eqdef       \lim_{T\to\infty} \frac{1}{T} \int_0^T   \barf (\ODEstate_t)   \, dt
	\label{e:meanTargetBias}
\end{equation}

	\begin{figure}
		\centering
		\includegraphics[width= .5\hsize]{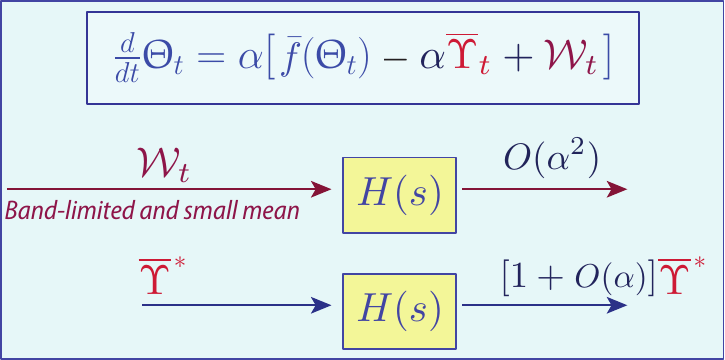}  
		\caption{Perturbative mean flow and its implication to filter design. }
		\label{fig:QSArep}
	\end{figure}

\subsection{Stability theory for QSA}
\label{s:QSAstable}

We opt for the usual control-theoretic definition for the full state process:
We say that $\bfPsi  = (\bfODEstate,\bfPhi)$ is \textit{ultimately bounded} if there is a fixed constant $B<\infty$ satisfying the following:   for each initial condition $\ODEstate_0 =\theta$,  $\Phi_0 = z$,  there is a finite time $t_0 = t_0(\theta,z)$ such that
\begin{equation}
	\|  \ODEstate_t\|   \le B\,, \qquad t\ge t_0
	\label{e:UB}
\end{equation}
Two criteria to establish this property are available based entirely on the mean vector field $\barf$:    
\begin{romannum}
\item[1.]   \textbf{Lipschitz Lyapunov Function:}      For a $C^1$ function $V\colon\Re^d\to\Re_+$ and a constant $\delta_0>0$,
\begin{equation} 
	\frac{d}{dt} V(\odestate_t)  \leq  -  \delta_0    V(\odestate_t)    \,, \qquad   \textit{when $\|\odestate_t\|\ge \delta_0^{-1}$}
	\label{e:ddt_bound_LyapfunTmp}
\end{equation}
This is equivalently expressed $\nabla  V(\theta) \cdot \barf(\theta)  \le  -  \delta_0    V(\theta)$ for all  $\|\theta\|\ge \delta_0^{-1}$.   This implies a similar bound for the QSA ODE under very general conditions.   Most crucial is that $f$,  $\barf$,   and $V$ each satisfy a global Lipschitz bound.   

\item[2.]   \textbf{ODE@$\infty$:}     This is the vector field obtained by scaling 
\begin{equation}
	\barf_\infty (\theta) 
	\eqdef
	\lim_{r \to \infty }   \barf^r (\theta)
	\quad 
	\textit{with}
	\quad \barf^r (\theta)= r^{-1} \barf ( r \theta)  \,,\qquad \theta\in\Re^d 
	\label{e:barfinfty_QSA}
\end{equation} 
If this exists, and if the ODE with this vector field is locally asymptotically stable, then we establish a relaxation of	\eqref{e:ddt_bound_LyapfunTmp} that also implies that $\bfPsi$ is ultimately bounded. 
\end{romannum}

The first criterion might look odd to those accustomed to quadratic Lyapunov functions.     Suppose that $V_1$ is quadratic, of the form $V_1 (\theta) = \theta^\transpose P \theta$ with $P>0$  $(d\times d)$, and solves 
$\nabla  V_1(\theta) \cdot \barf(\theta)  \le  -  \delta_1    V_1(\theta)$ for some $\delta_1>0$ and all $\|\theta\| \ge \delta_1^{-1}$.   It follows from the chain rule that 	\eqref{e:ddt_bound_LyapfunTmp} holds using $V = \sqrt{1+ V_1}$, and $\delta_0 \in (\delta_1/2,1)$.

The second approach is based on the very notion of ultimate boundedness:  if this condition fails to hold, then by definition the state blows up.   Hence we look at the evolution on a compressed spatial scale.   
An example is shown in \Cref{fig:ODEinfty},  with details postponed to \Cref{s:ESC}.

 \begin{wrapfigure}[16]{r}{0.32\textwidth}

 	\centering
	\vspace{-.75em}
	
	\includegraphics[width=0.95\hsize]{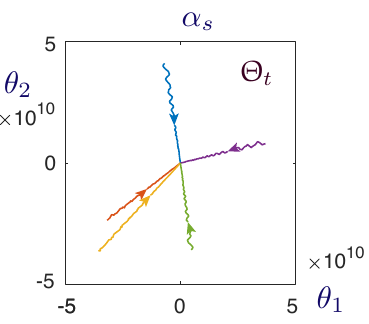}
	\caption{Solutions to \textit{extremum seeking control} for the Rastrigin objective, starting from very large initial conditions}%
	\label{fig:ODEinfty}
\end{wrapfigure} 

\textit{The ODE@$\infty$ is also a useful technique to establish ultimate boundedness of the mean flow}:   consider a large initial condition,  with magnitude $r = \| \odestate_0\|$.    The scaled state process $\odestate^r_t  = r^{-1} \odestate_t$ solves the scaled ODE,
\begin{equation}
\ddt \odestate^r_t    =  r^{-1} \barf(\odestate_t)     = \barf^r(\odestate_t^r)   \,,\qquad   \|  \odestate^r_0 \|=1
\label{e:ODEinf-r}
\end{equation}
This simple conclusion leads to a proof of ultimate boundedness for the mean flow, following two observations:

\wham{(i)}
 The convergence in	\eqref{e:barfinfty_QSA}
is uniform on compact subsets of $\Re^d$  (an application of Lipschitz continuity).

\wham{(ii)}
 if $\barf_\infty$ is locally asymptotically stable, then it must be globally exponentially asymptotically stable, with the origin the unique stationary point, $ \barf_\infty(\Zero) =\Zero$.

 Based on these conclusions we obtain semi-exponential stability, in the following sense:  there is $b_e<\infty$ and $\delta_e>0$ such that for any initial conditions satisfying $\| \odestate_0\| \ge \delta_e^{-1}$,  solutions to 	\eqref{e:barf_ODE_gen}  satisfy
 \begin{equation}
\| \odestate_t  \|  \le b_e \|\odestate_0\| e^{-\delta_e t} \,,\qquad 0\le t\le  T(\odestate_0) 
\label{e:ExpStableAlmost}
\end{equation}
where $T$ is the first entrance time: $
T(\odestate_0) \eqdef  \min \{  t\ge 0 : \| \odestate_t\|  \le \delta_e^{-1}  \} $.

The bound \eqref{e:ExpStableAlmost} is extended to the QSA ODE in  \Cref{t:ODEstate_is_bdd},  
under either of these two stability criteria.

\subsection{Bridge building}      
\label{s:ESCetc}

As motivation, to provide examples, 
and to provide a short literature survey,
we present a brief survey of gradient free optimization (GFO) and how it is related to QSA and
extremum seeking control (ESC).   
The definition of ESC is postponed to \Cref{s:ESC}, based on interpretation of its standard block diagram description shown in \Cref{fig:ESC}.

	\begin{figure}[h]
		\begin{center}
			\includegraphics[width=0.8\hsize]{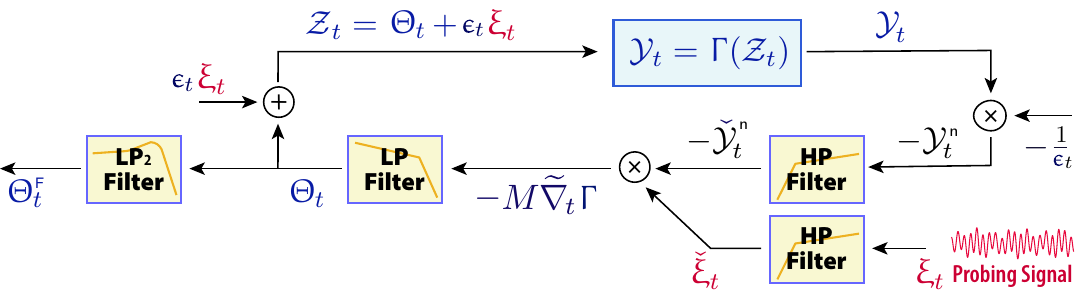}
		\end{center}
		\caption{A Typical Architecture for Extremum Seeking Control for Gradient-Free Optimization}
		\label{fig:ESC}
	\end{figure}

\subsubsection{SPSA}

\begin{subequations}

Within the area of GFO two algorithms of Spall are most easily described and justified.   Each are SA algorithms of the form \eqref{e:SA_recur}, differing only in the definition of $f$:
\begin{align}
\text{\sf 1SPSA:} \quad 
f(\theta,\qsaprobe) &
=
		- \frac{1}{\epsy} 
		\qsaprobe \Obj(\theta + \epsy \qsaprobe)
\label{e:1qSPSA}
		 \\
\text{\sf 2SPSA:} \quad 
f(\theta,\qsaprobe) &=
		-
		\frac{1}{2\epsy}   \qsaprobe\bigl[\Obj(\theta + \epsy \qsaprobe) -\Obj(\theta - \epsy \qsaprobe) \bigr]   
\label{e:2qSPSA}
\end{align}
It is assumed that the probing sequence $\{\qsaprobe_n\}$  is i.i.d., and in much of the theory it is assumed that the entries of $\qsaprobe_n$ have support in $\{-1,1\}$.   
In either case, the mean vector field is defined via \eqref{e:barf}:
\[
 \barf(\theta) \eqdef   \Expect[  f(\theta,\qsaprobe_{n+1})  ] \,,  \qquad \theta\in\Re^d  
\]
This is independent of $n$ under the assumption that the probing sequence is i.i.d..   

\label{e:SPSA_Gen}
\end{subequations}

\begin{subequations}

Motivation for either is usually based on a first-order Taylor series approximation.  The following exact representations follow from   the Fundamental Theorem of Calculus:
\begin{align}
\text{\sf 1SPSA:} \quad 
f(\theta,\qsaprobe) &=
			 		-  \qsaprobe\qsaprobe^\transpose  
						 \frac{1}{\epsy}  \int_{0}^\epsy   \nabla \Obj\,( \theta + t \qsaprobe)  \, dt   
						  -  \frac{1}{\epsy}  \qsaprobe \Obj( \theta)  
\label{e:1qSPSA_FTC}
		 \\
\text{\sf 2SPSA:} \quad 
f(\theta,\qsaprobe) &
		=  -   \qsaprobe\qsaprobe^\transpose  \frac{1}{2\epsy}  \int_{-\epsy}^\epsy   \nabla \Obj\,( \theta + t \qsaprobe)  \, dt  
\label{e:2qSPSA_FTC}
\end{align}
These representations in terms of average gradients provide the first step in the proof of the following:

\label{e:SPSA_Gen_FTC}
\end{subequations}


  \begin{proposition}
\label[proposition]{t:SPSAbias}
Suppose that the $d$-dimensional sequence $\bfqsaprobe$ is i.i.d.\ with zero mean,  and   that $ \qsaprobe_{n+1} \eqdist - \qsaprobe_{n+1}$  (equality in distribution).   Then the mean vector fields for 1SPSA and 2SPSA coincide.

Suppose in addition that the objective $\Obj$ is continuously differentiable ($C^1$)  and strongly convex in a neighborhood of the optimizer $\theta^\opt$.
Then,   there exists $\epsy^0>0$ such that the function $\barf$ has a  root $\theta^\ocp$ for $0<\epsy\le \epsy^0$, and for either 1SPSA or 2SPSA,  
\begin{equation}
\| \theta^\ocp - \theta^\opt \| = O(\epsy^2)  \qquad \textit{and}\qquad 
\Obj( \theta^\ocp ) -  \Obj(\theta^\opt )  = O(\epsy^4)  
\label{e:SPSAbias}
\end{equation}
\qed
\end{proposition}

SPSA is a perfect example of a successful algorithm that can be improved using deterministic exploration.

\subsubsection{ESC and QSA}   
\label{s:ESC_QSA}

Spall's SPSA recursions admit natural analogs as QSA ODEs, with vector fields
\begin{subequations} 
\begin{align}
\text{\sf 1qSGD:} \quad 
f(\ODEstate_t,\qsaprobe_t)  &=
- \frac{1}{\epsy} 
		\qsaprobe_t \Obj(\ODEstate_t + \epsy \qsaprobe_t)
\label{e:1qSGD}
		 \\
\text{\sf 2qSGD:} \quad 
f(\ODEstate_t,\qsaprobe_t) &=
-
		\frac{1}{2\epsy}   \qsaprobe_t\bigl[\Obj(\ODEstate_t + \epsy \qsaprobe_t) -\Obj(\ODEstate_t - \epsy \qsaprobe_t) \bigr]   
\label{e:2qSQD}
\end{align}
in which {\sf qSGD} refers to \textit{quasi-stochastic gradient descent}.   In examples that follow we typically take the sinusoidal probing signal \eqref{e:qSGD_probe}.   

\label{s:qSGD}
\end{subequations} 

In \Cref{s:ESC} we show that 1qSGD is a simple instance of extremum seeking control.

And we also find a challenge:  in both domains,  ESC and 1qSGD,  convergence theory is limited because the vector field $f$ is not Lipschitz continuous in its first variable.   \textit{Lipschitz continuity is a  critical component of theory for both SA and QSA}.   
A remedy is introduced in this paper, through the introduction of a state-depending probing gain.   Details are found in    \Cref{s:SDprobe}.

\subsubsection{Averaging}

The theory of averaging for variance reduction is well-established in the SA literature,  with application to optimization  \cite{frapol20,bacmou11,moujunwaibarjor20} and to  GFO  \cite{spa03}.   The idea is very simple:  given the estimates $\{\theta_n : 0\le n  \le N\}$ from any SA recursion,  obtain a new estimate via  
\begin{equation}
\thetaPR_N \eqdef \frac{1}{N-N_0}  \sum_{n=N_0+1}^N   \theta_n
\label{e:PR}
\end{equation}
where $1\le N_0 <N$ is chosen to discard large transients.  
This apparently simple ``hack'' provides enormous benefit:  under mild conditions, the variance of this estimate of $\theta^\ocp$ is minimal in a strong sense  \cite{bor20a,kusyin97}.   The technique goes by the name of Polyak-Ruppert averaging in honor of the creators of the technique and original analysis  \cite{rup88,pol90,poljud92} (note Polyak's  contributions in \cite{pol90} prior to his collaboration with Juditsky).
The article \cite{dipren97} is an early application of these ideas to accelerate GFO.

This is similar to the low pass filter shown in \eqref{fig:ESC}, generating  $\ODEstateF_t$ from $\ODEstate_t$.
We will find better options for fixed gain QSA   in \Cref{s:filterDesign}.

 \subsubsection{Multiplicative noise}
 
One objective of this paper is to bring to light the  challenges posed by multiplicative noise.     Many optimistic results in prior work consider additive noise models for which the nuisance term $\Upupsilon$ is not present:
\begin{romannum}
\item
Estimation error bounds of order $O(\alpha^2)$  are obtained in \cite{arikrs02} for  ESC when the objective is quadratic.    
 
\item
There are recent optimistic results in the stochastic approximation literature for  fixed gain algorithms \cite{moujunwaibarjor20,durmounausamscawai21}.   Through design it is possible to obtain zero asymptotic bias and optimize variance (via averaging).
These results are obtained for models with additive white noise.   
 \end{romannum}
 For QSA we obtain  both $O(\alpha^2)$ bounds for AAD and target bias through careful design of filters and the probing signal.  
 Theory shows that we  cannot expect    $O(\alpha^2)$ AAD for general models, unless a filter is  designed following the specifications in \Cref{t:Couple_theta}.
 
In stochastic settings with Markovian noise it is not at all clear how to obtain such strong conclusions---averaging cannot remove the inherent $O(\alpha)$ AAD that is a consequence of statistical memory.  Details are found in  \Cref{s:SAbias}.

  \medskip

\wham{Overview}
The remainder of the body of the paper is organized into five additional sections:     
 The main technical results of the paper are contained in \Cref{s:QSAtheory}, including a proof of the P-mean flow representation \eqref{e:BigGlobalODE}.    
 Implications to extremum seeking control are summarized in \Cref{s:ESC}, 
with examples in \Cref{s:Experiments},  
  and conclusions and directions for future research are summarized in
\Cref{s:conc}. 
A literature survey is contained in \Cref{s:lit}, and proofs of some technical results are contained in the Appendix.

\section{QSA Theory}
\label{s:QSAtheory}

The theory described in this section is based on the fixed-gain QSA ODE with $a_t\equiv \alpha>0$:    
\begin{equation}
	\ddt\ODEstate_t = \alpha f(\ODEstate_t,\qsaprobe_t)
\label{e:QSAgen_alpha}
\end{equation}
Convergence of $\{\ODEstate_t\}$ to $\theta^*$  cannot be expected;  instead with careful design of the algorithm we obtain bounds on asymptotic bias of order $O(\alpha^2)$ and variance of order $O(\alpha^4)$.

\subsection{Markovian framework for analysis}
\label{s:markov}

	In the theory of SA the  stochastic sequence $\{\qsaprobe_{n+1}\}$ appearing in \eqref{e:SA_recur} is not always assumed to be i.i.d..
Sharp results on asymptotic variance are available when it can be expressed as a function of a well behaved Markov chain.  
It may come as a big surprise to learn that the techniques extend to QSA analysis, and the application of techniques rooted in Markov chain theory are the only tools available to obtain \eqref{e:BigGlobalODE} and its progeny. 

		The pair $\bfPsi =  (\bfODEstate,\bfPhi)$ is the state process for a time-homogeneous dynamical system, 
	which may be expressed as a continuous function of its initial condition.    
	There is a mapping   $\flow \colon \Re\times \prstate\to\bivstate$ such that
	\begin{equation}
		\Psi_{t+s} = \flow_s(\ODEstate_t,\Phi_t)\,,  \qquad   s,t\in\Re_+\,, \  \Psi_t = (\ODEstate_t,\Phi_t)\in \bivstate =  \Re^d \times \prstate
		\label{e:flow}
	\end{equation}
	The process $\bfPsi$ can also be regarded as a Feller Markov process on   $\bivstate$. 
	Provided  $\bfPsi $  is ultimately bounded,   it admits at least one invariant measure \cite[Thm.~12.1.2]{MT}, which we denote $\upvarpi$.
	We  present general conditions under which $\upvarpi$ is unique, and in this case denote $\barg \eqdef \int g(q) \upvarpi(q)$ for any continuous function $g:\bivstate \to \Re$.

The following simple result will aid in interpretation of other expectations.   
For any function $h\colon\Uppi\to\Re$ that is continuously differentiable, define the continuous function $\clD h$ via
\begin{equation}
 \clD h\, (\theta,z)  \eqdef  \partial_\theta h\,  (\theta,z )   \cdot f (\theta, G(z))    
	+   \partial_z   h \, (\theta, z) \cdot W  z    \,,\qquad (\theta, z) \in\Uppi  
\label{e:diffGen}
\end{equation}
In Markov terminology, the  functional $\clD$ is known as the \textit{differential generator} for $\bfPsi$.

\begin{proposition}
\label[proposition]{t:QSAgenerator}
Let $h\colon\Uppi\to\Re$ be a $C^1$ function, and denote $g=\clD h$.  The following then hold: 
\begin{romannum}
\item   
$ g(\Psi_t)  =\ddt h(\Psi_t) $.

\item
Suppose that for some initial condition $\Psi_0$, the resulting trajectory $\{ \Psi_t : t\ge 0\}$ is uniformly bounded.   Then,
\[
\lim_{T\to\infty}   \frac{1}{T}   \int_0^T  g(\Psi_t)  \, dt   = \Zero
\]

\item
Suppose that an invariant measure  $\upvarpi$ exists with compact support.
Then $\Expect[g(\Psi)]=0$  when $(\ODEstate,\Phi) = \Psi\sim \upvarpi$.   
In particular,  on choosing $h(\theta,z) \equiv \theta$,
\begin{equation}
 			\Expect[f(\ODEstate, G(\Phi))]= \Zero
 \label{e:fMeanZero}
\end{equation}
\qed
\end{romannum}
\end{proposition}

\begin{proof}
Part (i) follows from the chain rule,  and (ii) from the fundamental theorem of calculus:
\[
  h(\Psi_T)  =  h(\Psi_0)     +   \int_0^T  g(\Psi_t)  \, dt    \,,\qquad T>0
\]
Part (iii) also follows from the fundamental theorem of calculus, using the expression above,  but with   a random initial condition.
Choose  $\Psi_0 =  (\ODEstate_0, \Phi_0)$ to be a $\Uppi$-valued random variable with distribution $ \upvarpi$.   We then have $\Psi_t\sim \upvarpi $ for all $t$.   Taking expectations of each side  completes the proof:
 \[
 \Expect[  h(\Psi_0)  ]   =  \Expect[  h(\Psi_0)  ]     + T  \Expect[  g(\Psi_0) ]
\] 
\end{proof}

Our main motivation for a Markovian framework comes from the application of solutions to \textit{Poisson's equation}:  
	for functions   $g,\, \hag\colon\prstate\to\Re$, this is expressed
\begin{equation}
		\hag(\Phi_0) = 
		\hag(\Phi_T) +   \int_0^T   \tilg(\Phi_t)\, dt\,,\qquad  T\ge 0\,, 
		\label{e:PoissonDefn}
\end{equation}
	where $\tilg(z) = g(z) -\barg$ for $z\in\prstate$,  with $\barg$ the steady-state mean.  We say that $\hag$ is the \textit{solution} to Poisson's equation, and that $g$ is the \textit{forcing function}.    If $\hag$ is continuously differentiable, then from \eqref{e:LinearMarkovQSA} we obtain $\nabla^\transpose \hag(z) W z = - \tilg(z)$ for $z\in\prstate$.
	
	Justification for smooth solutions to Poisson's equation 
	is
	 obtained in \cite{laumey22,laumey22e} when  the forcing function   is analytic,  and subject to assumptions on the frequencies defining $\bfPhi$. 	
	 A review is contained in the Appendix, and a summary of required assumptions in Assumption (A0) below.
	\begin{romannum}
		\item[]\textit{Please note:}
		\item
		The solution $\hag$ is not unique.   We always normalize so that $\Expect[\hag(\Phi)] =0$.
		
\item
In the applications of Poisson's equation used in this paper, it is frequently true that the function $g$ depends on $\Phi_t$ only through $\qsaprobe_t$.
\textit{This is this is not in general true for $\hag$}. 
		
		\item 
		Finally, on notation:  we often  write $\hag_t$ instead of  $ \hag(\ODEstate_t,\Phi_t)$  (similar notation for other functions of $\bfPsi$).
		
	\end{romannum}

	We will not require solutions for the joint process $\bfPsi$, but   require a slight abuse of notation:  for a vector-valued function on the joint state space  $g\colon\Uppi\to \Re$,  we denote by $\hag (\theta,\varble)$   the solution to Poisson's equation with $\theta\in\Re^d$ fixed:
	\begin{equation}
		\hag(\theta,  \Phi_0) = 
		\hag(\theta, \Phi_T) +   \int_0^T   \tilg(\theta,  \Phi_t)\, dt\,,\qquad  T\ge 0\, .
		\label{e:PoissonDefnAbuse}
	\end{equation}
	This is applied with $g=f$  [the QSA ODE vector field \eqref{e:QSA_ODE_gen}];
	the solution $\haf$ plays a part in the    definition of $\barUpupsilon$.

Further implications require a few assumptions and notation:
	
\wham{Assumptions}

The following assumptions are in force throughout much of the paper.

The first assumption sets restrictions on frequencies.

\begin{subequations}

\wham{(A0a)}   
$\qsaprobe_t  = G_0(\qsaprobe _t^0  )$ for all $t$,   with $\qsaprobe _t^0$ defined in \eqref{e:qsaprobe0}.  The function 
$G_0\colon\Re^K \to\Re^m$ is assumed to be analytic, with the coefficients in the Taylor series expansion for $G_0(\qsaprobe _t^0  )$ absolutely summable.     

\wham{(A0b)}   
The frequencies $\{\omega_1\,,\dots\,,  \omega_K\}$ are chosen of the form
\begin{equation}
	\begin{aligned}
		&\omega_i  = \log(a_i/b_i) > 0\,, \ \  1\le i\le K\,, \ \ \textit{where   $\{a_i,b_i\}$ are positive integers.   }
		\\
		&  \textit{$\{\omega_i : 1\le i\le K\}$ are linearly independent over the field of rational numbers.}
	\end{aligned} 
	\label{e:logFreq}
\end{equation}

\wham{(A1)}    
The functions $\barf$ and $f$ are
	 Lipschitz continuous:  for   a constant $\Lip_f  <\infty$,
\begin{equation}
\begin{aligned} 
		\|\barf(\theta') - \barf(\theta)\| &\le \Lip_f \|\theta' - \theta\|, 
		\\
		\|f(\theta',\qsaprobe) - f(\theta,\qsaprobe)\| 
		+
		\|f(\theta,\qsaprobe') - f(\theta,\qsaprobe)\| 
		&\le \Lip_f \bigl[ \|\theta' - \theta\|  +   \|\qsaprobe' - \qsaprobe\|  \bigr]
\end{aligned} 
\label{e:fLip}
\end{equation}
for all $ \theta' \,, \theta\in\Re^d$,
$\qsaprobe\,, \qsaprobe' \in\Re^m $.

\wham{(A2)}   
The vector fields $f$ and $\barf$   are each twice continuously  differentiable,  with derivatives denoted 
\begin{align}
A(\theta, z) & = \partial_\theta f \, (\theta, z)  \,,
\qquad
 \barA(\theta) = \partial_\theta \barf\, (\theta)
		\label{e:barfDer}
\end{align}

\wham{(A3)} 
There exists a  solution $\haf$ to Poisson's equation that is continuously differentiable on $\Uppi$,  and
normalized so that $ \Expect[ \haf(\theta,  \Phi) ] = \Zero$ for each $\theta$,  with $ \Phi\sim \uppi$.   Its Jacobian with respect to $\theta$ is denoted
\begin{equation}
\haA(\theta ,z)  \eqdef  \partial_\theta \haf\, (\theta, z)
\label{e:der_haA}
\end{equation}
Moreover,  there are continuously differentiable solutions to Poisson's equation for each of the two forcing functions $\haf$ and $\Upupsilon$,  with
	\begin{equation}
		\Upupsilon(\theta,z) =   - \haA   (\theta,z )    f (\theta, G(z))    
\label{e:Upupsilon}
\end{equation}
The solutions are denoted, respectively,    $\hahaf$  and $\haUpupsilon$:
\begin{align}
\hahaf(\theta,  \Phi_{t_0} )  &= 
			\int_{t_0} ^{t_1}  \haf( \theta ,\Phi_t)    \, dt    + \hahaf( \theta,   \Phi_{t_1} )  
\nonumber
   \\
\haUpupsilon(\theta,  \Phi_{t_0} )  &= 
			\int_{t_0} ^{t_1}  [ \Upupsilon( \theta ,\Phi_t)-  \barUpupsilon(\theta ) ]     \, dt    + \haUpupsilon( \theta,   \Phi_{t_1} )  			 \,,   && 0\le t_0\le t_1
			\label{e:haEverythingElse}
\\
\textit{with}\quad 
\barUpupsilon(\theta ) & =     \Expect[   \Upupsilon(\theta ,\Phi )  ]
				   = -
				\int_\prstate   \haA(\theta ,z )      f(\theta,G(z))   \,  \uppi(dz) \,,  && \theta\in\Re^d 
\nonumber
\end{align}
They are also normalized: $    \Expect[ \hahaf(\theta,  \Phi) ] =    \Expect[ \haUpupsilon(\theta,  \Phi) ] =   \Zero$ for each $\theta$.

\wham{(A4)}
The ODE $\ddt \odestate_t  = \barf  ( \odestate_t  )$ is globally asymptotically stable with  $\theta^\ocp$ the unique equilibrium. Moreover, the matrix   $A^\ocp = \barA (\theta^\ocp)$ is Hurwitz.

	\end{subequations}	

\smallbreak

Assumption~(A0b) is justified in \Cref{s:BakerPoisson}, but deserves some explanation here.
It is imposed for two reasons:  
\begin{romannum} 
\item   It is a sufficient condition ensuring that $\barUpupsilon\equiv \Zero$,   assuming that Assumption~(A3) holds so that  $\Upupsilon$ is well defined.    See the example following \Cref{t:Couple_theta} to see that constraints on the frequencies are indeed needed for this. 

\item  If   $f(\theta, G(z))$ is an analytic function of $(\theta,z)$   on an appropriate domain, then  Assumption~(A3) holds.
\end{romannum}
Claim (i) is an application of \Cref{t:varpi}~(iii),
while claim~(ii) follows from \Cref{t:hah}.

\smallskip

As for the remaining assumptions, we require  (A1)--(A4) to obtain sharp bounds on bias and variance. 
In particular, the functions introduced in (A3) determine all of the terms in  \eqref{e:BigGlobalODE}, with  $\barUpupsilon$ is defined in \eqref{e:haEverythingElse}.

The assumptions on $G_0$ in Assumption~(A0a) are imposed so that the probing signal is almost periodic \cite{amepro13,boh18}.  
This combined with the Lipschitz conditions in  (A1) imply a uniform version of the Law of Large Numbers:
\begin{proposition}
\label[proposition]{t:uniLLN}
Under (A0a) and (A1) the uniform Law of Large Numbers holds:
\begin{equation}
\lim_{T\to\infty }   \sup_{\theta, z_0} \frac{1}{ 1 +  \| \theta\|  }  \left |
\frac{1}{T }\int_0^T  \bigl[f(\theta,\qsaprobe_t) - \barf(\theta) \bigr]  \, dt
				\right|   =0
\label{e:LLNf}
\end{equation}
where the supremum is over  $\theta_0\in\Re^d$,
			  and   initial conditions $\Phi_0 = z_0\in\Upomega$.
\qed
\end{proposition}

	\subsection{Steady  state realizations and the LLN}
\label{s:coupling}

Poisson's equation   \eqref{e:PoissonDefn} is one of several concepts from the theory of Markov processes adapted to the deterministic setting of this paper.   Two more described here are used to obtain the LLN for the joint process  $\bfPsi$, criteria for a unique invariant measure $\upvarpi$,   and sample path coupling of trajectories from distinct initial conditions.

\wham{Equicontinuity}

We say that $\bfPsi$ is an \textit{e-process}   if  for each continuous function $g\colon \bivstate   \to\Re$ with compact support,    the family 
of functions $ \{ g ( \flow_t (\varble , \varble ) ):  t\ge 0 \}$ is equicontinuous on compact subsets of $\bivstate$ \cite[Ch.~6]{MT}.

Under ultimate boundedness in the form \eqref{e:UB},  denote $	\sstate_0 =\{  (\theta,z) :  \|\theta\| \le B_0 \}$.

\begin{theorem}
	\label[theorem]{t:e-chain}
	Suppose that (A1)-(A3) hold, and in addition $\bfPsi$ is an  e-process that is ultimately bounded, so that \eqref{e:UB} holds.   Then, there is a Feller transition kernel $\Inv$ satisfying $\Inv(q_0, \sstate_0) = 1$ for each $q_0\in\bivstate$, and the following hold for each continuous function $g\colon\bivstate\to\Re$:     
	\whamrm{(i)}  For any   compact set $K\subset \bivstate$,
	\[
	\lim_{T\to\infty }   \sup_{\Psi_0\in K}  \left |
	\frac{1}{T }\int_0^T  \bigl[   g(\Psi_t) -   \Inv g\, (\Psi_0) \bigr]  \, dt
	\right|   =0
	\]
	where   $ \Inv g\, (q_0) \eqdef \int \Inv (q_0 , dq)  g(q)$ for $q_0\in \bivstate$ and $g:\bivstate   \to\Re$.
	
	\whamrm{(ii)}  $\Inv g\, (\Psi_t) =  \Inv g\, (\Psi_0) $  for each $ t\ge 0  $  and  initial condition $\Psi_0 $.
	
	\whamrm{(iii)}    The probability measure $\Inv(q_0,\varble) $ is invariant for $\bfPsi$, for each   $q_0\in \bivstate$. 
	\qed 
\end{theorem}

\begin{figure}[h]
	\centering
	\includegraphics[width=0.8\hsize]{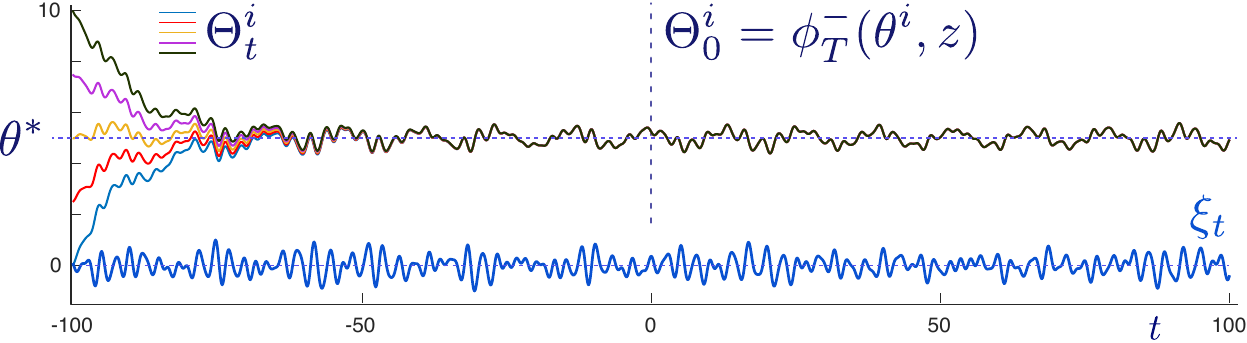}
	
	\caption{Topological coupling from the past. }
	\label{fig:CFP}
\end{figure}

\wham{Coupling from the past}
This is  a concept introduced in \cite{prowil98} to achieve so-called \textit{perfect sampling} from a given distribution.      
Translating to the setting of this paper,  
the construction involves consideration of the QSA ODE initialized
in ``the past'',  at a time $-T$.  The   process $\{\Phi_t : -\infty <t <\infty \}$ is held fixed as we let $T\uparrow \infty$.

\Cref{fig:CFP} illustrates what we can expect,  based on a  scalar linear QSA ODE,    $\ddt \ODEstate_t =  - \alpha [\ODEstate_t - \theta^*  + \qsaprobe_t]$,   with $\alpha>0$ and $\bfqsaprobe $   a linear combination of   sinusoids.    The plot shows initialization at $-T = -100$,  with $\{\theta^i\}$ equally spaced between $0$ and $10$.    The five trajectories   may be expressed,  
\begin{equation}
	\ODEstate_t^i   =  \flow_{t+T}(\theta^i, \Phi_{-T} )  \,,\qquad    t\ge -T \, ,  \quad i=1,\dots, 5
	\label{e:CFPfive}
\end{equation}
In the Appendix,  this complex notation is abbreviated to $ \flow^-_T (\theta^i, z)$,   with $z = \Phi_0$.   

Uniqueness of the invariant measure follows from a proof that  $  \| \ODEstate^i_t  -   \ODEstate^j_t\|  \to 0$ as $T\to\infty $, 
for each    $t\ge 0$ and initial conditions  $\theta^i \neq \theta^j$.  This is obvious in this simple example, but is shown to hold in broad generality.    
Hence coupling is established in a topological sense, rather than the almost sure coupling obtained in \cite{prowil98} for Markov chains.

Coupling implies that there is a stationary version of the pair process of the form
\begin{equation}
	\Psi_t^\infty =  ( \ODEstate^\infty_t ,   \Phi_t^\infty )=  ( \flow_\infty^-(\Phi_t^\infty) ,   \Phi_t^\infty )
	\,, \qquad -\infty < t<\infty
	\label{e:PsiState}
\end{equation}
in which $\flow_\infty^- \colon\prstate \to\Re^d$ is continuous.

\subsection{Three steps towards the perturbative mean flow }
 
 The notation in \eqref{e:Upupsilon}  is complex, which brings us to the following compact alternative:
	For $h\colon \Uppi \to\Re^d$ that is continuously differentiable in its first variable we let $\Df h$ denote its directional derivative in the direction $f$,  where $f$ is the vector field for the QSA ODE:	
\begin{equation}
[ \Df h ] (\theta,z )  = \partial_\theta h\,  (\theta,z )   \cdot f (\theta, G(z))  \,, \qquad (\theta,z )  \in \Uppi
		\label{e:Dfg}
\end{equation} 
This defines a component of the differential generator \eqref{e:diffGen};
 recall from  \Cref{t:QSAgenerator}  that  $g=\clD h$ implies that  $ g(\Psi_t)  =\ddt h(\Psi_t) $.

The following companion to  \Cref{t:QSAgenerator} will be used in the following:
	for any smooth solution to Poisson's equation with forcing function $h\colon\Uppi\to\Re$,
	\begin{align}
		\ddt \hah(\ODEstate_t,z )  &=  \alpha [\Df \hah ] (\ODEstate_t,z )  \,,   \ z\in\prstate  
		\label{e:chainRule_f}
		\\
		\ddt \hah(\ODEstate_t, \Phi_t )  &=  \alpha [\Df \hah ] (\ODEstate_t, \Phi_t )     
		-  [ h(\ODEstate_t, \Phi_t )  -  \barh(\ODEstate_t) ]   
		\label{e:chainRule_f2}
	\end{align}	
Using $h=\haf$,  the definition \eqref{e:Upupsilon} gives
\begin{equation}
\Upupsilon_t  =          - [\Df \haf ] ( \ODEstate_t, \Phi_t)     =  - \frac{1}{\alpha}  \{ \ddt \haf_t +  [ f(\ODEstate_t, \qsaprobe_t )  -  \barf(\ODEstate_t) ]\}    
\label{e:UpupsilonDer}
\end{equation}

We now proceed through the three steps, with this representation in view:
	\begin{equation}
		\ddt\ODEstate_t = \alpha \bigl[  \barf (\ODEstate_t )     +  \tilXi_t  ] \,,
		\label{e:QSA_ODE_fixed}
	\end{equation} 
	where $  \tilXi_t   =  f(\ODEstate_t,\qsaprobe_t) -  \barf(\ODEstate_t) $ is called the  \textit{apparent noise}.  Understanding \eqref{e:BigGlobalODE} is equivalent to determining the functions $\{ \clW^i \}$ in the representation
	\begin{equation}
		\begin{aligned} 
			\tilXi_t  &  =
			-  \alpha \barUpupsilon_t 
				+    \alpha^2 \clW_t^0 + \alpha \ddt \clW_t^1  +   \tfrac{d^2}{dt^2} \clW_t^2  
		\end{aligned} 
		\label{e:EffNoiseDecomp}
	\end{equation}
	
	\wham{Step 1:}     Apply \eqref{e:chainRule_f2}  with $g=\haf$:
	\[
	\ddt  \haf(\ODEstate_t,  \Phi_t)   =   \partial_\theta \haf(\ODEstate_t,  \Phi_t)  \ddt \ODEstate_t    - 
	[ f(\ODEstate_t,  \qsaprobe_t)  -\barf(\ODEstate_t)] 
	\]
	This gives the first transformation of the apparent noise:
	\[
	\tilXi_t  =    \underbrace{-\ddt  \haf(\ODEstate_t,  \Phi_t)  }_{\text{High pass}}
	+    
	\underbrace{\alpha   \partial_\theta \haf(\ODEstate_t,  \Phi_t) f(\ODEstate_t,  \qsaprobe_t)   \phantom{\ddt }}_{\text{Attenuation}}
	\]
Recalling \eqref{e:Upupsilon}  gives  in shorthand notation,
\begin{equation}
		\tilXi_t  =   -\ddt  \haf_t   -\alpha \Upupsilon_t
		\label{e:EndOfStep1}
\end{equation}

	\wham{Step 2:}    Repeat previous argument with $\haf$:
\begin{equation}
	\begin{aligned}
		\ddt \hahaf  (\ODEstate_t,\Phi_t)    &=    \alpha [\Df \hahaf ] (\ODEstate_t, \Phi_t )      -  \haf_t     
		\\
		&\Rightarrow  \haf_t   =   \alpha [\Df \hahaf ] (\ODEstate_t, \Phi_t )   
		-  \ddt \hahaf_t  				
		\\					
		&\Rightarrow 
		\ddt \haf_t   =   \alpha \ddt  [\Df \hahaf ] (\ODEstate_t, \Phi_t ) 
		-  \tfrac{d^2}{dt^2} \hahaf_t  
	\end{aligned} 
	\label{e:step2}
\end{equation}
	
\wham{Step 3:}    Repeat   with $\Upupsilon$:
\begin{equation}
	\begin{aligned}
		\ddt \haUpupsilon (\ODEstate_t,\Phi_t)    &=    \alpha [\Df \haUpupsilon ] (\ODEstate_t, \Phi_t )      - [  \Upupsilon_t     -   \barUpupsilon(\ODEstate_t)]
		\\
		&\Rightarrow  \Upupsilon_t   = \barUpupsilon_t +  \alpha [\Df \haUpupsilon ] (\ODEstate_t, \Phi_t ) 
		-\ddt \haUpupsilon_t
	\end{aligned} 
	\label{e:step3}
\end{equation}
	Steps 2 and 3 combined with  \eqref{e:EndOfStep1} lead to the P-mean flow  representation:


\begin{theorem}
\label[theorem]{t:PMF}
Suppose that continuously differentiable solutions to Poisson's equation exist, for each of the three
forcing functions $f$,   $\haf$,  and $\Upupsilon$.   Then, 
\begin{romannum}
\begin{subequations}
\item   The pre-P mean flow representation holds:
\begin{align} 
		\ddt Y_t  & =   \alpha\bigl[   \barf( Y_t    )   
			- \alpha  \bigl(B_t  \haf_t  +  \Upupsilon_t  \bigr)  \bigr]
		\,,  \qquad B_t    =   \int_0^1 \barA ( Y_t   -  r \alpha \haf_t)    \, dr  
 \label{e:prematureMeanflow}
\\
 \ODEstate_t & =	Y_t    - \alpha \haf_t
 \label{e:Ydef}
\end{align}
\end{subequations}
\begin{subequations}
\item
The P mean flow  representation \eqref{e:BigGlobalODE}  holds with   
		\begin{align}
			\clW_t^0  = \clW^0  (\ODEstate_t,\Phi_t)    & \eqdef  -  [\Df \haUpupsilon ] (\ODEstate_t,\Phi_t)    
			\label{e:AllTheNoise0}   
			\\
			\clW_t^1  = \clW^1  (\ODEstate_t,\Phi_t)    & \eqdef - [\Df \hahaf ] (\ODEstate_t,\Phi_t)   +    \haUpupsilon  (\ODEstate_t,\Phi_t)       
			\label{e:AllTheNoise1}
			\\
			\clW_t^2  = \clW^2  (\ODEstate_t,\Phi_t)    & \eqdef   \hahaf(\ODEstate_t,\Phi_t)  
			\label{e:AllTheNoise2}
\end{align}  
 	\label{e:AllTheNoise}%
\end{subequations}%
\end{romannum}%
\end{theorem}
\begin{proof}
	Using the definition $Y_t \eqdef \ODEstate_t  + \alpha \haf_t$ in \eqref{e:Ydef}, we obtain from \eqref{e:QSA_ODE_fixed} and \eqref{e:EndOfStep1},
	\[
	\ddt Y_t = \ddt \ODEstate_t+ \alpha \ddt \haf_t = \alpha[\barf(\ODEstate_t) -  \alpha\Upupsilon_t]
	\]
	The following identity is part of the proof of the mean value theorem:
	\[\begin{aligned}
		\barf(\ODEstate_t)  = \barf(Y_t - \alpha \haf_t) &= \barf(Y_t) + \int^1_0 \ddr \barf(Y_t - r\alpha \haf_t) \, dr
		\\
		&= \barf(Y_t) - \alpha\int^1_0 \partial_\theta \barf(Y_t - r\alpha \haf_t) \, dr \haf_t
	\end{aligned}
	\]
	which completes the proof for part (i).
	
	The representation
	\eqref{e:BigGlobalODE} in part (ii)
	is obtained on substituting the identities in  \eqref{e:step2} and \eqref{e:step3}  into \eqref{e:EndOfStep1}.
\end{proof}

\Cref{t:hah} provides conditions ensuring existence of smooth solutions to Poisson's equation.
Analysis in \Cref{s:BakerPoisson} also leads to the proof that $\barUpupsilon(\theta) = \Zero$ when the frequencies satisfy (A0).



\subsection{Some stability theory}
\label{s:someStability}

 The Lyapunov bound is recalled here in slightly modified form, and given a new name:
\begin{romannum}  
	\item[\textbf{(V4)}] 
	There exists a $C^1$ function $V:\Re^d\rightarrow[1,\infty)$ and  a constant $\delta_0>0$ such that  the following bound holds for \eqref{e:barf_ODE_gen} for any time $\tau\ge 0$ for which $ \|\odestate_\tau\| \ge   \delta_0^{-1} $:
	\begin{equation}
		\frac{d}{d\tau} V(\odestate_{\tau})  \leq  -  \delta_0  V(\odestate_{\tau}) 
		\label{e:ddt_bound_Lyapfun}
	\end{equation}
	Moreover, the  function is Lipschitz continuous with linear growth:  
	there exists a constant $\Lip_V  <\infty$ such that  
	\begin{equation}
		\begin{aligned}
			|  V(\theta') - V(\theta)|  &  \le \Lip_V  \|\theta' - \theta\|\,, &&   \textit{ for all $\theta$,  $\theta'$. }
			\\
			V(\theta)  &\ge  \|\theta \|\,,  && \textit{when $\|\theta \|  \ge \delta_0^{-1} $}
		\end{aligned} 
		\label{e:VLip}
	\end{equation}	  
\end{romannum}
The use of ``(V4)'' is because of its close connection with a similar ``drift condition'' in the theory of Markov chains and processes \cite{MT}.  

It is sometimes convenient to consider a relaxation:
\begin{romannum}
	\item[\textbf{(V4')}] 
	There is a function $V\colon\Re^d\to\Re_+$ satisfying the bound \eqref{e:VLip},
	together with constants $T<\infty$ and $\delta_1>0$ such that for each initial condition $\odestate_0$ and  $\tau\ge 0$,
	\begin{equation}
		V(\odestate_{\tau+T}) - V(\odestate_\tau) 
		\leq   - \delta_1 \|\odestate_\tau \| , \quad \text{for all } \tau \geq  0\,, \ \|\odestate_\tau \|> \delta_1^{-1}
		\label{e:V4T}
	\end{equation}
\end{romannum}
This unifies the theory since it is obviously implied by (V4),  and we will see that it is  implied by asymptotic
stability of the ODE@$\infty$.

The implications of (V4') for the mean flow are summarized first:
\begin{proposition}
	\label[proposition]{t:qsaprobeExpAS}
	Suppose that (V4') holds, along with (A1) and (A4).   Then the ODE \eqref{e:barf_ODE_gen} is exponentially asymptotically stable:  for positive constants $b$ and $\delta$,  and any initial  condition $\odestate_0$,
	\[
	\| \odestate_t -  \theta^\ocp \|   \le  b \| \odestate_0-  \theta^\ocp \|    e^{-\delta t}  \,,\qquad t\ge 0
	\]
\end{proposition}

We now focus on the implications to the QSA ODE. Consider the family of stopping times,  for any $b>0$,   
\[
\tau^\alpha_{b} (q_0) =  \min \{  t\ge 0 :    \| \ODEstate_t -\theta^* \|  \le  b  \}  \,,\qquad q_0 = (\ODEstate_0,\Phi_0)\in\bivstate \, .
\]

Filter design proceeds under the assumption that the state process is ultimately bounded, in the sense of 	\eqref{e:UB}, but we require something slightly stronger for the family of QSA ODEs with fixed gain:

\wham{\boldmath{$\alpha^0$}-ultimate boundedness:}   With  $\alpha^0>0$ constant, the family of QSA ODEs,
parameterized by a constant gain $\alpha \in (0,\alpha^0]$, 
is   \textit{$\alpha^0$-ultimately bounded} if there are fixed constants $\lilb \leq \bigb < \infty$  such that 
$\sup_{q_0 \in K}\tau^\alpha_{\lilb}(q_0) <\infty$ for any compact set $K \subseteq \bivstate$, and    
\begin{equation*}
	\| \ODEstate_t -\theta^* \|  \le  {\bigb}  \,,\qquad t\ge \tau^\alpha_{\lilb}(q_0)
\end{equation*}

\Cref{t:ODEstate_is_bdd}  extends the bound \eqref{e:ExpStableAlmost} to the QSA ODE under (V4'),
and provides sufficient conditions for   $\alpha^0$-ultimately boundedness.

\begin{theorem}
	\label[theorem]{t:ODEstate_is_bdd}
	Suppose that  (A1), (A2), (A4) and (V4') hold for the QSA ODE \eqref{e:QSAgen_alpha}.
	Then, there is $\alpha^0>0$ and   positive constants $\bdd{t:ODEstate_is_bdd}$ and $\bdde{t:ODEstate_is_bdd}$ such that for any $\alpha\in (0,\alpha^0]$ and any initial condition $q_0 = (\ODEstate_0,\Phi_0)$,  the following hold with   $\lilb =  1/\bdde{t:ODEstate_is_bdd} $:
	
	\whamrm{(i)} 
	$\displaystyle   \tau^{\alpha}_{\lilb}(q_0)   
	=  \frac{\alpha^0}{\alpha} \tau^{\alpha^0}_{\lilb}(q_0) < \infty$.

	\whamrm{(ii)}  $\displaystyle		\| \ODEstate_t  - \theta^*\|   \le  \bdd{t:ODEstate_is_bdd} \| \ODEstate_0  - \theta^* \| \exp( -\alpha \bdde{t:ODEstate_is_bdd} t )  \,, $ 
	\quad
	for $ t \le \tau^\alpha_{\lilb}(q_0)  $
	
	\smallskip

	\whamrm{(iii)}  There is coupling of solutions:
	\[
	\| \flow_t (\ODEstate_0,\Phi_0) - \flow_t(\theta^*,\Phi_0) \| \leq
	\bdd{t:ODEstate_is_bdd}  \exp\bigl(-\bdde{t:ODEstate_is_bdd} (t - \tau^\alpha_{\lilb}(q_0) ) \bigr) 
	\,, 
	\quad\textit{for $ t \ge \tau^\alpha_{\lilb}(q_0)  $.}
	\]
\end{theorem}

The full proof of the theorem is postponed to the Appendix. Part (i) follows from time-scaling,  while (ii) is based
on
the following two steps:
\begin{alphanum}
	\item     Bounds on the difference between solutions to
	QSA  ODE \eqref{e:QSAgen_alpha}
	and
	the solutions to the ODE, 
	\begin{equation}
		\ddt \odestate_t  =   \alpha \barf  ( \odestate_t  )
		\label{e:barf_ODE_gen_alpha}
	\end{equation}
	
	\item  A proof that (V4') implies a semi-exponential for the mean flow as in \eqref{e:ExpStableAlmost}.
\end{alphanum}

Part (iii) is based on the theory of  Lyapunov exponents, defined in terms of the sensitivity process
$\clS_t^0 \eqdef \partial_\theta \flow_t(\theta,z)$ with    $\Psi_0 = (\theta ,z ) \in \bivstate$.  Subject to (A1) and (A2), it is the solution to the ODE
\begin{equation}
\ddt \clS_t^0  =   \alpha  A(\ODEstate_t, \qsaprobe_t) \clS^0_t \,,   \qquad   \clS_0^0 = I  \ (d\times d) \, .
\label{e:SensitivityODE}
\end{equation}
The Lyapunov exponent is defined as  
\begin{equation}
	\Lambda_{\ODEstate} = \lim_{t\to\infty}\tfrac{1}{t} \log (\| \clS_t^0 \|)
	\label{e:LyapExp}
\end{equation}
Under the assumptions of \Cref{t:ODEstate_is_bdd}, it is shown that  $\Lambda_{\ODEstate} $ is negative and independent of the initial condition,  which implies the desired conclusions.

\Cref{t:ODEstate_is_bdd}~(iii) implies uniqueness of  the invariant measure for  $\bfPsi$.    The following summarizes further conclusions;
in particular, how $\alpha$ impacts steady-state bias.

\begin{theorem}
	\label[theorem]{t:varpi}
	Consider the QSA ODE \eqref{e:QSAgen_alpha} subject to  (A0a), (A1)--(A4).
	Suppose moreover that (V4') holds. 
	Then,  
	\whamrm{(i)} 
	$\bfPsi$ admits a unique invariant measure $\upvarpi$.
	It is also an e-process that admits the steady-state realization \eqref{e:PsiState}.    Consequently,   the limits \eqref{e:biasdef},\eqref{e:norms} and \eqref{e:meanTargetBias} hold for each initial condition $\Psi_0$. 
	
	\whamrm{(ii)} 
	$\displaystyle
	\Expect  \bigl[ \barf(\ODEstate)  \bigr]
	=   \alpha  \Expect  \bigl[  \barUpupsilon(\ODEstate)  \bigr]   - \alpha^2 \Expect  \bigl[  \clW^0    (\Uppsi)  \bigr]  
	$,
	where $\Uppsi = (\ODEstate, \Upphi)  $ denotes a $\bivstate$-valued random vector with distribution $ \upvarpi$. This implies $b_\barftwo = O(\alpha)$.

	
	\whamrm{(iii)} 
	If in addition (A0) holds, then     $\barUpupsilon(\theta) = 0 $  for all $\theta$, $b_\barftwo = O(\alpha^2)$, and
	\[
	\Expect  \bigl[ \barf(\ODEstate)  \bigr]
	=   - \alpha^2 \Expect  \bigl[  \clW^0    (\Uppsi)  \bigr]  
	\]

\end{theorem}

\smallskip

\subsubsection{QSA solidarity}

Here we establish bounds between the solutions of the QSA ODE and the mean flow \eqref{e:barf_ODE_gen_alpha}.    However, it  is convenient to eliminate the $\alpha$ in  \eqref{e:barf_ODE_gen_alpha} through a change of time-scale:  
If $\bfodestate$ is a solution to  \eqref{e:barf_ODE_gen_alpha},   then $\{ \odestate_{t/\alpha} : t\ge 0\}$ is a solution to \eqref{e:barf_ODE_gen}.    

Similarly, on denoting $\ODEstate_t^\alpha = \ODEstate_{t/\alpha} $,    the scaling is removed from \eqref{e:QSAgen_alpha}:
\begin{equation}
	\ddt\ODEstate^\alpha_t =   f(\ODEstate^\alpha_t, \qsaprobe^\alpha_t   )
	\label{e:QSAgen_alpha_time-scaled}
\end{equation}
with $\qsaprobe^\alpha_t \eqdef \qsaprobe_{t/\alpha} $.
Hence time is speeded up for the probing signal when $\alpha>0 $ is small.   

\begin{proposition}
	\label[proposition]{t:Couple_theta_local}
	Consider the QSA ODE  subject to (A1).
	There is $\alpha^0>0$, $B_0<\infty $,  and $\epsy_0 \colon\Re_+\to\Re_+$
	such that the following uniform bounds hold for respective solutions to \eqref{e:QSAgen_alpha_time-scaled} and \eqref{e:barf_ODE_gen}, with common initial condition  $\ODEstate_0 =\odestate_0=\theta_0\in\Re^d$.
	For each
	$\Phi_0\in \prstate$, 
	$T>0$,
	and $\alpha\in (0,\alpha^0]$: 
	\begin{equation}
		\|\ODEstate_T^\alpha - \odestate_T\| \leq      B_0 \exp(2 L_f T)    \bigl[  1     +      \|\theta_0\|  \bigr]  \epsy_0( \alpha)
		\label{e:Couple_theta_local}
	\end{equation}
	with  $  \epsy_0( \alpha  ) \to 0 $ as $\alpha \downarrow 0$.  
\end{proposition}

\begin{figure}[h]
	\centering
	\includegraphics[width= .6\hsize]{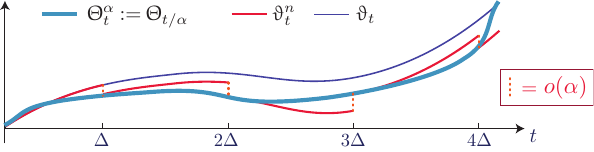}  
	\caption{Solidarity between the time-scaled QSA ODE and the mean flow. }
	\label{fig:SolidarityFixedGain}
\end{figure}

The proof of the proposition is contained in the Appendix. 
The main idea is illustrated in \Cref{fig:SolidarityFixedGain}:  the time axis is partitioned into intervals of width $\Delta>0$ on which $\ODEstate_t^\alpha$ is approximately constant.     With 
$\tau_n = n\Delta$,   we let   $\{ \odestate^n_t :    t \ge   \tau_{n} \}$ denote the solution to \eqref{e:barf_ODE_gen}  with $\odestate^n_{\tau_n} = \ODEstate^\alpha_{\tau_n}$.    The proof follows in three steps:  
1.\  bounds on the error $\|   \odestate^n_t - \ODEstate^\alpha_t \|$ for $t\in [\tau_n, \tau_{n+1}]$;
2.\     bounds on  $\|   \odestate^n_t - \odestate_t \|$ for all $t$ and $n$,  and
3.\  combine 1 and 2 (with minor additional approximations) to complete the proof.

\subsubsection{ODE@\texorpdfstring{$\infty$}{infinity}}

We now explain how to establish (V4') based on the ODE@$\infty$ introduced in \Cref{s:QSAstable}, with vector field $
\barf_\infty $ defined in 	\eqref{e:barfinfty_QSA}.  
The limit exists in many applications, and is often much simpler than $\barf$.  In particular, $\barf_\infty(0) =0$ (the origin is a stationary point),  and the function is radially linear:  $\barf_\infty (c\theta)  = c\barf_\infty (\theta) $ for any $\theta\in\Re^d$ and $c\ge 0$.  Based on these properties it is known that the origin is locally stable in the sense of Lyapunov if and only if it is globally exponentially asymptotically stable \cite{bormey00a}.

The following result follows from arguments leading to the proof of   \cite[Prop.~4.22]{CSRL}:
\begin{theorem}
	\label[theorem]{t:V4TfromBM}
	Suppose that the limit \eqref{e:barfinfty_QSA} exists, and suppose that the origin is  asymptotically stable for the  ODE@$\infty$.   Then, there is a Lipschitz continuous solution to \eqref{e:V4T}.  \qed
\end{theorem}

The proof of the theorem is contained in the Appendix, where it is shown that the following function is a solution to \eqref{e:V4T},
similar to \eqref{e:notBMV}:
\begin{equation}
	V(\theta) = \int_0^{T}  \| \odestate_t^\infty \|\,dt  \,,\qquad   \odestate_0^\infty =\theta\in\Re^d
	\label{e:BMV}
\end{equation}
Lipschitz continuity of both $\barf_\infty$ and $V$ is inherited from $\barf$.

\smallskip

To illustrate application of the theorem, consider the gradient flow for which $\barf =   -  \nabla \Obj$:
\begin{equation}
	\ddt \odestate =     -  \nabla \Obj\, (\odestate)
	\label{e:GradFlow}
\end{equation}
with $\Obj\colon\Re^d\to\Re$,   with minimum denoted $\Obj^\opt$.
The existence of $\barf_\infty$ implies that the following limit also holds: 
\begin{equation}
	\Obj_\infty (\theta) \eqdef \lim_{r \to \infty }  r^{-2}  \Obj(r\theta)    \,, \qquad \theta\in\Re^d 
	\label{e:ObjInfinity}
\end{equation} 
This follows from the representation
\[
\Obj(\theta) -\Obj( 0) = \int_0^1 \ddt \Obj(t\theta) \, dt  = \int_0^1 \theta^\transpose\nabla \Obj\, (t\theta) \, dt   
\]
which implies the representation $\barf_\infty(\theta) = - \nabla \Obj_\infty(\theta)$ and
\begin{equation}
	\Obj_\infty (\theta)   =  - \int_0^1 \theta^\transpose  \nabla \Obj_\infty(t\theta) \, dt   \,, \qquad \theta\in\Re^d 
	\label{e:ObjInfinity2}
\end{equation}

The gradient flow is ultimately bounded under mild assumptions:

\begin{proposition}
	\label[proposition]{t:gradFlowODE}
	Suppose that $\Obj_\infty$ is  $C^1$,   and $\Obj_\infty(\theta)>0$   for $\theta\neq 0$.     Then the ODE@$\infty$ for the gradient flow is exponentially asymptotically stable:  for some $B<\infty$, $\delta>0$, and any initial $\odestate^\infty_0\in\Re^d$,
	\begin{equation}
		\|\odestate^\infty_t \|^2 
		\le  B   \|\odestate^\infty_0 \|^2  e^{ -  \delta    t} \,,\qquad \textit{for $t\ge 0$.}
		\label{e:expASgradFlowODE}
	\end{equation}
	Consequently,  the gradient flow itself is ultimately bounded.
\end{proposition}

\begin{proof}
	Both $\Obj_\infty$  and $\| \nabla \Obj_\infty \|^2$ are bounded above and below by quadratic functions---this follows from continuity and the radial homogeneity properties
	\[
	\Obj_\infty( r \theta)  = r^2\Obj_\infty(  \theta) \,,\qquad   \nabla \Obj_\infty\, ( r \theta)  = r\nabla\Obj_\infty\, (  \theta)\,,\qquad \theta\in\Re^d\,, \  r>0
	\]
	Hence the following constants exist as finite positive numbers,
	\[
	\kappa_+ = \max \frac{\Obj_\infty(\theta)} {\| \theta\|^2} 
	\,,\quad 
	\kappa_- = \min \frac{\Obj_\infty(\theta)} {\| \theta\|^2}  \,, 
	\qquad
	\chi_+ = \max \frac{ \| \nabla\Obj_\infty\, (  \theta) \|^2} {\| \theta\|^2}  \,,\quad 
	\chi_- = \min \frac{ \| \nabla\Obj_\infty\, (  \theta) \|^2} {\| \theta\|^2}  \,,  
	\]
	where the min and max are over non-zero $\theta\in\Re^d$.

	We have $\ddt \Obj_\infty( \odestate^\infty) = -  \| \nabla \Obj_\infty\, ( \odestate^\infty) \|^2$,  and from the above definitions we conclude that  $\ddt \Obj_\infty( \odestate^\infty) = - \delta  \Obj_\infty( \odestate^\infty)$ with $\delta = \chi_- /\kappa_+$.
	Hence $ \Obj_\infty( \odestate^\infty_t) \le e^{-\delta t}  \Obj_\infty( \odestate^\infty_0)$,  which completes the proof using 
	$B = \kappa_+/{\kappa_-}$.  
\end{proof}

\subsection{Implications to filter design}
\label{s:filterDesign}


The matrix $A^*  =  \barA(\theta^*)$ is invertible under (A4).
In this case we define a new function $\barY\colon\Re^d\to\Re^d$ via
\begin{equation}
 \barY  =[A^*]^{-1} \barUpupsilon
\label{e:barY}
\end{equation}
along with the fixed vectors   $\barUpupsilon^* = \barUpupsilon(\theta^*)$  and $\barY^* = [A^*]^{-1} \barUpupsilon^*$. 
This notation is used for approximations obtained in  \cite{laumey22,laumey22e,CSRL} for vanishing gain algorithms.

For reasons that will become clear in the proof of \Cref{t:Couple_theta} that follows,   a filter that obtains AAD of order $O(\alpha^2) $  can be chosen to be of second order, with bandwidth approximately equal to $\alpha$.   
It must also have unity DC gain, which brings us to the following special form for the filtered estimates:
\begin{equation}
 \tfrac{d^2}{dt^2}  \ODEstateF_t 
+
2\gamma\zeta\tfrac{d}{dt}  \ODEstateF_t
+\gamma^2\ODEstateF_t  = \gamma^2 \ODEstate_t
\label{e:GVSOS}
\end{equation}
It is the unity gain that prevents us from eliminating the error inherited from $\barUpupsilon$,  illustrated using $\barY^\ocp$ in
\Cref{fig:QSArep}.   

If $\barY^*=0$, then the coefficients $\gamma$ and $\zeta$ can be designed to reduce AAD dramatically.
Theory predicts that a first order filter is not effective in general.   

In the following companion to \Cref{t:BigQSA} we consider the family of QSA ODEs with fixed gain  $\alpha$,
and let the natural frequency $\gamma$ scale with $\alpha$.

\begin{subequations}
\begin{theorem}
\label[theorem]{t:Couple_theta}
Consider the solution to the QSA ODE with fixed gain.  
For each $\alpha>0$,   the filtered estimates are obtained using \eqref{e:GVSOS},  in which
the value of $\zeta \in (0,1)$ is fixed,  and $\gamma =\eta \alpha$ with $\eta>0$ also fixed.

Subject to  Assumptions~(A1)--(A4) and the Lyapunov bound (V4'), the following conclusions hold for $0<\alpha\le \alpha^0$:
\begin{align}
	\|\ODEstate_t - \theta^*\|  & \leq O(\alpha)  + o(1)  
\label{e:fixedGain1}
\\	
\|\ODEstateF_t - \theta^*\|  & \leq O(\alpha \gamma + \alpha^2)  + O(\alpha \|\barUpupsilon^*\|) + o(1)   
\label{e:fixedGain1PR}
\end{align}
with $\barUpupsilon^* =\barUpupsilon(\theta^*)$. 
Consequently,   the AAD for   $\ODEstateF$  is  bounded by $O(\alpha^2)$ provided    $\barUpupsilon^* = 0$.
\end{theorem}
 \label{e:BigFilter}
\end{subequations} 

 The full proof is postponed to \Cref{s:AppStability}.  
The proof of \eqref{e:fixedGain1} follows from two conclusions from the given assumptions:  that the ODE \eqref{e:barf_ODE_gen} is exponentially asymptotically stable,  by \Cref{t:qsaprobeExpAS},   which implies the existence of a smooth Lyapunov function for the mean flow.
One solution is
\begin{equation}
	V(\theta) = \int_0^{T} e^{\delta_0 t} \| \odestate_t -\theta^* \|^2\,dt  \,,\qquad   \odestate_0 =\theta\in\Re^d
	\label{e:notBMV}
\end{equation}
with $\delta_0>0$ chosen so that the integrand is vanishing,  and $T>0$ sufficiently large so that  $e^{\delta_0 T} \| \odestate_T -\theta^* \|  \le \half  \| \theta -\theta^* \| $ for any $\theta$.  
This Lyapunov function is then applied to  the representation  \eqref{e:prematureMeanflow}.       The proof of \eqref{e:fixedGain1PR} is then simpler since we can apply \eqref{e:fixedGain1} to justify linearization around $\theta^*$,  applying an entirely linear systems theory analysis.  

The approximations \eqref{e:BigFilter}  imply bounds on the absolute deviation of parameter estimates,  and hence the AAD.
Bounds on bias and variance also  follow as corollaries to \Cref{t:Couple_theta}.

\begin{corollary}
	\label[corollary]{t:asympt_bias} 
	Under the assumptions of \Cref{t:Couple_theta},
	
	\begin{subequations}
		\begin{romannum}
		\item
		The asymptotic bias and variance \eqref{e:biasdef} admit the bounds,  
		\begin{align}
			\bias_{\ODEstate}  &= O(\alpha)\,,
			\quad 
			\sigma^2_{\ODEstate} = O(\alpha^2)  \,,
			\label{e:part1corr}
			\\
			\bias_{\ODEstateF} &= O(\alpha) \, ,
			\quad \sigma^2_{\ODEstateF} = O(\alpha^2)  \, ,
			\label{e:part1corrF}
		\end{align}	
		\item
		If in addition (A0) holds, then  
		\begin{align}
			\bias_{\ODEstate} &= O(\alpha^2)  \,,
			\quad  \sigma^2_{\ODEstate} = O(\alpha^2)  \, , 
			\label{e:part2corr}
			\\
			\bias_{\ODEstateF} &= O(\alpha^2)\,,
			\quad  \sigma^2_{\ODEstateF} = O(\alpha^4) \, .
			\label{e:part2corrF}
		\end{align}
	\end{romannum}
	\label{e:Cor_bias}
	\end{subequations}
\qed
\end{corollary}

Assumption~(A0) has the largest impact on bias and variance.    Equation~\eqref{e:part2corr}
tells us that the variance is of order $O(\alpha^2)$ subject to this restriction on frequencies,
which is remarkable when compared with standard results from SA theory 
(see  \eqref{e:BM} and discussion that follows).    
Filtering brings the variance down to   $O(\alpha^4)$;   a restatement of the second bound in \eqref{e:part2corrF}.

\subsection{Examples}

\subsubsection{Choice of Frequencies}
  \begin{subequations}
	
Two linear QSA ODEs will be used to illustrate the impact of $\barUpupsilon^*$ on AAD:  
	\begin{align}
		\ddt \ODEstate_t &= \alpha[A_t\ODEstate_t + 2\sin(\omega t)+1]  
		\label{e:lin_barY0}
		\\
		\ddt \ODEstate_t &= \alpha[   A_t\ODEstate_t + 2\cos(\omega t)+1]  
		\label{e:lin_barYnot}
\end{align}  
where  $A_t = -(1+\sin(\omega t))$ with $\omega>0$. 
Assumption (A0b) holds for \eqref{e:lin_barY0} with $K=1$  in \eqref{e:qsaprobe0}, using
 $\qsaprobe^0_t = \cos(\omega t)$. 
 In the second case \eqref{e:lin_barYnot}, Assumption (A0b) is \textit{violated}:  any realization of $\qsaprobe^0_t $
  will require $K=2$ with $\omega_1=\omega_2$.

We   obtain a formula for  $\barUpupsilon^* $ in each case, based on the following observations:  
	\begin{romannum}
		\item They share the common mean vector field $\barf(\theta) =-\theta+1$, so that $\theta^\ocp =1$.
		
		\item  They also share the same linearizations:  $A(\theta,\qsaprobe) = \partial_\theta f(\theta,\qsaprobe) =  -(1+\sin(\omega t)) $  and $A^\ocp=-1$.   
		
Consequently, $\barY(\theta) = [A^\ocp]^{-1} \barUpupsilon(\theta) = - \barUpupsilon(\theta) $,  and
for any $\theta$ and $t$,
\[ 
 \haA( \theta, \Phi_t )   =  - \omega^{-1} \cos(\omega t)   
\]
Observe that  $ \int \haA( \theta, z )  \,\uppi(dz) = \Zero$ 
and $\ddt  \haA( \theta, \Phi_t ) = -[  A (\theta,\qsaprobe_t) - \barA(\theta)]$, as required.

		\item  A crucial difference is the steady-state apparent noise:  $f(\theta^\ocp, \qsaprobe_t) =  \sin(\omega t)$ for  \eqref{e:lin_barY0},  and 
		$f(\theta^\ocp, \qsaprobe_t) = - \sin(\omega t) + 2\cos( \omega t)$ for  \eqref{e:lin_barYnot}. 
	\end{romannum}
From the \eqref{e:haEverythingElse}  we find that $\barUpupsilon(\theta) = \barUpupsilon^\ocp$ (independent of $\theta$):  
\[
	\begin{aligned}
-\barY(\theta)  = 
		\barUpupsilon(\theta)  & \eqdef 
		- \lim_{T\to\infty}\frac{1}{T} \int_0^T  \haA( \theta, \Phi_t ) \cdot  f(\theta, \qsaprobe_t)  \, dt
		\\
		&
		= 
		\begin{cases}
			\displaystyle
			\ \ 
			  \   0    &   \text{Model    \eqref{e:lin_barY0}}
			\\
			\displaystyle
			\ \ 
			-\lim_{T\to\infty}\frac{1}{T} \int_0^T \{ -\omega^{-1}  \cos(\omega t) \} \cdot \{ 2\cos(\omega t) \} \, dt  =  \frac{1}{\omega}  \quad  &   \text{Model    \eqref{e:lin_barYnot}}
		\end{cases}
\end{aligned} 
\] 
In each case,  we have used the fact that  $ \haA( \theta, \Phi_t )$ and $ \haA( \theta, \Phi_t )    \cdot \sin(\omega t)$ each have time average equal to zero.
	
\label{e:TwoLinExamples}
\end{subequations}

\wham{Simulation Setup} 
We are interested in comparing AAD for $\bfODEstate$ with and without filtering.
The special case $\omega =  0.1$ was considered, 
many  gains were tested in the range $\alpha \in [10^{-3}, 0.9]$,
and in each experiment the solution to the QSA ODE was approximated
using an Euler scheme with sampling time of $0.1$~sec. 
 
\begin{subequations}
 Data $\{\ODEstate_t\}$ for each choice of $\alpha$ was then used to obtain the   two filtered estimates:
\begin{equation}
{\ODEstate}^{1 \text{\tiny\sf F}}_t = \int^t_0 h^1_{t - \tau} \ODEstate_\tau\,d\tau, \qquad {\ODEstate}^{2 \text{\tiny\sf F}}_t = \int^t_0 h^2_{t - \tau} \ODEstate_\tau\,d\tau
\end{equation}
where $h^1_t$ and $h^2_t$ have Laplace transforms:
\begin{equation}
	H_1(s) = \frac{\gamma}{s+\gamma}, \qquad  H_2(s) = \frac{\gamma^2}{s^2+2\zeta \gamma s + \gamma^2 }
	\,,
	\qquad  \textit{ with $\gamma = \alpha$, $\zeta =0.8$. }
	\label{e:filter_numerics}
\end{equation}
\label{e:filters}
\end{subequations}
The steady-state $L_1$ norm of the error was approximated using the Monte-Carlo estimate,
\begin{equation}
\| \ODEstate_{\alpha}\|_{L_1} \eqdef  \frac{1}{T - T_0} \int^T_{T_0}  |\ODEstate_\tau -\theta^* | \, d\tau
\label{e:AAD_MC}
\end{equation}
taking  $T_0 = 0.8 T$, so that only the final $20\%$ of the run was included.  This was repeated to obtain $ \|{\ODEstate}^{1 \text{\tiny\sf F}}_\alpha \|_{L_1}$ and $\|{\ODEstate}^{2 \text{\tiny\sf F}}_\alpha \|_{L_1}$ for each value of $\alpha$ considered. 

\begin{figure}[h]
	\centering
	\includegraphics[width=1\hsize]{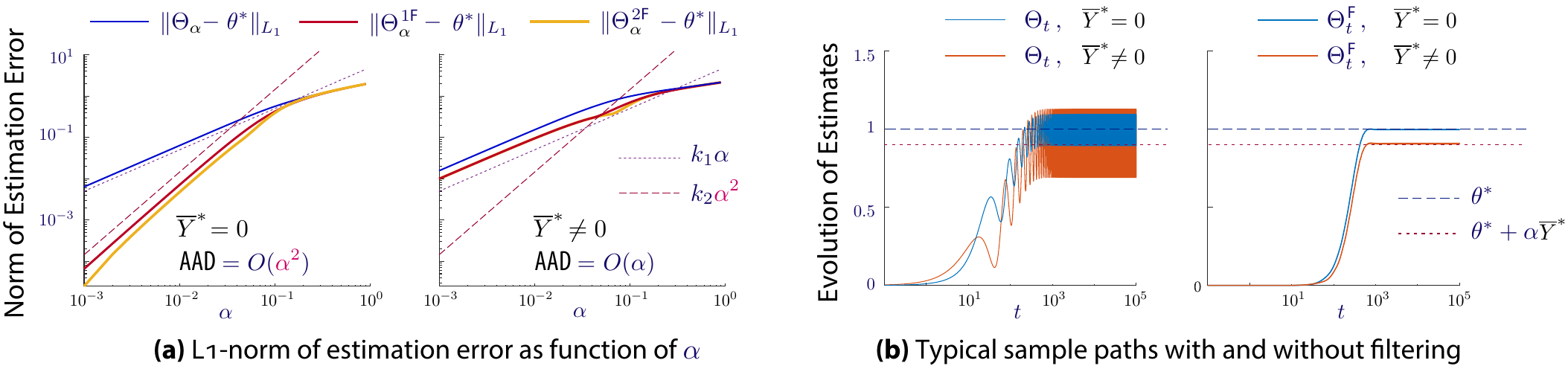}
	\caption{AAD as a function of $\alpha$ for different filtering techniques (left).   Impact of $\barY^*$ on  estimation error (right).}%
	\label{fig:alpha2bias}
\end{figure}

\wham{Results}
The final estimates for each case are shown plotted against $\alpha$ on a logarithmic scale in \Cref{fig:alpha2bias}~(a).
Also shown for comparison with expected bounds are plots of the linear functions $r_1(\alpha) = k_1 \alpha$, $r_2(\alpha) = k_2 \alpha^2$. The constants $k_1, k_2$ were chosen to aid comparison. The results agree with \Cref{t:Couple_theta}. In particular, we see in \Cref{fig:alpha2bias} that the term $\alpha\barY^*$ dominates the estimation error when $\barY^* \neq 0$. In this case, filtering has no improvement on reducing AAD of order $O(\alpha)$. When $\barY^*=0$, \Cref{fig:alpha2bias}~(a) shows that  filtering reduces AAD to $O(\alpha^2)$ for  $\alpha<0.1$,   and \textit{this is observed for both first and second order low pass filters}.  An explanation for the similar performance is contained in \Cref{s:Linsys_Exp}.

The evolution of the process $\{{\ODEstate}^{2 \text{\tiny\sf F}}_t\}$ for the specific case of $\alpha =0.01$ is shown in \Cref{fig:alpha2bias}~(b).    The volatility of $\bfODEstate$ is massive in each case, and reduced dramatically with filtering.

These results illustrate the importance of carefully designing the probing signal: the inclusion of a phase shift in one of the entries of $\bfqsaprobe$ might appear harmless. In fact,  this small change results in significant estimation error---10\%\ in this example.

\subsubsection{Linear Systems Analysis}
\label{s:Linsys_Exp}
The next example was designed to illustrate \Cref{t:Couple_theta},  and explain why a first order filter may be effective in some cases such as the results shown in
\Cref{fig:alpha2bias}.

Consider the time-homogeneous linear system with almost periodic input,   
\begin{equation}
	\ddt x_t =   \alpha A^*   x_t   		 +  \alpha \clW^*_t  
	\label{e:lin_simple}
\end{equation}
The reason for the change in notation is that $x_t$ is interpreted as an approximation of  $\ODEstate_t -\theta^\ocp$,   in which case $\clW^*_t \eqdef \clW(\theta^*,\Phi_t)$ with $\clW_t$ defined in \eqref{e:AllTheNoise}.    An approximation of this form is used in the proof of \Cref{t:Couple_theta}.


Let $X(s)$, $W(s)$ denote the respective Laplace transforms of the state and input in \eqref{e:lin_simple},  
and $W^i(s)$ the transforms of the components of $\clW_t$ shown in \Cref{t:PMF}.
They are related by
\[
X(s)  =   \alpha  [Is -  \alpha A^* ] ^{-1}  W(s)   =   \alpha  [Is -  \alpha A^* ]^{-1}  \bigl(  \alpha^2  W^0(s)   + \alpha s  W^1(s)   + s^2  W^2(s)   \bigr) 
\]
and using a superscript ``{\sf F}'' for the filtered signals,  
\[
\XF(s)  =   \alpha  [Is -  \alpha A^* ] ^{-1}  \WF{}(s)   =   \alpha  [Is -  \alpha A^* ]^{-1}  \bigl(  \alpha^2   \WF{0}(s)   + \alpha s   \WF{1}(s)   + s^2   \WF{2}(s)   \bigr) 
\]

The assumptions on the  filter \eqref{e:GVSOS} imposed in  \Cref{t:Couple_theta} ensure that $ s^2  \WF{2} (s) $ and $    s  \WF{1} (s)  $  are attenuated to obtain the desired error bounds---see  \Cref{t:FilterDesign} for details.

The matrix-valued transfer function $[Is - \alpha A^*]^{-1}  $ will attenuate these signals if $\alpha$ is small, and the signals are band-limited.     In the experiment illustrated in \Cref{fig:alpha2bias},   a second order filter is not needed because $\alpha A^*= - \alpha$,  and the spectrum of $\clW_t^*$ is discrete.

\subsubsection{Gradient Descent}

The next example illustrates \Cref{t:Couple_theta}.

Minimization of a scalar-valued objective $\Obj: \Re^d \to \Re_+ $ might be achieved using the   following ``quasi stochastic gradient descent'' ODE:
\begin{equation}
	\ddt \ODEstate_t = -\alpha [\qsaprobe_t \qsaprobe^\transpose_t\nabla\Obj(\ODEstate_t) + \beta \qsaprobe_t]
	\label{e:SGD}
\end{equation}
This is a version of QSA with      $\barf(\theta) = -\Sigmaqsa \nabla\Obj(\theta) $, and $\Sigmaqsa = \lim_{T\to\infty} \frac{1}{T} \int^T_0 \qsaprobe_t \qsaprobe^\transpose_t\,dt$.    The introduction of quasi-randomness is for application to non-convex objective functions.  In particular, the zero mean signal $ \beta \qsaprobe_t$ will ensure the ODE does not converge to a saddle point, and may help to avoid convergence to local minima \cite{gehuajinyua15,orvkerprobacluc22}.

The results that follow are based on the
\textit{Rosenbrock} objective \cite{simulationlib}, 
\begin{equation}
	\Obj(\theta) =100(\theta_2-\theta_1^2)^2+(1-\theta_1)^2
	\label{e:Rosen}
\end{equation}

\wham{Simulation setup:}

The ODE \eqref{e:SGD} was approximated through an Euler scheme with sampling time of $0.1$~sec. 
The probing signal was $\qsaprobe_t = [\sin(t/4),\sin(t/e^2)]^\transpose$ and the value $\beta =5$ in \eqref{e:SGD} 
was selected by trial and error.

Gains were chosen in the range $\alpha \in [10^{-4}, 10^{-2}]$.
For each $\alpha$, the data $\{\ODEstate_t\}$ was used to obtain the filtered estimates $\{\ODEstate^{1 \text{\tiny\sf F}}_t ,\ODEstate^{2 \text{\tiny\sf F}}_t\}$ through \eqref{e:filters} with $\gamma = \alpha$ and $\zeta = 0.8$. The steady-state $L_1$ norm of the error (i.e. the AAD) was then approximated exactly as   in \eqref{e:AAD_MC}.

\begin{wrapfigure}[15]{r}{0.5\textwidth}	
	\centering
	\includegraphics[width=1\hsize]{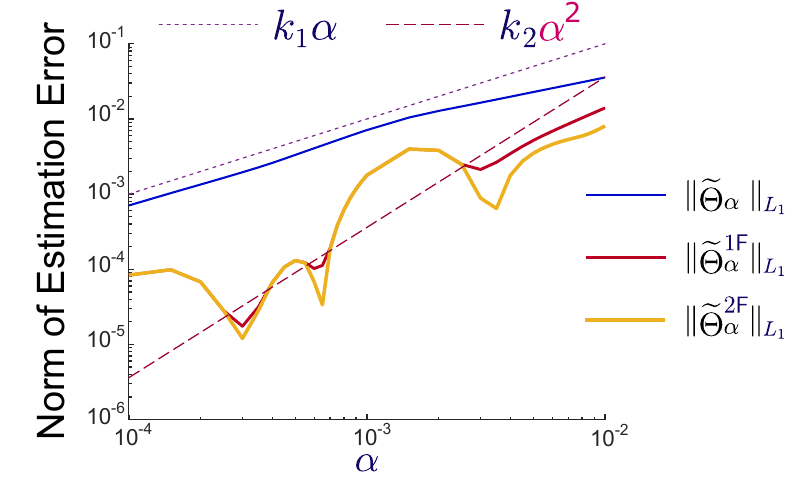}
	
	\caption{Optimization of the Rosenbrock objective: AAD as a function of $\alpha$ for different filtering techniques.}%
	\label{fig:SGDCamel}
\end{wrapfigure} 

\wham{Results:}  AAD estimates are shown plotted against $\alpha$ on a logarithmic scale in \Cref{fig:SGDCamel}.
Similarly to \Cref{fig:alpha2bias}, plots of the linear functions $r_1(\alpha) = k_1 \alpha$, $r_2(\alpha) = k_2 \alpha^2$ are shown for comparison with the bounds expected by \Cref{t:Couple_theta}. Again, the constants $k_1, k_2$ were chosen to aid comparison. 

The results agree with \Cref{t:Couple_theta}. For the unfiltered case, the AAD is of order $O(\alpha)$. Moreover, the AAD of estimates filtered by either filter in \eqref{e:filters} are of order $O(\alpha^2)$.   This improvement is predicted by  \Cref{t:Couple_theta} for the second order filter since $\barUpupsilon^* = \Zero$.

\subsubsection{Almost periodic linear systems}
\label{s:QPLS}

A \textit{quasi periodic linear system} is a special case of QSA with constant gain:
\begin{equation}
	\ddt \ODEstate_t =  \alpha A_t \ODEstate_t\,,\qquad A_t\eqdef A(\qsaprobe_t) \, .
	\label{e:linQSA}
\end{equation}
The mean flow vector field is linear, $\barf(\theta) = A^* \theta$, with $A^*$ the steady state mean of $A_t$,  and we have
$\haf(\theta, z)  = \haA(z) \theta $, where $\haA$ is defined in 	\eqref{e:der_haA}
and solves
\[
\haA(  \Phi_{t_0} ) = 
\int_{t_0} ^{t_1}   [ A( \qsaprobe_t)  - A(\theta) ] \, dt    + \haA(   \Phi_{t_1} )   \,,\qquad 0\le t_0\le t_1
\]

The goal in some of this literature is to construct a 
time varying linear transformation  $\clT_t$ 
(known as the Lyapunov-Perron transformation)
that results in time-homogeneous dynamics,  $	\ddt  Y_t = A^* Y_t$,    $    \ODEstate_t = \clT_t Y_t $
(see for example \cite[Thm. 1]{mur78}).
The fact that this transformation
is not always possible is clear from the fact that \eqref{e:linQSA} may be unstable for some $\alpha$ even when $A^*$ is Hurwitz;  conversely, the time varying system  \eqref{e:linQSA} may be stable even when $A^*$ has a positive eigenvalue---see 
\cite{belbenmee84} for a survey.   

The Lyapunov exponent admits the representation 
\begin{equation}
	\Lambda_{\ODEstate}  = \lim_{t\to\infty}\tfrac{1}{t} \max_\theta \log (\| \ODEstate_t \|) 
	\label{e:LyapExpODEstate}
\end{equation}
where the maximum is over initial conditions   $  \ODEstate_0 =\theta $ satisfying $\|\theta \|=1$.   
\Cref{t:varpi} imples that $	\Lambda_{\ODEstate} <0$ for all $\alpha>0$ sufficiently small.

	\begin{proposition}[Quasi Periodic Linear Systems]
		\label[proposition]{t:QPL}
		Consider the linear QSA ODE \eqref{e:linQSA}, subject to (A0)--(A4),  and let 
		$\Lambda_{\ODEstate}^\alpha$ denote the Lyapunov exponent associated with \eqref{e:linQSA}.

		Then, there exists $\alpha^0>0$ such that
		$\Lambda_{\ODEstate}^\alpha <0$ for   $0<\alpha\le \alpha^0$, and
		\begin{equation}
			\lim_{\alpha\downarrow 0}  \frac{1}{\alpha} \Lambda_{\ODEstate}^\alpha  \leq   \text{\rm Re}(\lambda_1) <0
			\label{e:linQSAasymptotics}
		\end{equation}
		where $\lambda_1$ is an eigenvalue of  $A^\ocp$ with maximal real part. 
	\end{proposition}
	\begin{proof}
		
		The following calculation is identical to the first step in the proof of  \Cref{t:PMF}, 
		using the shorthand $\haA_t = \haA(\Phi_t)$:
		\[
		\begin{aligned}
			\ddt \haf(\ODEstate_t,  \Phi_t)    \eqdef  \ddt \bigl( \haA_t \ODEstate_t  \bigr)
			& = \bigl(  \ddt \haA_t  \bigr)  \ODEstate_t  +  \haA_t  \bigl(  \ddt \ODEstate_t \bigr)  
			\\
			& =    - [ A_t  - A^\ocp ] \ODEstate_t    +  \alpha \haA_t  A_t \ODEstate_t   
		\end{aligned} 
		\] 
		Note that    $\Upupsilon_t \eqdef - \haA_t  A_t \ODEstate_t  $ in the notation of
		\Cref{t:PMF}.

		The resulting identity $A_t \ODEstate_t  =   A^\ocp +   \alpha \haA_t  A_t  -  \ddt \bigl( \haA_t \ODEstate_t  \bigr)$  is substituted into  \eqref{e:linQSA}  to obtain
		\[
		\ddt \ODEstate_t =  \alpha   \bigl[  A^\ocp +   \alpha \haA_t  A_t  \bigr] \ODEstate_t   - \alpha \ddt \bigl( \haA_t \ODEstate_t  \bigr)
		\]
		This motivates the introduction of a new state variable $Y_t =  [I + \alpha  \haA_t] \ODEstate_t $ to obtain the state space model
		\[
		\ddt Y_t =  \alpha   \bigl[  A^\ocp +   \alpha \haA_t  A_t  \bigr]     \bigl[  I + \alpha  \haA_t  \bigr]   ^{-1}  Y_t
		\]
		The inverse exists for all sufficiently small $\alpha>0$ since $\{  \haA_t \}$ is bounded,   from which we obtain 
		\begin{equation}
			\ddt Y_t =  \alpha   \bigl[  A^\ocp +   \alpha \clE^\alpha_t  \bigr]      Y_t
			\label{e:linQSArep}
		\end{equation}
		in which  $\{\clE^\alpha_t\}$ is uniformly bounded in $t$, and $\alpha$ in a neighborhood of the origin.  
		For each $t$,
		\[
		\lim_{\alpha\downarrow 0} \tfrac{1}{\alpha} 
		\clE^\alpha_t   = \haA_t  A_t   -    A^\ocp\haA_t   
		\]

		The remainder of the proof is standard.  First, for any positive definite matrix $P$, on taking $V(\theta) = \theta^\transpose P \theta $
		we have 
		\[
		2\Lambda_{\ODEstate}(\theta) 
		= \lim_{t\to\infty}\tfrac{1}{t} \log (V(\ODEstate_t) )   
		= \lim_{t\to\infty}\tfrac{1}{t} \log (V(Y_t) )  
		\]
		This is simply a consequence of the fact that  all norms are equivalent on $\Re^d$,  along with uniform bounds on the norm of $[I + \alpha  \haA_t] $ and its inverse.
		
		The matrix $\delta I +  A^\ocp $ is Hurwitz for any $0\le \delta <  | \text{Re}(\lambda_1)|$, so that there is a solution $P>0$ to the Lyapunov equation  ${A^\ocp}^\transpose P + P A^\ocp   + 2 \delta P =  - I $.   Applying \eqref{e:linQSArep} gives
		\[
		\ddt V(Y_t)  =   -\alpha \bigl[  2\delta V(Y_t)  +  \| Y_t\|^2 \bigr]    + \alpha^2 Y_t^\transpose \bigl[  {\clE^\alpha_t  }^\transpose P + P \clE^\alpha_t  \bigr]  Y_t.
		\]
		Dividing both sides by $ V(Y_t)$ we obtain
		\[
		\ddt \log(V(Y_t)) = -2\alpha \delta - \alpha \frac{\|Y_t \|^2}{V(Y_t)}  +\alpha^2 \frac{Y_t^\transpose \bigl[  {\clE^\alpha_t  }^\transpose P + P \clE^\alpha_t  \bigr]  Y_t}{V(Y_t)}
		\]
		with the second term in the right side becoming small as $\delta$ approaches $ | \text{Re}(\lambda_1)|$. For $\alpha>0$ sufficiently small we obtain the upper bound,
		\[
		\begin{aligned}
			2\Lambda_{\ODEstate}(\theta)    =  \lim_{t\to\infty}\tfrac{1}{t} \log (V(Y_t) )    & \le -2\alpha \delta +  O(\alpha^2)
			\\
			\textit{giving}\quad       \limsup_{\alpha\downarrow 0}  \frac{1}{\alpha} \Lambda_{\ODEstate}^\alpha & \le  -   \delta
		\end{aligned} 
		\]
		This establishes \eqref{e:linQSAasymptotics} since  $\delta <  | \text{Re}(\lambda_1)|$ is arbitrary.
	\end{proof}

\subsection{Extensions}
\label{s:extensions}

Two extensions follow.
\begin{romannum}
\item 
First, we review the main conclusions from \cite{laumey22,laumey22e} for vanishing gain QSA.

\item
Then, we explain how bias bounds extend to stochastic approximation.    
\end{romannum}

\subsubsection{Vanishing gain}

The main result  of \cite{laumey22,laumey22e},  concerns rates of convergence for the QSA ODE,  and averaging techniques to improve these rates.  The gain is taken of the form \eqref{e:arho} with  $\rho \in (\half, 1)$ and $t_e=1$.    
In this setting, the use of  Polyak-Ruppert averaging is sometimes effective.   At time $T$, the averaged parameter estimate is defined as 
\begin{equation}
\ODEstatePR_T \eqdef \frac{1}{T-T_0}  \int_{T_0}^T   \ODEstate_t\, dt
\label{e:PRQSA}
\end{equation}
in which the starting time scales linearly with $T$:  for fixed $\kappa>1$ we chose   $T_0 $ to solve  $1/(T-T_0) = \kappa/T$.

 \begin{subequations}
\begin{theorem}
\label[theorem]{t:BigQSA}
Suppose that  Assumptions~(A1)--(A4) and that the Lyapunov bound (V4') hold. Suppose moreover that $a_t$ is of the form \eqref{e:arho} with  $\rho \in (\half, 1)$ and $t_e=1$.  Then the following conclusions hold:
\begin{align}
\ODEstate_t  & = \theta^\ocp  -  a_t   \haf^*_t  + O(a_t \|\barY^*\|)  + o(a_t)    &&  t\ge t_0
\label{e:BigCouplingQSAcor}
\\	
		 \ODEstatePR_T  &= \theta^*+   a_T  [c(\kappa,\rho)  +o(1)]   \barY^* +  O(T^{-2\rho}) \,,\qquad &&T_0\ge t_0
		\label{e:PR_2rhobound} 
\end{align}
where $c(\kappa,\rho) =c(\kappa,\rho) = \kappa[1-(1-1/\kappa)^{1-\rho}]/(1-\rho)$,
and  $\barY^* = [A^\ocp]^{-1} \barUpupsilon^* $.  
 
Consequently,    $\ODEstatePR_T$ converges to $\theta^\ocp$ with rate bounded by $O(T^{-2\rho})$ if and only if $\barY^* = 0$.
\end{theorem}

 \end{subequations}

This is a remarkable conclusion, since we can choose $\rho$ to obtain a rate of convergence arbitrarily close to $O(T^{-2})$.   

The key step in the proof is to obtain a version of the perturbative mean flow, this time for the scaled error,
 \begin{equation}
\scerror_t = \frac{1}{a_t}  \bigl(  \ODEstate_t -  \theta^\ocp \bigr) \, ,   \qquad t\ge t_0
\label{e:scaled_error}
\end{equation}
The analysis is entirely  local, so that a linearization around $\theta^\ocp$ is justified:
\[
\ddt \scerror_t   =    a_t       [o(1) I +  A^\ocp  ]     \scerror_t   +      \tilXi_t   
\]
in which the definition of the apparent noise $\{  \tilXi_t    \}$  is unchanged.    Applying 	\eqref{e:EndOfStep1} then gives,
\begin{equation} 
\ddt \scerror_t   =    a_t       [o(1) I +  A^\ocp  ]     \scerror_t    -\ddt  \haf_t   -a_t \Upupsilon_t
\label{e:scaledErrorRep_cor_rho}
\end{equation} 
The change of variables $Y_t = \scerror_t   +      \haf_t $ then gives an approximation similar to  \eqref{e:prematureMeanflow}:
\[
\ddt Y_t   =    a_t    \bigl(     [o(1) I +  A^\ocp  ]     Y_t     \bigr)
	  -a_t  \bigl(         [o(1) I +  A^\ocp  ]      \haf_t   +\Upupsilon_t \bigr)
\]
This leads to two consequences:   a rate of convergence of order $O(a_t)$ for $\ODEstate_t$,  and a bound of order $O(a_t^2)$  for $\ODEstatePR_t$.  However, \textit{a rate of order $O(a_t^2)$ is only possible if $\barUpupsilon^*=\Zero$.}

\subsubsection{Implications to stochastic approximation}
\label{s:SAbias}

To obtain extensions to the stochastic setting recall  the definition of $\Upupsilon$ in \eqref{e:Upupsilon},  
and the representations for its steady-state mean value:
\begin{equation}
\begin{aligned}
		\Upupsilon(\theta,z)   &=   -  [\partial_\theta \haf   (\theta,z ) ]\cdot   f (\theta, G(z))       
   \\
\alpha \Expect[ \Upupsilon(\Psi)  ]   &            
			=   - \Expect [ \tilXi]  \,,\qquad\qquad \tilXi = f(\Psi )  -  \barf(\ODEstate)   
\end{aligned} 
\label{e:UpupsilonDer2}
\end{equation}
The second identity is obtained on combining \eqref{e:EndOfStep1} with \Cref{t:QSAgenerator}  (with expectations in steady-state).


Something similar appears in a stochastic setting,   and presents a similar challenge for 
	variance
 reduction via averaging.  Here we establish an analog of \Cref{t:varpi}~(ii).

Consider the general SA recursion \eqref{e:SA_recur} with constant gain,
\[
 \theta_{n+1} = \theta_{n}+ \alpha  \bigl[  \barf(\theta_n)  +  \tilXi_n \bigr]
		\,,\qquad    \tilXi_n   =  f(\theta_n,\qsaprobe_{n+1})     - \barf(\theta_n)
\]
It is assumed throughout that $f$ satisfies the uniform Lipschitz condition imposed in (A1).

Assume a finite state space Markovian realization for the probing sequence:   there is an irreducible Markov chain on a finite state space $\state$ and function $G\colon\state \to \Re^{d\times d}$ that specifies $\qsaprobe_n = G(\Phi_n)$.   
Then $\Psi_n \eqdef(\theta_n,\Phi_n)$ is a Markov chain, and it satisfies the Feller property under our standing assumption that $f$ is continuous in $\theta$.  

Assume that the sequence is bounded in $L_2$ for at least one initial condition:  $\sup_n \Expect[\| \theta_n\|^2 ]<\infty$.   The Feller property combined with tightness implies the existence of  at least one invariant probability measure $\upvarpi$ on  $\clB(\Re^d\times \state)$.  Moreover,  under the invariant distribution,  the random vector $\theta_n$ has a finite $p$th moment for $p<2$.
Existence of $\upvarpi$  is justified in Section 12.1.1 of  \cite{MT}, and the uniform second moment implies uniform integrability  of $\{ \| \theta_n\|^p : n\ge 1\}$ from one initial condition, implying $ \Expect[\| \theta_n\|^p] <\infty$ 
when $(\theta_n, \Phi_n)\sim \upvarpi$.


Suppose that $\Psi_0 \sim \upvarpi$ so that $\{ \Psi_n : n\ge 0\}$ is a stationary process.  Taking expectations of both sides of the recursion gives:
\[
\Expect[ \theta_n ] =\Expect[ \theta_{n+1} ] 
			=\Expect[  \theta_{n} ]
			+ \alpha \Expect \bigl[  \barf(\theta_n) \bigr] 
			+ \alpha \Expect \bigl[ \tilXi_n \bigr] 
\]
From this we obtain the identity  
\[
\Expect \bigl[  \barf(\theta_n) \bigr] 
			= -\Expect \bigl[ \tilXi_n \bigr] 
\]
When $\{ \tilXi_n \}$ is a martingale difference sequence, such as when $\bfPhi$ is i.i.d., the right hand side is zero,  which is  encouraging.   In the Markov case this conclusion is no longer valid.

We can see this by bringing in Poisson's equation for the Markov chain:
\[
\Expect[ \haf(\theta ,  \Phi_{n+1})   \mid \clF_n]  =  \haf(\theta ,  \Phi_{n})  -  
\bigl[ f(\theta,\qsaprobe_{n})     - \barf(\theta)\bigr]  
\]
where $\clF_n = \sigma(\Psi_0\,,\dots\,, \Psi_n)$.   This implies a useful representation for the apparent noise, 
\begin{align}
\tilXi_{n}   &=  f(\theta_n,\qsaprobe_{n+1})     - \barf(\theta_n)
=
  \bigl[ \haf(\theta_n ,  \Phi_{n+1})   -  \haf(\theta_n ,  \Phi_{n+2})  \bigr]  + \clW_{n+2} 
 \label{e:SAdecomp}
\end{align}
where $ \clW_{n+1}  =  \haf(\theta ,  \Phi_{n+2})  - \Expect[ \haf(\theta ,  \Phi_{n+2})   \mid \clF_{n+1}]  $ is a martingale difference sequence.

It is assumed that the solution to Poisson's equation is normalized to have zero mean:  $\Zero = \sum_z \haf(\theta , z) \, \uppi(z)$ for each $\theta$,  with $\uppi$ the unique invariant probability measure for $\bfPhi$.
Under the Lipschitz assumption on $f$ we have $\| \haf(\theta , z) \|  \le L_{\haf} [1 + \| \theta\| ]$ for some constant $ L_{\haf} $ and all $(\theta, z)$   (apply any of the standard representations of Poisson's equation in standard texts, such as \cite{MT}).


\begin{proposition}
\label[proposition]{t:ApparentNoiseSA}
The  tracking bias may be expressed
\[
\Expect[\barf(\theta_n)]  =     \alpha \Expect[\Upupsilon_n ]  \,,
\qquad
\textit{with}
\quad
 \Upupsilon_n      \eqdef   -  \frac{1}{\alpha} \bigl[ \haf(\theta_{n+1} ,  \Phi_{n+2})   -  \haf(\theta_n ,  \Phi_{n+2}) \bigr] \,,
\]
where the expectations are taken in steady-state, so independent of $n$.
\end{proposition}
This is entirely consistent with the conclusion in the continuous time case,     in which the derivative appearing in \eqref{e:UpupsilonDer2} is replaced by a difference.

 There are deeper connections explained in the Appendix---see in  particular the formula \eqref{e:newUpupsilon}.   
 As in the deterministic setting, we expect $\Upupsilon_n$ to be bounded over $0 < \alpha\le\alpha^0$ for suitably small $\alpha^0>0$,  but it may be large if the Markov chain has significant memory.

\begin{proof}
Applying \eqref{e:SAdecomp},  the   partial sums of the apparent noise may be expressed 
\[
\sum_{n=1}^{N} \tilXi_{n}  = 
\clM_{N+2}     
+
\bigl[ \haf(\theta_0 ,  \Phi_2)  -  \haf(\theta_{N} ,  \Phi_{N+2}) \bigr]
+
\sum_{n=0}^{N-1} \bigl[ \haf(\theta_{n+1} ,  \Phi_{n+2})   -  \haf(\theta_n ,  \Phi_{n+2}) \bigr]
\]
where $\{  \clM_{N+2}  =   \sum_{n=1}^{N}\clW_{n+2}    \}$ is a zero-mean martingale.
Consider now the stationary regime with  $\Psi_0 \sim \upvarpi$.  
Taking expectations of both sides, dividing each side by $N$, and letting $N\to\infty$ gives the desired result.  \end{proof}

\section{Implications to Extremum Seeking Control}
\label{s:ESC}

\subsection{What is extremum seeking control?}

	While much of ESC theory concerns tracking the minimizer of a time-varying objective  (that is,  $\Obj$ depends on both the parameter $\theta$ and time $t$),   it is simplest to first explain the ideas in the context of global optimization of the static objective $\Obj \colon\Re^d\to\Re$.

We begin with an explanation of the appearance of $- M  \tilnabla_t\Obj $ in Fig~\ref{fig:ESC}.     The low pass filter with output  $\ODEstate_t$ is designed so that the derivative $\ddt \ODEstate_t$ is small in magnitude, to justify a quasi-static analysis.   
	An example is
	\begin{equation}
		\ddt \ODEstate_t =   - \sigma [ \ODEstate_t  -\theta^\ctr ]  +   \alpha   U_t  \,,   \qquad U_t = -  M \tilnabla_t\Obj 
		\label{e:ESC_LP}
\end{equation}
with parameters satisfying $0< \sigma < \alpha$. 

Motivation for the high pass filter   is provided using two extremes.  In each case it is assumed that the probing gain is constant:   $\upepsilon_t \equiv \epsy$, independent of time.  
	
	\wham{Pure differentiation:}   In this case, the figure is interpreted as follows
	\begin{equation}
		M \tilnabla_t\Obj  =  \bigl(\ddt \qsaprobe_t  \bigr) \bigl(   \tfrac{1}{\epsy}  \ddt  \Obj(\ODEstate_t + \epsy \qsaprobe_t)  \Bigr)
		\label{e:ESC1}
	\end{equation}
	Adopting the notation from the figure $  \cqsaprobe_t = \ddt \qsaprobe_t  $,  we obtain by the chain rule 
	\begin{equation}
		M \tilnabla_t\Obj  =   \cqsaprobe_t {\cqsaprobe_t}^\transpose \nabla \Obj\, (\ODEstate_t + \epsy \qsaprobe_t)   + \clW_t
		\label{e:ESC1explained}
	\end{equation}
	where $\clW_t =    \cqsaprobe_t     \nabla^\transpose \Obj\, (\ODEstate_t + \epsy \qsaprobe_t) \ddt  \ODEstate_t  $ is small by design of the low pass filter---consider \eqref{e:ESC_LP}, with $\alpha>0$  small.
	
	This justifies the diagram with $M_t= \cqsaprobe_t {\cqsaprobe_t}^\transpose $  time varying, but it is the expectation of $M_t$ that is most important in analysis.

	The pair of equations (\ref{e:ESC_LP}, \ref{e:ESC1}) is an instance of the QSA ODE \eqref{e:QSA_ODE_gen}, with $2d$-dimensional probing signal $( \qsaprobe_t,\cqsaprobe_t)$.

\wham{All pass:} This is the special case in which the high pass filter is removed entirely, giving 
	\begin{equation}
		M \tilnabla_t\Obj  =    \qsaprobe_t  \tfrac{1}{\epsy} \Obj\, (\ODEstate_t + \epsy \qsaprobe_t)    
		\label{e:ESC-0}
\end{equation}
This takes us back to the beginning:  based on \eqref{e:1qSGD} we conclude that 1qSGD is the ESC algorithm in which $\sigma=0$ in the low pass filter, and the high pass filter is unity gain.

In \Cref{s:ESCbig} we explain how the more general ESC ODE can be cast as QSA.

\subsection{State dependent probing}
\label{s:SDprobe}

\begin{subequations}

It is time to explain the time varying probing gain, denoted $\upepsilon_t $ in  \Cref{fig:ESC}.    This will be state dependent for two important reasons:
\begin{romannum}
\item   The vector fields for ESC and the special case 1qSGD are \textit{not Lipschitz continuous} unless the objective is Lipschitz.    For 1qSGD,  the vector field includes  $    \clYn_t   = \frac{1}{\epsy} \Obj(\ODEstate_t + \epsy \qsaprobe_t)$,  and we show below that every ESC algorithm admits a state space representation driven by the same signal.

\item  If the observed cost $\Obj(\clZ_t)$ is large, then it makes sense to increase the exploration gain to move more quickly to a more desirable region of the parameter space.   
\end{romannum}
 Two choices for $\upepsilon_t \equiv \upepsilon(\ODEstate_t)$ are proposed here:
 \begin{align}
 \upepsilon(\theta) &= \epsy \sqrt{1 + \Obj(\theta) - \Obj^-} 
 \label{e:upepsy1}
\\
 \upepsilon(\theta) &= \epsy\sqrt{1+ \| \theta - \theta^\ctr\|^2 / \upsigma_p^2 }
\label{e:upepsy2}
\end{align}
where in  \eqref{e:upepsy1} the constant  $ \Obj^- $  chosen so that $ \Obj(\theta) \ge \Obj^-$ for all $\theta$.
In the second option,   $\theta^\ctr$ is interpreted as an a-priori estimate of $\theta^\opt$,  and $\upsigma_p$ plays the role of standard deviation around this prior.   
 \label{e:upepsy}
 \end{subequations}

The first is most intuitive, since it directly addresses (ii):  the exploration gain $\upepsilon_t$  is large when $ \Obj(\ODEstate_t) $ is far from its optimal value.   However, it does not lead to an online algorithm since $\Obj(\ODEstate_t) $ is not observed.   In a discrete-time implementation we would use the online version:
\[
 \upepsilon_{t_n}  = \epsy \sqrt{1 + \Obj(\clY_{t_{n-1}} ) - \Obj^-}
\]

In cases  \eqref{e:upepsy1} or  \eqref{e:upepsy2} we introduce the new notation,
\begin{equation}
    \clYn( \theta,\qsaprobe)    = \frac{1}{\upepsilon} \Obj(\theta + \upepsilon \qsaprobe)\,, 
        \label{e:normalizedObservations}
\end{equation}
with the usual shorthand  $   \clYn_t  = \clYn(\ODEstate_t,\qsaprobe_t) $.  
The use of \eqref{e:upepsy2} leads to a Lipschitz vector field for ESC:
 \begin{proposition}
 \label[proposition]{t:Lip_qSGD} 
 The function $   \clYn$ defined in \eqref{e:normalizedObservations}
is uniformly Lipschitz continuous in $\theta$ under either of the following assumptions on $\upepsilon = \upepsilon(\theta)$,  and with   objective $\Obj$ whose gradient is uniformly Lipschitz continuous:  
 \begin{romannum}
\item 
 $\upepsilon$ is defined by \eqref{e:upepsy2}.
 
\item 
$\upepsilon$ is defined by \eqref{e:upepsy1},   
and
there is $\delta>0$ such that  $\|\nabla \Obj (\theta)\|  \ge \delta \| \theta\| $ whenever $\|\theta\| \ge \delta^{-1}$.     
 \end{romannum}
Moreover, under either (i) or (ii), the following approximation holds,
\begin{equation}
	\clYn( \theta,\qsaprobe)    =    \frac{1}{\upepsilon(\theta)}   \Obj(\theta)   +   \qsaprobe^\transpose \nabla\, \Obj(\theta)   +  O(\upepsilon)     \, ,
	\label{e:clYnTS}
\end{equation}
where the error term $ O(\upepsilon)$ is bounded by a fixed constant times $\upepsilon(\theta)$.  
 \end{proposition}
 
 \begin{proof}
 The proof follows from a variation on \eqref{e:1qSPSA_FTC} and the definition \eqref{e:normalizedObservations}:
\[
    \clYn( \theta,\qsaprobe)    = 
\qsaprobe^\transpose   \int_{0}^1   \nabla \Obj\,( \theta + t  \upepsilon(\theta)\qsaprobe)  \, dt   
						 +				 
    \frac{1}{\upepsilon(\theta)} \Obj(\theta )  
 \]
 The integrand in the first term is Lipschitz in $\theta$,  since the composition of Lipschitz functions is Lipschitz.  The second term is Lipschitz since its gradient is bounded under the given assumptions.
\end{proof}

The state-dependent probing gain ensures that the the Lipschitz condition (A1) is satisfied, which is essential to  global stability theory.  An example of divergence using a fixed probing gain is given next.

\wham{Finite escape time for ESC}
The  Lipschitz conditions in (A1) cannot be relaxed in the global stability theory for QSA.   
	This is why establishing global stability of ESC is challenging when $\Obj$ is not Lipschitz continuous.

	The  ESC-0 ODE is recalled here:    
	\begin{subequations}
		\begin{equation*}
			\begin{aligned}
				\dot{\ODEstate}_t &= -\alpha\frac{1}{\epsy} \qsaprobe_t \Obj(\Uppsi)  \,, 
				\quad
				\Uppsi_t = \ODEstate_t + \epsy \qsaprobe_t  \, .
			\end{aligned}
		\end{equation*}
		Consider the scalar ODE with quadratic objective  $\Obj(\theta) = \theta^2$
		and probing signal  $\qsaprobe_t = \cos(\omega_0 t)$.
		For this simple example, we obtain 
		\[
		\begin{aligned}
			\ddt \Uppsi &= -\frac{\alpha}{\epsy} \qsaprobe_t \Obj(\Uppsi) + \epsy \ddt {\qsaprobe}_t \\
			&= -\frac{\alpha}{\epsy} \cos(\omega_0 t) \Uppsi^2_t - \epsy \omega_0 \sin(\omega_0 t) \,.
		\end{aligned}
		\]
		\textit{This ODE has finite escape time when $\Uppsi_0<0$ and $|\Uppsi_0|$ is sufficiently large.   }
		
		To justify this claim, we   bound  $\{ \Uppsi_t : 0\le t <  t_\diamond\}$ with
		\[
		t_\diamond =2\epsy  \frac{1}{\alpha } \frac{1}{|\Uppsi_0|} \, .
		\]
		Assume that $\epsy|\Uppsi_0|^{-1}$ is sufficiently small so that $ \cos(\omega_0 t ) \ge \half$ for $0\le t\le  t_\diamond $.
		This implies   the lower bound $-\ddt \Uppsi_t \geq   \alpha  \Uppsi^2_t /(2\epsy)$, and   hence
		\[
		\ddt \frac{1}{\Uppsi_t}  =  - \frac{1}{\Uppsi^2_t}  \bigl(\ddt \Uppsi_t  \bigr)\geq     \frac{\alpha}{2\epsy}  \,   ,\quad \textit{ for $t<t_\diamond$.}
		\]
		Integrating both sides from $0$ to any value $T <t_\diamond$ gives
		\[
		\begin{aligned}
			\frac{1}{\Uppsi_T}  - & \frac{1}{\Uppsi_0} \geq \frac{\alpha}{2\epsy} T  
			\quad
			\Longrightarrow
			\quad
			\Uppsi_T \leq 	\Big( \frac{1}{\Uppsi_0}	 + \frac{\alpha}{2\epsy} T \Big)^{-1}		 \, .
		\end{aligned}
		\]
		In conclusion,    for a value $t_\bullet\in (0, t_\diamond)$,
		the solution $\{\Uppsi_t : 0\le t<t_\bullet \}$ is continuous and decreasing, with
		\[
		\lim_{T \uparrow t_\bullet}   \Uppsi_T   = - \infty \, .
		\]
		
		Global stability is ensured if the probing gain is state-dependent---either of the choices in \eqref{e:upepsy}
		ensure success. 
	\end{subequations}

\subsection{ESC and QSA}
\label{s:ESCbig}

We demonstrate here that the general ESC ODE can be interpreted as QSA,    driven by the observation process   $   \clYn_t  = \clYn(\ODEstate_t,\qsaprobe_t) $.  We find that a more informative interpretation of ESC is described as a \textit{two time-scale} variant of the QSA ODE.

For simplicity we opt for the first order low pass filter,
\begin{equation}
\ddt \ODEstate_t  = -\sigma \ODEstate_t  -   \alpha  \tilnabla_t\Obj\,,\qquad\tilnabla_t\Obj  =  \cqsaprobe_t \chclYn_t
\label{e:LPss}
\end{equation}
 The high pass filter is a general linear filter of order $q\ge 1$,   with state space realization defined by matrices  $(F,G,H,J)$  of compatible dimension.  For a scalar input $u_t$,  with the output of the high pass filter denoted $y_t$,  and the $q$-dimensional state  at time $t$ is denoted $Z_t$, we then have by definition,  
\begin{equation}
\begin{aligned}
\ddt Z_t   &=  F Z_t + G u_t
   \\
     y_t        & =   H^\transpose  Z_t + J u_t
\end{aligned} 
\label{e:HPss}
\end{equation}
In the architecture \Cref{fig:ESC} the high pass filter is used for  $d+1$ different choices of input and zero initial conditions are assumed: 
when $u_t=\clYn_t$, the output is $y_t = \chclYn_t$,   and the $i$th component   $\cqsaprobe_t^i$  of the filtered probing signal is the output with $u_t = \qsaprobe_t^i$.

\begin{proposition}
\label[proposition]{t:ESC=QSA}
The ESC ODE obtained using the low pass filter \eqref{e:LPss} and high pass filter \eqref{e:HPss} is itself a state space model with $(d+q)$-dimensional state $X_t = [\ODEstate_t^\transpose,  Z_t^\transpose]^\transpose$.  The nonlinear dynamics can be expressed 
	\begin{equation}
	\begin{aligned}
		\ddt X_t  &=  
		\alpha  \begin{bmatrix}    
			- \frac{\sigma}{\alpha}  I   &  -     \cqsaprobe_t \rH ^\transpose
			\\
			0   & \frac{1}{\alpha}   \rF   
		\end{bmatrix}  X_t   
		+
		\alpha \begin{bmatrix}    
			-  \rJ  \cqsaprobe_t   
			\\
			\frac{1}{\alpha}\rG 
		\end{bmatrix}   \clYn_t  \, ,
	\end{aligned} 
	\label{e:ESC=QSA}
\end{equation}
It is a controlled nonlinear state space model with input $(\qsaprobe, \cqsaprobe)$ for any choice of $\upepsilon$ in  \eqref{e:upepsy}. 
\end{proposition}

\Cref{t:PMF} can be freely applied to the state space representation \eqref{e:ESC=QSA} because the theorem makes no assumptions on the magnitude of $\alpha$, or even stability of the QSA ODE.

Recall that $\alpha$ is a fixed constant, so the fact that $f$ depends on this gain is irrelevant in the definition for the QSA vector field,
\begin{equation}
	f(x,\cqsaprobe, \qsaprobe) =    
	\begin{bmatrix}    
		- \frac{\sigma}{\alpha}  I   &  -     \cqsaprobe \rH ^\transpose
		\\
		0   & \frac{1}{\alpha}   \rF   
	\end{bmatrix}  x
	+
	\begin{bmatrix}    
		-  \rJ  \cqsaprobe   
		\\
		\frac{1}{\alpha}\rG 
	\end{bmatrix}   \clYn (\theta,\qsaprobe) \, .
	\label{e:ESCvectorField}
\end{equation}
Three solutions to Poisson's equation are required to write down the P-mean flow:
\whamrm{(i)}  The solution 
$\haclYn$ with forcing function $\clYn$.
\whamrm{(ii)} $ \hacqsaprobe$   with forcing function $ \cqsaprobe$ (similar to $\haG$ in \eqref{e:ESCapprox_haf}).  
\whamrm{(iii)}    $ \widehat{Q}$   with forcing  function $Q(\theta,\Phi) = -  \rJ  \cqsaprobe\clYn (\theta,\qsaprobe)$.

\begin{subequations}

	\begin{theorem}
		\label{t:PMFESC}
		If $\Obj$ is analytic, the  P-mean flow representation holds,
		\begin{equation} 
			\ddt X_t    =  \alpha[ \barf (X_t)    - \alpha  \barUpupsilon_t  +   \clW_t] \,,
			\label{e:ESC=QSA-PMF}
		\end{equation}
		in which for any $x=(\theta;\varsigma)$ and $z\in\prstate$,
		\begin{equation}
			\barf(x) =    
			\begin{bmatrix}    
				- \frac{\sigma}{\alpha}  I   & 0
				\\
				0   & \frac{1}{\alpha}   \rF   
			\end{bmatrix}  x
			+
			\begin{bmatrix}    
				-  \rJ  \Expect[  \cqsaprobe (\Phi)  \clYn (\theta,\qsaprobe]
				\\
				\frac{1}{\alpha}\rG \Expect[  \clYn (\theta,\qsaprobe)  ]
			\end{bmatrix}    \, ,
			\label{e:SuperBarf}
		\end{equation}
		with  expectations in steady-state.
		The functions   $\haf$    and $\barUpupsilon$ admit the representations,  
		\begin{align}
			\haf(x,z) & =    
			\begin{bmatrix}    
				0	   &  -     \hacqsaprobe(z )\rH ^\transpose
				\\
				0   & 0   
			\end{bmatrix}  x
			+
			\begin{bmatrix}    
				-  \rJ  \widehat{Q}(\theta,z)
				\\
				\frac{1}{\alpha}\rG \haclYn (\theta,z)
			\end{bmatrix}   \,  ,
			\label{e:BigESCfish}
			\\
			\Upupsilon(x,z) & = \frac{1}{\alpha} 
			\begin{bmatrix} \hacqsaprobe(z) \rH^\transpose \{ \rF x + \rG\clYn(\theta,\qsaprobe(z) )\}
				\\
				0
			\end{bmatrix} \, .
			\label{e:UpOopsESC}
		\end{align}
	\end{theorem}
\end{subequations}

\begin{proof}
 The expression  \eqref{e:BigESCfish} follows directly from \eqref{e:ESCvectorField}.  There is simplification because terms not involving  $   \qsaprobe$ or $\cqsaprobe$ vanish.  The formula \eqref{e:UpOopsESC} then follows from the definition   $\Upupsilon(x,z)  = -\partial_x\haf(x,z) f(x,z)   $.

\end{proof}

The interpretation of the P-mean field representation is entirely different here because $\Upupsilon$ is no longer a nuisance term, but a critical part of the dynamics.
Applications to design remains a topic for future research.

\wham{Interpretation as two time-scale QSA:}
The ODE \eqref{e:ESC=QSA} falls into this category,  where   $(Z_t, \Phi_t)$ represents the fast state variables. The basic principle is that we can approximate 
the solution to the $(d+q)$-dimensional state space model through the following quasi-static analysis.    For a given time $t$,   let $\{  \barZ_r  : r\ge t\}$ denote the solution to the state space model defining $\bfmZ$ with $\ODEstate_r \equiv \theta$ for all $-\infty < r<\infty$:
\[
 \barZ_r (\theta)  = \int_{-\infty}^r  e^{F (r-\tau)} G Y_\tau\, d\tau\,, \quad 
 					\textit{with} \ \  Y_\tau =  \frac{1}{\upepsilon(\theta) }   \clY(\theta + \upepsilon(\theta) \qsaprobe_\tau )   
 \]
The next step is to substitute the solution to obtain the approximate dynamics,
\begin{equation}
\ddt \ODEstate_t  \approx       
			-\sigma \ODEstate_t     - \alpha     \bigl[ \cqsaprobe_t H^\transpose \barZ_t(\ODEstate_t)  			   +   J \cqsaprobe_t    \clYn_t \bigr]
\label{e:QSAtwo-time}
\end{equation}
A more useful approximation is obtained on applying \Cref{t:Lip_qSGD}:
\[
\begin{aligned}
	\clYn_t
	& \approx  \frac{1}{\upepsilon(\theta)}   \Obj(\ODEstate_t)   +   \qsaprobe_t^\transpose \nabla \Obj\, (\ODEstate_t) \, ,
	\\
	\rH ^\transpose \barZ(\theta,  \Phi_t)  
	&\approx  \rH ^\transpose  \int_{-\infty}^t  e^{\rF  (t-\tau)} \rG \bigl\{  
	\frac{1}{\upepsilon(\theta)}   \Obj(\theta)   +   \qsaprobe_\tau^\transpose \nabla\, \Obj(\theta)  \bigr\}\, d\tau   
	\\
	& =h_0  \frac{1}{\upepsilon(\theta)}  \Obj(\theta)   +    \cqsaprobe_t^\transpose \nabla\, \Obj(\theta) \,  ,
\end{aligned} 
\]
where $h_0 = -\rH ^\transpose \rF^{-1} \rG$ is the DC gain of the high pass filter.

Substitution in \eqref{e:QSAtwo-time} gives
\[
\begin{aligned} 
	\ddt \ODEstate_t  \approx       
	-\sigma \ODEstate_t   &  - \alpha M_t \nabla  \Obj(\ODEstate_t) 
	- \alpha \cqsaprobe_t    \frac{1}{\upepsilon_t} (h_0+\rJ) \Obj(\ODEstate_t)   
	+O(\alpha \upepsilon)  \,,
	\\
	&\textit{with}\quad M_t = \cqsaprobe_t   [\cqsaprobe_t  + \rJ \qsaprobe_t ]^\transpose  \, .
\end{aligned} 
\]

Once tight error bounds in these approximations are established,  design and analysis proceeds based on the $d$-dimensional QSA approximation.   
In particular, without approximation this QSA ODE is stable provided the high pass filter is \textit{passive}, such as a lead compensator.   Passivity combined with the positivity of $\Sigmaqsa$ implies that $M+M^\transpose>0$ with
$M = \Expect_\uppi  [ M_t]$.

Theory for two time-scale QSA is not yet available to justify this approximation.   
The extension of stochastic theory is a topic of current research.

\subsection{Perturbative mean flow for qSGD}

\Cref{t:bigQSA+ESC} below is a summary of the conclusions of  \Cref{t:PMF} and  \Cref{t:Couple_theta} for qSGD ODEs.

Additional local stability structure is obtained in terms of Lyapunov exponents \eqref{e:LyapExp}, similarly to what is done in \Cref{t:ODEstate_is_bdd}. It will be seen in \Cref{t:bigQSA+ESC}  that a negative Lyapunov exponent exits for 2qSGD subject to conditions on the objective function and the gains.


  In \eqref{e:LTIESClin} the  time $t_0$ is chosen so that $\| \ODEstate_t - \theta^*\| \le B \alpha $ for some $B$, and all $t\ge t_0$.   

\begin{subequations}

\begin{theorem}[QSA Theory for ESC]
\label[theorem]{t:bigQSA+ESC}
Consider an objective function $\Obj$ satisfying the following:  
\begin{romannum}
\item    $\Obj$ is analytic with Lipschitz gradient satisfying  $\|  \nabla\Obj (\theta)\| \ge \delta\|\theta\| $  for some $\delta>0$ and all $\theta$ satisfying $\|\theta\|\ge \delta^{-1}$.   
   
\item  The objective has a  unique minimizer $\theta^\opt$,  and it is the only solution to $\nabla\Obj (\theta) =\Zero$.

\item  $P = \nabla^2\Obj\, (\theta^\opt)$ positive definite.

\end{romannum} 

Consider 1qSGD or 2qSGD with constant gain $\alpha$.  Suppose that the probing signal is chosen of the form 		\eqref{e:qSGD_probe} with $K=d$,    frequencies satisfying \eqref{e:logFreq}, and
   $\Sigmaqsa>0$.

Then, for either algorithm, there are positive constants  $\epsy^0$ and $\alpha^0$ such that the following approximations are valid for $0<\epsy\le \epsy^0$,  $0<\alpha\le \alpha^0$:     
\begin{align}
   \ddt \ODEstate_t   &=    \alpha \bigl[ -  \Sigmaqsa \nabla\Obj\, (\ODEstate_t)  + \clW_ t  + O(  \epsy^2)  \bigr]  \,,  &&  t\ge 0
\label{e:LTIESC}
\\
   &=    \alpha \bigl[ - \Sigmaqsa P  [\ODEstate_t  -\theta^\opt]  + \clW_ t  + O(\alpha^2 +   \epsy^2) +o(1)  \bigr]  \,, && t\ge t_0
\label{e:LTIESClin}
\end{align} 
Moreover, the
 following hold for $0<\epsy\le \epsy^0$,  $0<\alpha\le \alpha^0$:
\begin{romannum}
\item   $\|  \ODEstate_t -\theta^\opt \| \le O( \alpha+\epsy^2) + o(1)$  

\item   $\|  \ODEstateF_t -\theta^\opt \| \le O( \alpha^2+\epsy^2) + o(1)$,  with filtered estimate obtained using the criteria of \Cref{t:Couple_theta}.
\end{romannum}
In addition,  the Lyapunov exponent \eqref{e:LyapExp} is negative in the special case of   2qSGD.
\end{theorem}
 
\end{subequations}
\begin{proof}
The analytic assumption is imposed so that (A3) holds, thanks to  \Cref{t:hah}. 

\Cref{t:Lip_implies_V4}  tells us that  the QSA ODE is $\alpha^0$   ultimately bounded  for some $\alpha^0>0$.   
The estimation error bounds in (i) and (ii) are then immediate from   \Cref{t:Couple_theta}.

Justification of \eqref{e:LTIESC} follows from \Cref{t:PMF} and the approximation $  
\barf( \theta)   = - \Sigmaqsa\nabla\, \Obj(\theta)   +  O(\epsy^2)   $ obtained for 1qSGD in \Cref{t:ESCapproximation}    and for 2qSGD in \Cref{t:ESCapproximation2}.    
 The term $\barUpupsilon_t$ appearing in \eqref{e:BigGlobalODE} is zero on application of \Cref{t:varpi}~(iii).

The approximation \eqref{e:LTIESClin} then follows from  \eqref{e:LTIESC}  and the Taylor series approximation
\[
\nabla\Obj\, (\ODEstate_t)     =   \nabla\Obj\, ( \theta^\opt)   + P  [\ODEstate_t  -\theta^\opt]   +   O (\| \ODEstate_t  -\theta^\opt \|^2)
\]  
The first term above is zero by assumption, $   \nabla\Obj\, ( \theta^\opt)   =\Zero$,   and $ \| \ODEstate_t  -\theta^\opt \|  \le \| \ODEstate_t  -\theta^\ocp \|  +  \| \theta^\ocp  -\theta^\opt \|  = O(\alpha + \epsy^2)$ for   $t\ge t_0$,  giving $ \| \ODEstate_t  -\theta^\opt \|^2  \le  O(\alpha^2 + \epsy^4)$  (recall that $\| \ODEstate_t - \theta^*\| \le B \alpha $ for $t\ge t_0$).
These approximations imply \eqref{e:LTIESClin}.

The proof that $\Lambda_{\ODEstate} <0$ for 2qSGD  follows from identifying the linearization matrix used in the LTI approximation \eqref{e:BigGlobalODElin}.   The following is established in  \Cref{t:ESCapproximation2}:  
\[
A( \theta,\qsaprobe)    =   A^0( \qsaprobe) +  O(\epsy)  \,,\qquad \textit{with }   A^0( \qsaprobe)  =
 - \qsaprobe \qsaprobe^\transpose \nabla^2\, \Obj(\theta)   
\]
in which the error $O(\epsy)$ depends smoothly on $\theta$.       
Hence  the dynamics  \eqref{e:SensitivityODE} for the 
 sensitivity process can be expressed
\begin{equation}
\ddt \clS_t^0  = -  \alpha  [ (\alpha+ \epsy^2) M_t +  \qsaprobe_t  \qsaprobe^\transpose_t P  ] \clS_t^0\, ,
						 \qquad t\ge t_0\,,
\label{e:SensitivityODE2aSGD}
\end{equation}
where  the norm of the matrix process $\{ M_t : t\ge t_0 \}$ is bounded by a constant that is
 independent of $\epsy$ or $\alpha$,  though $t_0$ is dependent on these parameters.
 
 The remainder of the proof is then identical to the proof of  \Cref{t:QPL}.     
\end{proof}

Approximations in \Cref{s:ESCtheory} suggest that a negative Lyapunov exponent is not likely for any variant of 1qSGD.    We obtain in this case,
\[
A( \theta,\qsaprobe)   =  
 -  \partial_\theta\Bigl( \frac{1}{\upepsilon(\theta)}   \Obj(\theta) \Bigr) \qsaprobe
 - \qsaprobe \qsaprobe^\transpose \nabla^2\, \Obj(\theta)   +  O(\epsy)  
\]
If $\upepsilon =\epsy$, independent of $\theta$,  this simplifies to 
\[
A( \theta,\qsaprobe)   =  
 -   \Bigl( \frac{1}{ \epsy}    \qsaprobe   \nabla^\transpose  \Obj\, (\theta) 
+ \qsaprobe \qsaprobe^\transpose \nabla^2\, \Obj(\theta)  \Bigr)  +  O(\epsy)  
\]
It is likely that the Lyapunov exponent is negative for small $\epsy>0$ and  $\alpha\ll \epsy$, but this is not an interesting setting.

\section{Examples}
\label{s:Experiments}

In each of the following experiments\footnote{Publicly available code for experiments were obtained under a GNU General Public License v2.0.}, the QSA ODEs were implemented using an Euler approximation with time discretization of $1~$sec. The probing signal respected \eqref{e:qSGD_probe} with $\{\omega_i\}$ irrationally related.


\subsection{Optimization of Rastrigin's Objective}
\label{s:Opt_Rast_NIPS}
We illustrate the fast rates of convergence of QSA with vanishing gain by surveying results from an experiment in \cite{laumey22,laumey22e}. 
\paragraph{Simulation Setup}
%
The $1$qSGD algorithm \eqref{e:1qSGD} and its stochastic counterpart \eqref{e:1qSPSA} were used for the purpose of optimizing the Rastrigin Objective in $\Re^2$
\begin{equation}
\Obj(\theta) = 20 +\sum _{i=1}^2\left[\theta_{i}^{2}-10\cos(2\pi \theta_{i})\right]  
\label{e:Rastrigin}
\end{equation}  
The experiments were conducted before discovery of the Lipschitz variant of $1$qSGD so $\upepsilon(\theta) = \epsy=0.25$ and the sample paths of $\bfODEstate$ were projected onto $[-B,B]^2$ with $B=5.12$, which is the evaluation domain commonly used for this objective \cite{simulationlib}. The QSA ODE was implemented through an Euler scheme with time sampling of $1~$sec with $a_t = \min\{0.5,(t+1)^{-0.85}\}$. $M=200$ independent experiments were run with $\ODEstate^m_0$ uniformly sampled from $[-B,B]^2$ for $\{1 \leq m \leq M\}$. For each experiment, the probing signal was selected of the form
\begin{equation}
	\qsaprobe^m_t = 2 [\sin(t/4+ \phi^m), \sin(t/e^2+ \phi^m)]^\transpose
	\label{e:exp_qsaprobe}
\end{equation}
where for each $m$, $\phi^m$ was uniformly sampled from $[-\pi/2, \pi/2]$. 
The noise for stochastic algorithm was a scaled and shifted Bernoulli with $p=1/2$ and had support in the set $\{-\sqrt{2}, \sqrt{2}\}$. The values $\pm \sqrt{2}$ were chosen so that the covariance matrices in both algorithms were equal.

In order to test theory for rates of convergence with acceleration through PR averaging, the empirical covariance was computed for $1$qSGD across all initial conditions,
\begin{equation}
	\begin{aligned}
	\barSigma_T & = \frac{ 1}{M} \sum_{i=1}^{M} \ODEstate_T^i  {\ODEstate_T^i }^\transpose   -  \barODEstate_T \barODEstate_T^\transpose\,,\qquad \barODEstate_T= \frac{ 1}{M} \sum_{i=1}^{M} \ODEstate_T^i \\ 
		\hat{\sigma}_T & = \sqrt{\tr(\barSigma_T)}
	\end{aligned}
	\label{e:empCov}
\end{equation}
In each experiment the trajectories $\{\ODEstate_t\}$ were accelerated through PR averaging with $\kappa =5$. If the scaled empirical covariance $a_T^{-2} \hat{\sigma}_T$ is bounded in $T$, we have $\|\ODEstatePR_T -\theta^* \| = O(T^{-2\rho})$.

\begin{figure}[h]
	\includegraphics[width=1\hsize]{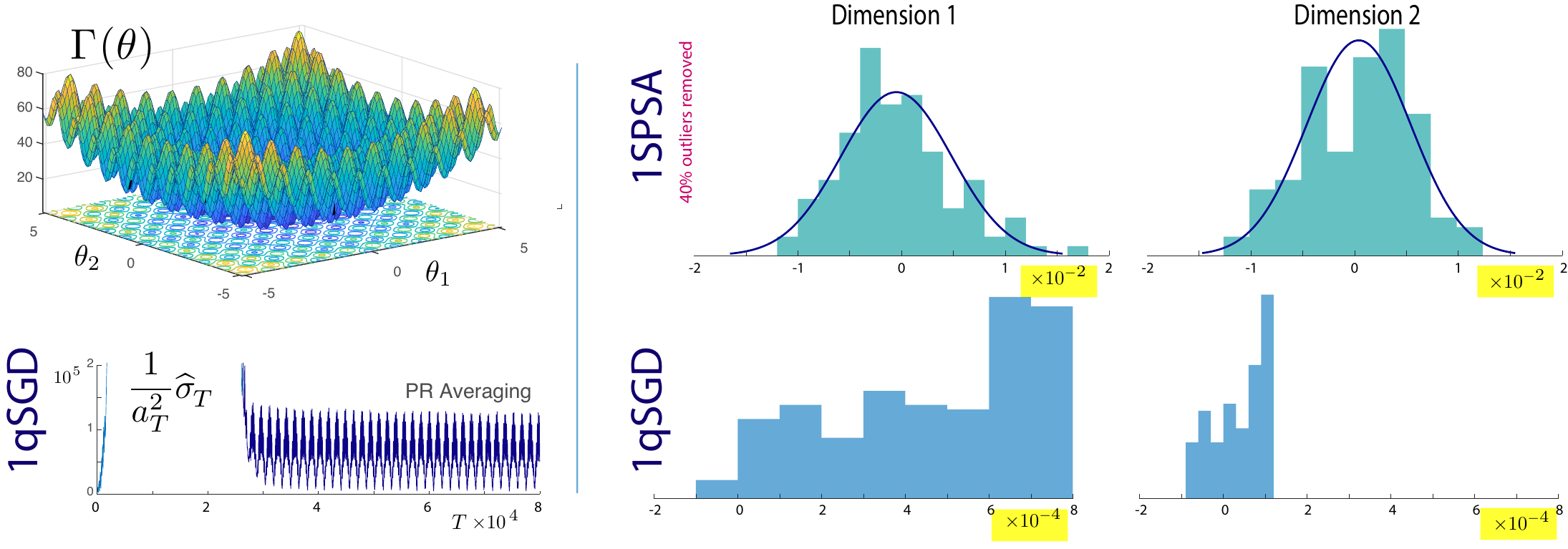}
	\caption{Rastrigin objective (top left), scaled empirical covariance (bottom left), histograms of estimation error for $1$SPSA with PR averaging (top middle and top right), histograms of estimation error for $1$qSGD with PR averaging (bottom middle and bottom right).}\label{fig:Rastrigin_results}
\end{figure} 

\paragraph{Results}
\Cref{fig:Rastrigin_results} shows plots for the scaled empirical covariance for $1$qSGD as well as histograms for the estimation error $\ODEstatePR_T - \theta^*$ for both algorithms. We see in \Cref{fig:Rastrigin_results} that the convergence rates of $O(T^{-2\rho})$ are achieved by both filtering techniques. 

The variance of the estimation error for the deterministic algorithm is much smaller than for its stochastic counterpart: \Cref{fig:Rastrigin_results} shows that the reduction is roughly two orders of magnitude.

Outliers were identified with Matlab's \textit{isoutlier} function and removed from the histograms in \Cref{fig:Rastrigin_results}. Around $40\%$ of the estimates were considered outliers for the stochastic algorithm, while none was observed for its deterministic counterpart. 


Recall that bias is inherent in both 1qSGD and 1SPSA when using a non-vanishing probing gain.
It would appear from \Cref{fig:Rastrigin_results} that this bias is larger for 1qSGD  [consider the histogram for estimates of $\theta_1^\ocp=0$, which shows that final estimates typically exceed this value].    In fact, there is no theory  to predict if 1qSGD is better or worse than 1SPSA in terms of bias.    The bias is imperceptible in the stochastic case due to the higher variance combined with the removal of outliers.  


\subsection{Vanishing vs Fixed Gain Algorithms -  Rastrigin's Objective}
\label{s:vanishing_vs_fixed_gain}
Experiments were performed to test   the performance of the Lipschitz version of $1$qSGD for both constant and vanishing gain algorithms.   Two take-aways from the numerical results surveyed below:
\begin{romannum}
	\item 
	The value of a second order filter can be substantial when using a fixed gain algorithm.
	Recall that this conclusion was not at all clear for either of the linear examples 	(\ref{e:lin_barY0}, \ref{e:lin_barYnot}),  even though in the first case the example satisfies the assumptions of \Cref{t:Couple_theta}.

	\item  
	However, in this example the flexibility of a vanishing gain algorithm is evident: 
	with high gain during the start of the run  (when $a_t\approx \alpha$) there is a great deal of exploration,  especially when $\ODEstate_t$ is far from the optimizer  due to the use of $\upepsilon(\theta)$.
	The vanishing gain combined with the use of distinct frequencies  then justifies the use of PR-averaging, which results in extremely low AAD in this example.
\end{romannum}

\paragraph{Simulation Setup}
For this objective, $M=5$ individual experiments were carried out in a similar fashion to the experiments summarized in \Cref{s:Opt_Rast_NIPS}. The Lipschitz variant of $1$qSGD was implemented using each of the initial conditions so projection of the sample paths $\{\ODEstate^m_t\} $ was not necessary. Each experiment used $\upepsilon(\theta) = \epsy \sqrt{1 + \|\theta \|^2}$ with $\epsy = 0.6$ and the same probing signal $\qsaprobe^m_t$ as in \eqref{e:exp_qsaprobe}. The following three choices of gain were tested for each initial condition:
\[
\text{(a) } a_t = 0.1 (t+1)^{-0.65}  
\quad \text{(b) } a_t = \alpha_b = 3 \times 10^{-3} 
\quad \text{(c) } a_t = \alpha_s = 7 \times 10^{-4}
\]
For the vanishing gain case (a) the process $\{\ODEstate^m_t\}$ was accelerated through PR averaging with $\kappa =5$, while for the fixed-gain case two filters of the form \eqref{e:filter_numerics} where used with $\zeta =0.8$ and $\gamma=\eta  \alpha$ for various values of $\eta \geq 1$. These experiments were repeated for $\ODEstate^m_0$ uniformly sampled from $10^{10}[-B,B]^2$ with $B =5.12$.

\paragraph{Results}
The top row of \Cref{fig:Rastrigin_comp} shows the evolution of $\{\ODEstate^m_t\}$ for each $m$ with $\eta=5$. The second row shows the evolution of $\{\Obj(\ODEstatePR_T),\Obj({\ODEstate}^{1 \text{\tiny\sf F}}_T),\Obj({\ODEstate}^{2 \text{\tiny\sf F}}_T),\Obj(\ODEstate_T)\}$ for the single path yielding the best performance for each gain choice across all 5 initial conditions.

We see in \Cref{fig:Rastrigin_comp} a clear advantage of vanishing gain algorithms: the algorithm explores much more in the beginning of the run, eventually deviating from $\theta^\opt$ by the bias inherent to the $1$qSGD algorithm.

For the runs that used $\alpha_b$ (case (b)), we have a good amount of exploration, but the steady state behavior is poor. Case (c) using the smaller value of $\alpha$ yielded better results.

\Cref{fig:Rastrigin_comp} shows the benefit of variance reduction from second order filter as opposed to a first order filter, based on runs that used $\alpha_s$. As opposed to the results in \Cref{fig:alpha2bias}, this example shows that we can't always obtain the best AAD with a first order filter. For these experiments, the filtered final estimation of $\theta^*$ yields comparable objective values to what is seen in the vanishing gain experiment. However, we do not always observe desirable exploration as in the vanishing gain case: we see several sample paths hovering around local minima of the objective rather than $\theta^\opt$.

Results for the trajectories with initial conditions of order $10^{10}$ were presented in \Cref{fig:ODEinfty}.

\begin{figure}[h]
	\includegraphics[width=1\hsize]{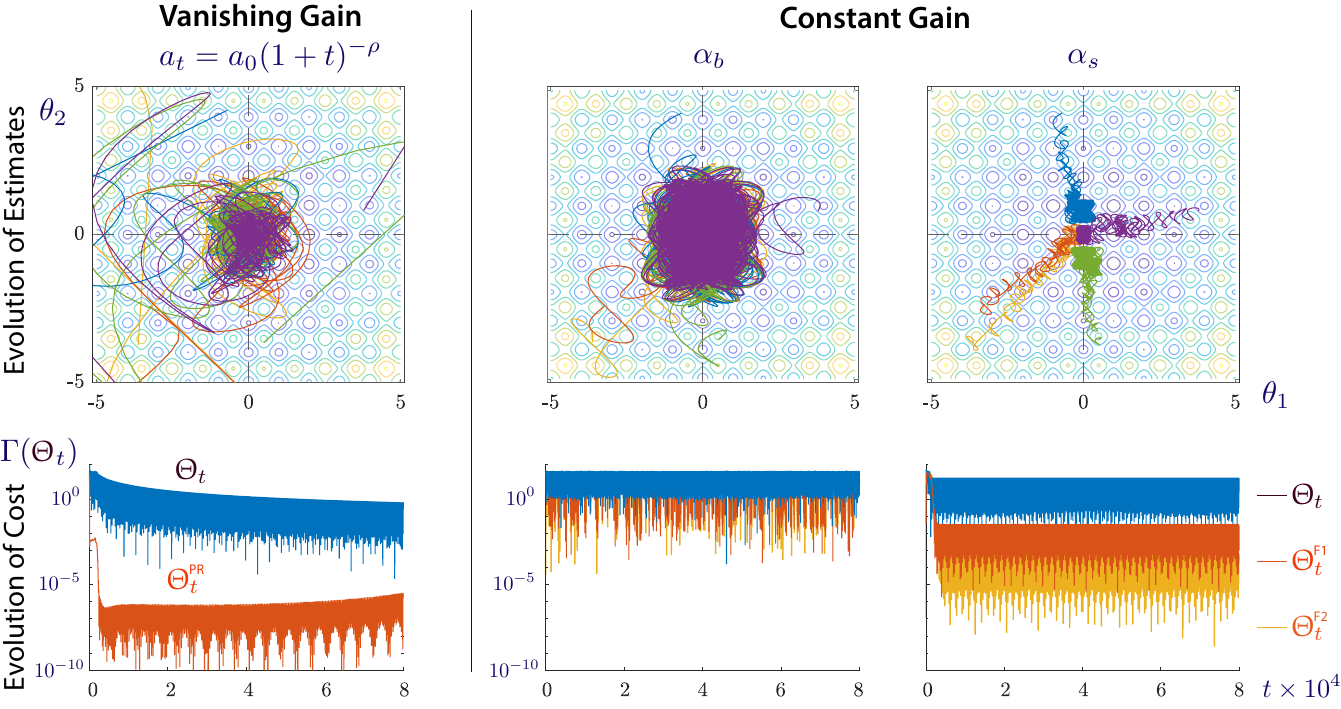}
	\caption{Comparison of the performance between vanishing and constant gain $1$qSGD for the Rastrigin Objective.}\label{fig:Rastrigin_comp}
\end{figure}

\subsection{The Walking Camel} 
\label{s:Camel_Walk}

The next numerical results are based on a tracking problem using the $1$qSGD algorithm. Although the results in this subsection will show that the algorithm successfully tracks the signal $\{\theta^\opt_t\}$, we are not advocating the use of $1$qSGD for tracking applications. The algorithm was employed to illustrate general algorithm design principles;
better performance is likely in algorithms with a carefully designed low pass filter such as \eqref{e:ESC_LP} with $\sigma>0$. Several experiments were conducted to address the following questions: 
\begin{romannum}
	\item 
	How does the frequency content of the probing signal affect performance?
	
	\item
	How does the rate of change of the signal $\{\theta^\opt_t\}$ impact tracking performance?
	
	\item  
	How does filtering reduce estimation error and variance in tracking?
	
\end{romannum}

\begin{figure}[h]
	\includegraphics[width=1\hsize]{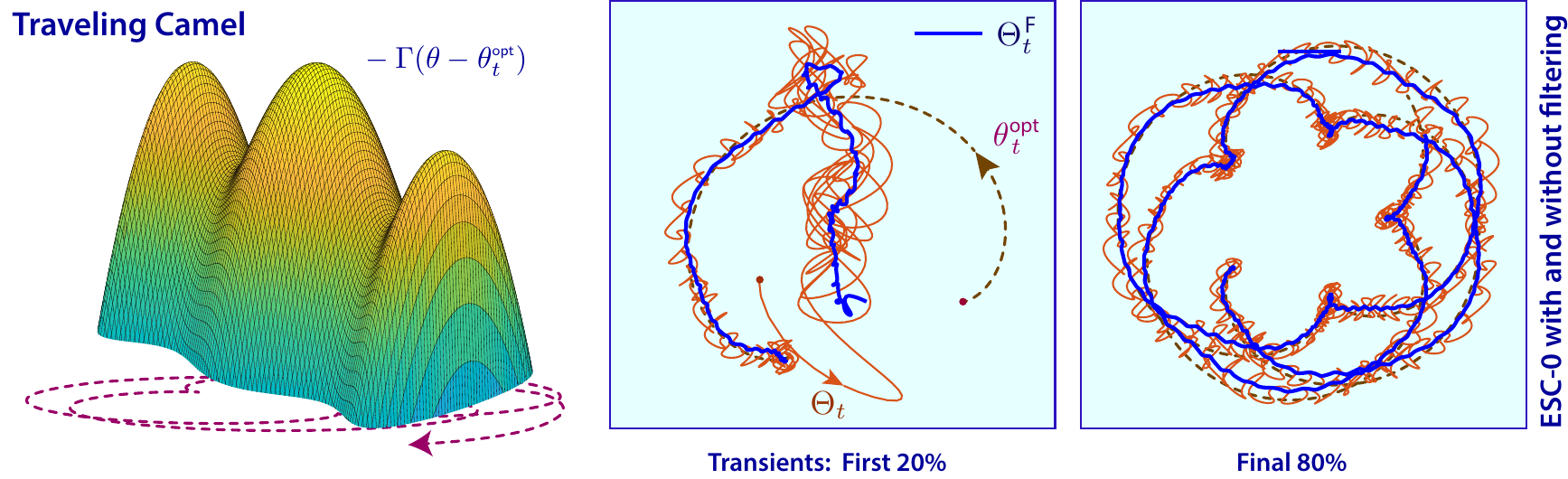}
	\caption{Tracking for the travelling camel with ESC-0}
\label{fig:moving_camel_lotus}
\end{figure} 

\paragraph{Simulation Setup}

The \textit{Three-Hump Camel} objective is the sixth-order polynomial function on $\Re^2$:
\begin{equation}
  \Obj(\theta) =2\theta_1^{2}     +\theta_1\theta_2+\theta_2^{2}
			+\tfrac{1}{6} \theta_1^6 -1.05\theta_1^{4}
\label{e:3Hump}
\end{equation}
A plot of the negative of this objective is shown on the left hand side of \Cref{fig:moving_camel_lotus}.   
Observe that the infimum of $\Obj$ over $\Re^2$ is not finite. 
The domain of this function is usually taken to be the bounded region $[-5,5]^2$   \cite{simulationlib}.

We considered a time varying version, of the form
\[
\Obj_t = \Obj(\theta - \theta^\opt_t)
\]
with the following choices of $\{\theta^\opt_t\}$:
\begin{equation}
	\begin{aligned}
		&\text{(a) }  \theta^\opt_t = \begin{bmatrix}
		 m \cos(\omega^\circ t) - h\cos(m\omega^\circ t/b^\circ) \\                                                                                      
		 m \sin(\omega^\circ t) -h\sin(m\omega^\circ t/b^\circ)
		\end{bmatrix}
	 \quad
		\text{where }
		 \omega^\circ = 2 \times 10^{-3}~\text{rad/sec} , b^\circ = 3/5, m =8/5, h = 0.6	\\
		&\text{(b) }\theta^\opt_t =  
		\begin{bmatrix}
			1\\1
		\end{bmatrix} 
	g_t, 
	\quad \text{where } 
	g_t=	\begin{cases} 
			\Delta(t/T_0),& t\leq T/2 \\
			\Pi(t/T_0) ,& t>T/2 
		\end{cases} \quad
		\text{and }  
		T_0=  \frac{c}{b} 
	~\text{sec with } b \in \{ 1, 3 , 5   \}, c = 5 \times 10^{-3} 
\end{aligned}
\label{e:theta_opt}
\end{equation}   
The choice (a) is an example of an epitrochoid path.  In (b) the notation $\Delta,\Pi$ is used to denote the unit triangle and square waves, respectively. The objective $\Obj_t$ was normalized so that $\theta^\opt = \Zero$.

Experiments were conducted for a single initial condition $\ODEstate_0$, chosen to be 
		a
 second (non-optimal) local minima of the objective, and the following probing signals:
\begin{equation}
\qsaprobe^a_t = 2 \begin{bmatrix}\sin(t/9)\\\sin(t/e^3) \end{bmatrix}, \quad \qsaprobe^b_t = 2 \begin{bmatrix} \sin(t/4) \\ \sin(t/e^2) \end{bmatrix},\quad \qsaprobe^c_t = 2 \begin{bmatrix} \sin(\sqrt{20}t/10) \\ \sin(\pi t/10)  \end{bmatrix}
\label{e:probingcamel}
\end{equation}
The assumptions of \Cref{t:Lip_qSGD}  are violated in this example, since the gradient of this objective is not Lipschitz continuous.     For this reason, the state-dependent gain  $\upepsilon$ was abandoned in favor of a constant gain in these experiments.
Choices for the parameters $\upepsilon(\theta) \equiv \epsy$ and $\alpha$ were problem specific: (a) $\epsy=0.2$ and $\alpha = 6 \times10^{-3}$, (b) $\epsy=0.5$ and $\alpha = 5\times10^{-3}$.

The process $\{\ODEstate_t\}$ was projected onto the set $[-5,5]^2$.
The filters in \eqref{e:filter_numerics} were used for variance reduction,
 with $\zeta =0.8$ and $\gamma= \eta \alpha$ for several $\eta \geq 1$.

\begin{figure}[h]
\centering
\includegraphics[width=0.85\textwidth]{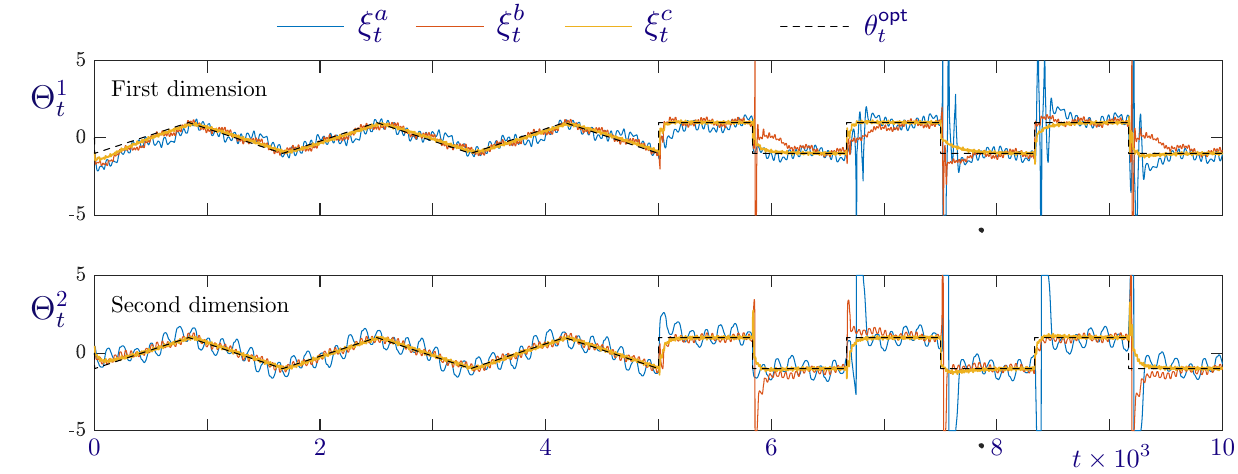}
\caption{Tracking of $\{\theta^\opt_t\}$ for several choices of probing signal.}
\label{fig:moving_camel_qsaprobe}
\end{figure}

\bigbreak

\wham{Results}  

\wham{Case (a)}
Phase plots of the trajectories $\{\ODEstate_t\},\{{\ODEstate}^{1 \text{\tiny\sf F}}_t \} $ are shown in \Cref{fig:moving_camel_lotus}. These plots display results with the probing signal $\qsaprobe^b_t$ and $\eta =5$. It is possible to see in \Cref{fig:moving_camel_lotus} that the output of $1$qSGD tracks $\{ \theta^\opt_t\}$ after a transient period. Since the gain is non-vanishing, $\bfODEstate$ oscillates around the trajectory $\{\theta^\opt_t\}$. As expected by \Cref{t:Couple_theta}, the filtered process has much less variability.

\wham{Case (b)}

\Cref{fig:moving_camel_qsaprobe} shows the evolution of $\{\ODEstate_t\}$ using each probing signal in \eqref{e:probingcamel}. Here, the signal $\{\theta^\opt_t\}$ has period $T_0$ with $b=3$ \eqref{e:theta_opt}. The plots in \Cref{fig:moving_camel_qsaprobe} illustrate how the frequency content of the probing signal impacts tracking performance: higher frequencies yield less variability over most of the run.

\begin{figure}[h]
 \centering
\includegraphics[width=0.85\textwidth]{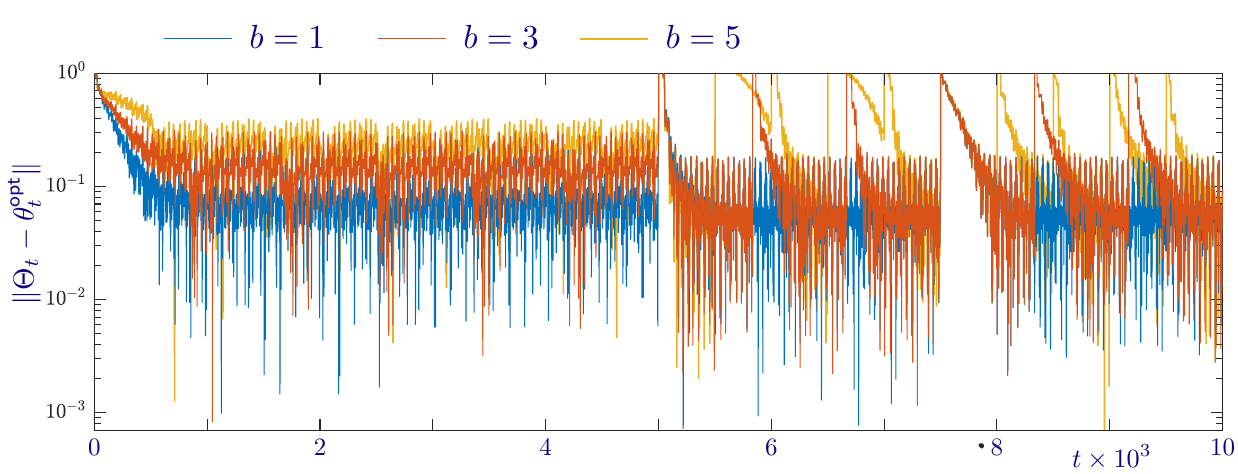}
\caption{Estimation error for different $b$ \eqref{e:theta_opt}.}
\label{fig:moving_camel_norm}
\end{figure} 


The evolution of the estimation error $\{\|\ODEstate_t - \theta^\opt_t \|\}$ is shown in \Cref{fig:moving_camel_norm} for each value of $b$ listed in \eqref{e:theta_opt}. Only $\qsaprobe^c_t$ in \eqref{e:probingcamel} was used as the probing signal for these experiments. We see in \Cref{fig:moving_camel_norm} that tracking performance deteriorates as the rate of change of $\{\theta^\opt_t\}$ increases. When $t>T/2$, the signal $\{\theta^\opt_t\}$ is nearly static for each period and hence we expect the rate of change of $\{\theta^\opt_t\}$ to have little effect on the AAD.
 This is confirmed by the results in shown in \Cref{fig:moving_camel_norm}: the steady state estimation errors are roughly the same for all values of $b$.

\Cref{fig:moving_camel_eta} shows the evolution $\{\Obj(\ODEstate_t), \Obj(  {\ODEstate}^{1 \text{\tiny\sf F}}_t  ), \Obj( {\ODEstate}^{2 \text{\tiny\sf F}}_t) \}$ as functions of time for $\eta \in \{5,15\}$. Again, the selected probing signal was $\qsaprobe^c_t$ and the signal $\{\theta^\opt_t\}$ had period $T_0$ with $b=3$ for both choices of $\eta$. We see in \Cref{fig:moving_camel_eta} that the filtered estimates $\{{\ODEstate}^{1 \text{\tiny\sf F}}_t\}$ yield lower objective values than $\{{\ODEstate}^{2 \text{\tiny\sf F}}_t\}$ for $t \leq  T/2$. When $t>T/2$, the estimates $\{{\ODEstate}^{1 \text{\tiny\sf F}}_t\}$ only yield lower objective values than $\{{\ODEstate}^{2 \text{\tiny\sf F}}_t\}$ for very small portions of the run. That is, a second order filter is preferable when the rate of change of $\{\theta^\opt_t\}$ is small.

Recall that the parameter $\eta$ scales the bandwidth of the filter. We see in \Cref{fig:moving_camel_eta} an improvement in tracking for both filters when $\eta$ is increased as a consequence of the increased bandwidth.

\begin{figure}[h]
 \centering
 \includegraphics[width=0.85\textwidth]{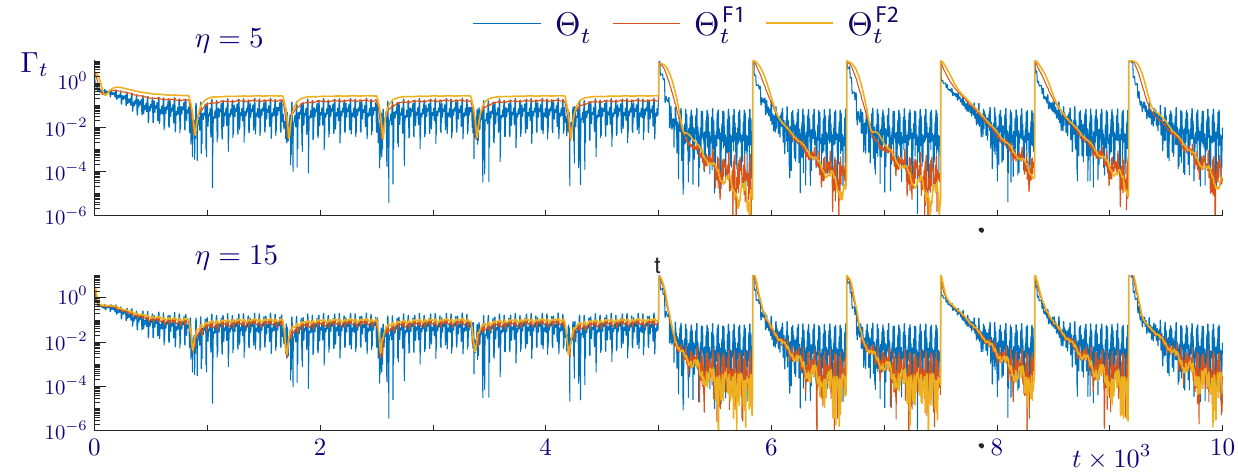}
\caption{Evolution of $\Obj$ for different filtering techniques.}\label{fig:moving_camel_eta}
\end{figure} 

 \clearpage

\section{Conclusions}
\label{s:conc}
 This paper introduces a new exact representation 
 of the QSA ODE, opening several doors for analysis.     
Major outcomes
of the perturbative mean flow representation include the clear path to obtaining transient bounds based on those of the mean flow, and guidelines on how to design filters to obtain bias 
	and AAD
of order $O(\alpha^2)$. There is much more to be unveiled:

\witem 
The use of filtering for acceleration of algorithms is not at all new.  
It will be exciting to investigate the implications of the acceleration techniques pioneered by  Polyak and Nesterov for nonlinear optimization,
particularly in their modern form  (see \cite[\S~2.2]{les22}, \cite{mohrazjov20} and the references therein). 

The integration of these two disciplines may provide insight on how to design the high pass filters shown in  \Cref{fig:ESC}, or suggest entirely new architectures.

\witem 
The introduction of normalization in the observations in the general form \eqref{e:normalizedObservations} was crucial to obtain global stability of ESC ODEs.   There are many improvements to consider.   First,   on considering the Taylor series approximation \eqref{e:clYnTS}, performance is most likely improved via a second normalization:
\[
\clYn_t =   \frac{1}{\upepsilon_t} \Bigl[ \Obj(\ODEstate_t + \upepsilon_t \qsaprobe_t)  -   \Obj^\star_t \Bigr] \,, 
\]
in which $\{ \Obj^\star_t \}$ are estimates of the minimum of the objective.   These might be obtained by passing $\{ \clY_r \eqdef \Obj(\ODEstate_r + \upepsilon_r \qsaprobe_r) \}$ through a low pass filter.  

\witem 
Far better performance might be obtained through an observation process inspired by  2SPSA.  Consider first a potential improvement of 2SPSA:   A state dependent exploration gain is introduced, so that
\eqref{e:2qSPSA} becomes
\[
\theta_{n+1} = \theta_n    -  \alpha_{n+1} \frac{1}{2\epsy}   \qsaprobe_{n+1} \bigl[\Obj(\theta_{n}  + \upepsilon_n \qsaprobe_{n+1} ) -\Obj(\theta_{n}  -  \upepsilon_n \qsaprobe_{n+1} ) \bigr]  \,
\] 
with $\upepsilon_n = \upepsilon( \theta_{n})$.  The division by $2\epsy$ (independent of state) remains, since 2SPSA in its original form satisfies the required Lipschitz conditions for SA provided $\nabla \Obj$ is Lipschitz continuous.

There are surely many ways to obtain an online version based on QSA.   One approach is through sampling:  denote $T_n = nT$ for  
a  given sampling interval $T>0$,  and
take $\clYn_t$  constant on each interval $[T_n,T_{n+1})  $, 
designed to mimic 2SPSA.   One option is the simple average,
\[
\begin{aligned}
	\clYn_{T_{n+1}}  \eqdef \frac{2}{T}  \frac{1}{2\epsy} 
	&\int_{T_n}^{T_n+T/2}   \qsaprobe_{t}  \bigl[ \Obj( \theta_n   +  \upepsilon_t \qsaprobe_t ) - \Obj(\theta_n - \upepsilon_t \qsaprobe_t ) ]  \, dt
\end{aligned} 
\]
with $\theta_n=  \ODEstate_{T_n} $.
This can be computed in real-time, based on two sets of observations:
\[
\begin{aligned}
	\Obj( \theta_n   +  \upepsilon_t \qsaprobe_t   )  \qquad  &    T_n \le t\le  T_n + T/2
	\\
	\Obj( \theta_n   -  \upepsilon_n \qsaprobe_{t-T/2}  ) \qquad   &     T_n + T/2  \le t\le T_{n+1}
\end{aligned} 
\]

\witem
The implications to reinforcement learning deserve much greater attention.  The applications of QSA and ESC in  \cite{mehmey09a,kildrs06,CSRL} are only a beginning.

\witem 
It may be straightforward to extend the P-mean flow representation  \eqref{e:BigGlobalODE} to tracking problems.   This may start with the general topic of   time inhomogeneous QSA, of the form 
\[
\ddt\ODEstate_t = \alpha f(\ODEstate_t,\qsaprobe_t;   t)\,.
\]
Analysis would require consideration of solutions to Poisson's equation, such as $\haf(\theta,\varble; t)$ for each $\theta\in\Re^d$ and $t\in \Re$.  
The representation will be more complex than 	\eqref{e:BigGlobalODE}, but will likely lead to sharper bounds than are presently available.

\clearpage
 
\section{Literature and Resources}
\label{s:lit}

Background on stability via Lyapunov criteria is standard  \cite{bor20a,kusyin97}. See  \cite{vid23} for more recent theory on Lyapunov function techniques for SA.
The 
 ODE@$\infty$ approach was introduced in 
\cite{bormey00a}, and refined in a series of papers since  \cite{bha11,rambha17,chedevborkonmey21}.  

\wham{The averaging principle}
The QSA ODE with fixed gain   \eqref{e:QSAgen_alpha} is not at all new to the dynamical systems community,  for which solidarity of the QSA ODE and the mean flow is known as the \textit{averaging principle}.     
Analysis of the larger state process  $\bfPsi =  (\bfODEstate,\bfPhi)$  may also be cast in the setting of singular perturbation theory,  in which $\bfODEstate$ is regarded as the slow variable.

The concepts are far older than stochastic approximation, with heuristics applied in the 18th century to obtain models for coupled planetary systems.  Firm theory emerged approximately one century ago  \cite{smi85},  which is several decades before Robbins and Monro introduced SA  \cite{robmon51a}.     Averaging and singular perturbation theory grew within the control systems community beginning in the 1970s  \cite{kokmalsan76},  which became a foundation of adaptive control (a close cousin of reinforcement learning) in the decades that followed.  
Any of the standard references will provide a fuller history, such as   \cite{kokorekha99,kha02,sanver07}.
 
 The discussion that follows concerns QSA with vanishing gain,  which is the typical setting of SA theory.
 
\wham{Quasi-stochastic approximation:}

The first work on QSA as an alternative to SA appears in \cite{lappagsab90,larpag12} with applications to finance,  and \cite{bhafumarwan03} contains initial results on convergence rates that are better than that obtained for their stochastic counterparts, but nothing like the bounds obtained here.   The setting for this prior work is in discrete time. Something similar to QSA appears in   \cite{bhafumarwan03} with applications to gradient-free optimization.  

The QSA ODE appeared in applications to reinforcement learning in \cite{mehmey09a},  which motivated the convergence rate bounds in  \cite{shimey11} for multi-dimensional linear models, of which 	\eqref{e:TwoLinExamples} are special cases (except that analysis is restricted to vanishing gain algorithms).  

The work of \cite{shimey11} was extended to nonlinear QSA in \cite{chedevbermey20b,chedevbermey21},  and initial results on PR averaging and the challenges with multiplicative noise is one topic of \cite[\S~4.9]{CSRL}.   Until recently it was believed that the best convergence rate possible is $O(1/T)$ for the vanishing gain algorithms considered.  
Recently, this was shown to be fallacy:  convergence rates arbitrarily close to $O(T^{-2})$ are obtained in  \cite{laumey22,laumey22e} via PR-averaging, but again this requires that a version of $\Upupsilon$ be zero.

\wham{Extremum seeking control:} ESC is perhaps among the oldest gradient-free optimization methods. It was born in the 1920s when a extremum seeking control solution was developed to maintain maximum transfer of power between a transmission line and a tram car \cite{leb22,EShistory2010,liukrs12}. Successful stories of the application of ESC to various problems have been shared over the 20th century -- e.g. \cite{rastrigin1960extremum,ras73,drali51,mee67,oba67}, but theory was always behind practice: the first Lyapunov stability analysis for ESC algorithms appeared in the 70s for a very special case \cite{luxlee71}.

General local stability and bounds on estimation error for ESC with scalar-valued probing signals were established 30 years later in \cite{krswan00} and later extended for multi-dimensional probing in \cite{arikrs02}, but the authors do not specify a domain of attraction for stability. Stronger stability results was than established in \cite{tannesmar06}, where the authors show that for an arbitrarily large set of initial conditions, it is possible to find control parameters so that the extermum seeking method results in a solution arbitrarily close the extrema to be found. This is known as semi-global pratical stability in the ESC literature. These analytical contributions sparked further research \cite{docpergua11,guadoc15,durstaebejoh13,pokeis22}.

It is pointed out in \cite{mil16} that there is a potential  ``curse of dimensionality'' when using sinusoidal probing signals,  which echos the well known curse in Quasi Monte Carlo \cite{asmgly07}.

\wham{Convergence rates for stochastic approximation:}

Consider first the case of vanishing step-size:
with step-size $\alpha_n = (1+n)^{-\rho}$,  $\half <\rho <1$, we obtain the following limits for the scaled  asymptotic mean-square error (AMSE):
	\begin{align}
		\lim_{n\to\infty} \frac{1}{\alpha_n} \Expect[ \| \theta_n -\theta^\ocp\|^2]  & = \trace (\Sigma_\theta(\rho)) 
		\label{e:PRrate1}
		\\
		\lim_{n\to\infty}  n \Expect[ \| \thetaPR_n -\theta^\ocp\|^2]  & = \trace (\Sigma_\theta^*) 
		\label{e:PRrate2}
\end{align} 
where $\{\thetaPR_n\}$ are obtained from $\{\theta_n\}$ via Polyak-Ruppert averaging \eqref{e:PR}.   The covariance matrix in the second limit has the explicit form $\Sigma_\theta =  G \Sigma_\Delta G^\transpose$ with $G =  {A^*}^{-1}$ and $\Sigma_\theta$ the covariance of $f(\theta^\ocp,\qsaprobe_{n+1}) $   (see \eqref{e:SA_recur}).  The covariance matrix $\Sigma_\theta^*$ is  minimal in a strong sense  \cite{pasgho13,chedevborkonmey21,bor20a,kusyin97}.

There is an equally long history of analysis for algorithms with constant gains.  For simplicity assume an SA algorithm with $\bfqsaprobe$ i.i.d., so that  the stochastic process $\bftheta$ is a Markov chain.   Stability of the ODE@$\infty$ implies a strong form of geometric ergodicity under mild assumptions on the algorithm
\cite{bormey00a}, which leads to bounds on the variance 
\begin{equation}
	\Expect[ \| \theta_n -\theta^\ocp\|^2]    =
	\Expect[ \| \theta_\infty -\theta^\ocp\|^2]    + O(\varrho^n)
	\label{e:BM}
\end{equation}
where the steady-state mean $ \Expect[ \| \theta_\infty -\theta^\ocp\|^2] $ is of order $O(\alpha)$ with step-size $\alpha>0$,  and $\varrho<1$ since $\bftheta$ is geometrically ergodic.   Averaging can reduce variance significantly \cite{moujunwaibarjor20,durmounausamscawai21}.


\wham{Zeroth order optimization:}

We let  $\theta^\ocp$ denote the limit of the algorithm (assuming it exists), and  $\theta^\opt$ the global minimizer of $\Obj$. 

We begin with background on methods based on stochastic approximation,  which was termed gradient free optimization (GFO) in the introduction.
\begin{subequations}
	
\wham{1. GFO }    
Much of the theory is based on vanishing gain algorithms, and the probing gain is also taken to be vanishing in order to establish that the algorithm is \textit{asymptotically unbiased}.

\end{subequations}

Polyak was also a major contributor to the theory of convergence rates for GFO algorithms that are asymptotically unbiased,  with \cite{poltsy90} a highly cited starting point.  If the objective function is $p$-fold differentiable at $\theta^\ocp$ then the  best possible rate for the AMSE is $O(n^{-\beta} )$ with $\beta = \frac{p-1}{ 2p} $. 
This work was motivated in part by the upper bounds and algorithms of Fabian \cite{fab68b}.
See \cite{dipren97,dip03,pasgho13} for more recent history, and algorithms based on averaging that achieve this bound;  typically using   averaging combined with vanishing step-sizes.

\smallskip

\smallskip


\wham{2. ESC} Theory has focused on constant-gain algorithms because many authors are interesting in tracking rather than solving exactly a static optimization problem.

It is shown in \cite{krswan00}, that it is possible to implement ESC to find a control law that drives a system near to an equilibrium state $x^\opt$, minimizing the steady-state value of its output at $-y^\opt$. It is assumed that $x^\opt = \Obj(\theta^\opt)$ where $\Obj: \Re \to \Re^d$ is a smooth function. For a scalar-valued perturbation signal of the form $\qsaprobe_t = \alpha \sin (\omega t)$, the estimation error $\| y^* - y^\opt \|$ is of order $O(\alpha^2)$, assuming stability of the algorithm. These results are later extended to the case where $\theta,\qsaprobe \in \Re^d$ in \cite{arikrs02} satisfying: $\qsaprobe^j_t = \alpha \sin(\omega_j t)$, for $\{1 \leq j \leq d\}$. In this multivariate setting, the function $\Obj : \Re^d \to \Re^d$ was assumed to be quadratic and time-varying. It is shown that the estimation error for tracking is not only a function of $\alpha$ but also the frequencies of $\qsaprobe$, since $\| y^*- y^\opt \|$ was of order $O(1/\omega^2_1 +d^2\alpha^2)$ where $\omega_1$ is the smallest frequency among $\{\omega_j\}$.

 \clearpage
 
 \bibliographystyle{abbrv}
\bibliography{strings,markov,q,QSA,extras,ESCcaio,SGDreferences}

  \clearpage
 
 \appendix

\section{Appendix}

A maximal value  $\alpha^0>0$  and   constants $\lilb\le \bigb$ are introduced in
\Cref{t:ODEstate_is_bdd},  and used in remaining results in  \Cref{s:QSAtheory}.  
Several of the results in the following construct these three scalars,  with values differing in each appearance. 

 \subsection{ODE QSA Solidarity}
 \label{s:Solidarity}
 
The first step in the proof 
of \Cref{t:Couple_theta_local}
is to  re-interpret the LLN \eqref{e:LLNf}:
\begin{lemma}
\label[lemma]{t:LLNfast}
If (A1) holds then the following uniform Law of Large Numbers holds:  there is a function $\epsy_f \colon\Re\to\Re_+$ such that for any $\theta\in\Re^d$ and $\Phi_0\in\prstate$,
\begin{equation}  
\left |
\frac{1}{T }\int_0^T  \bigl[f(\theta,\qsaprobe^\alpha_{t}) - \barf(\theta) \bigr]  \, dt
				\right|     \le    \bigl(1 +  \| \theta\| \bigr)  \epsy_f(  T /\alpha   ) 
\label{e:LLNfast}
\end{equation}
with  $  \epsy_f( r  ) \to 0 $ as $r\to\infty$.  
\end{lemma}

Armed with this lemma, we now compare solutions to \eqref{e:QSAgen_alpha_time-scaled} with solutions to the mean flow \eqref{e:barf_ODE_gen}.   If both are initialized at the same value $\odestate_0 =\ODEstate_0 =\theta_0$, we obtain for any $T>0$,
\[
\ODEstate^\alpha_T = \theta_0+  \int_0^T   f(\ODEstate_r^\alpha ,\qsaprobe^\alpha_r) \, dr   \,,
		\qquad
\odestate_T= \theta_0+  \int_0^T \barf( \odestate_r) \, dr
\]
 The challenge is that the bound \eqref{e:LLNfast} holds with $\theta$ fixed, so that it is not immediately applicable to comparison of the two ODEs.

To proceed we divide the positive real line into adjacent intervals $[\tau_n, \tau_{n+1})$,   and let $\{ \odestate^n_t :    t \ge   \tau_{n} \}$ denote the solution to \eqref{e:barf_ODE_gen}  with $\odestate^n_{\tau_n} = \ODEstate^\alpha_{\tau_n}$.   
We take   $\tau_n = \Delta  n$ for each $n$, with $\Delta \in (0,1]$ to be chosen as a function of $\alpha\le 1$.  The following corollary to   \Cref{t:LLNfast} will provide constraints on this choice. 

\begin{lemma}
\label[lemma]{t:shortSolidarity}
The following uniform bounds hold under (A1):  for a fixed constant $B_f$, and each $n\ge 0$,
\[
\|\odestate^{n+1}_{\tau_{n+1}}  -\odestate^n_{\tau_{n+1}}    \|   
=
\|  \ODEstate^\alpha_{\tau_{n+1}}   -  \odestate^n_{\tau_{n+1}}  \|  \le     
				\Delta B_f^n(\Delta,\alpha)  \eqdef
			B_f \Delta^2    +  \Delta [1 + \|  \odestate^n_{\tau_{n}}  \|  ]  \epsy_f(\Delta/\alpha)
\]
\end{lemma}

The proof of the lemma relies on \Cref{t:LLNfast} and an application of the Bellman-Gr{\"o}nwall inequality:
\begin{lemma}
\label[lemma]{t:BG}
Let $\bfodestate^i$, $i=1,2$, denote two solutions to \eqref{e:barf_ODE_gen}, with different initial conditions.   Then, subject to (A1),
\[
\|  \odestate^1_t  -  \odestate^2_t \|  \le \|  \odestate^1_0  -  \odestate^2_0 \|  e^{L_f t}\,,\qquad t\ge 0.
\]
\end{lemma}

\begin{proof}[Proof of \Cref{t:shortSolidarity}]
The equality $
\|\odestate^{n+1}_{\tau_{n+1}}  -\odestate^n_{\tau_{n+1}}    \|    = \|  \ODEstate^\alpha_{\tau_{n+1}}   -  \odestate^n_{\tau_{n+1}}  \|  $ follows from the definitions.   
To bound the latter, we write
\[
\begin{aligned}
\ODEstate^\alpha_{\tau_{n+1}}   -  \odestate^n_{\tau_{n+1}}  
	&=  \int_{\tau_n}^{\tau_{n+1}}   \bigl[  f(\ODEstate^\alpha_r, \qsaprobe^\alpha_r)  -   \barf(\odestate^n_r)\bigr]\, dr
			&&    \textbf{Bounds:}
   \\
	&=  \int_{\tau_n}^{\tau_{n+1}}  
		 \bigl[  f(\ODEstate^\alpha_{\tau_n}, \qsaprobe^\alpha_r)  -   \barf(\odestate^n_{\tau_n})\bigr]\, dr
		 &&     O\bigl( \bigl( 1+ \|\odestate^n_{\tau_n}\|  \bigr)\epsy_f(  \Delta /\alpha   )   \bigr)
	\\
	&\quad
		 +   \int_{\tau_n}^{\tau_{n+1}}   \bigl[ f(\ODEstate^\alpha_r, \qsaprobe^\alpha_r)   -   f(\ODEstate^\alpha_{\tau_n}  , \qsaprobe^\alpha_r) \bigr]\, dr	
		 \qquad
		 	 &&     O\bigl(  \Delta^2\bigr)
		 \\
	&\quad
		 +   \int_{\tau_n}^{\tau_{n+1}}   \bigl[      \barf(\odestate^n_{\tau_n})   -   \barf(\odestate^n_r)\bigr]\, dr
		 		  &&     O\bigl(  \Delta^2\bigr)
\end{aligned} 
\]
That is, the error is the sum of three terms, in which the first is bounded via \Cref{t:LLNfast}, and the second two are bounded by $O(\Delta^2)$ for $\Delta\le 1$ under the Lipschitz assumptions on $f$ and $\barf$. 
\end{proof}

The next step is to obtain uniform bounds on the error $\clE_n =   \|\odestate^{n}_{\tau_n}  - \odestate_{\tau_n}    \|   $ for each $n$,   with $\odestate_{\tau_n} $ the solution at time $\tau_n$ with initial condition $\odestate_0 = \ODEstate_0 =\theta_0$, so in particular $\clE_0=0$.   
\begin{lemma}
\label[lemma]{t:shortSolidarity2}
The following uniform bounds hold under (A1),  provided  $\epsy_f(  \Delta /\alpha   )\le 1$:      for a fixed constant $B_f^1$, and each $n\ge 0$,
\[
    \|\odestate^{n}_{\tau_n}  - \odestate_{\tau_n}    \|    \le B_f^1\exp(2 L_f \tau_n)   \bigl[  1     +      \|\theta_0\|  \bigr]  \cdot  \bigl[ \Delta +  \epsy_f(  \Delta /\alpha   )   \bigr]
\]
\end{lemma}

\begin{proof}
The proof begins with the triangle inequality to obtain a recursive bound:
\[
\begin{aligned}
\clE_{n+1} \eqdef   \|\odestate^{n+1}_{\tau_{n+1}}  - \odestate_{\tau_{n+1}}    \|  
  &   \le      \|\odestate^{n}_{\tau_{n+1}}  - \odestate_{\tau_{n+1}}    \|  
+ \|\odestate^{n+1}_{\tau_{n+1}}  - \odestate^n_{\tau_{n+1}}    \|
\\
  &   \le      \|\odestate^{n}_{\tau_{n+1}}  - \odestate_{\tau_{n+1}}    \|    
  +
\Delta B_f^n(\Delta,\alpha) 
\end{aligned} 
\]
The recursion is obtained on combining this bound with \Cref{t:BG}:
\[
\clE_{n+1}  \le     e^{L_f \Delta} \|\odestate^{n}_{\tau_n}  - \odestate_{\tau_n}    \|    +\Delta B_f^n(\Delta,\alpha) 
 =     e^{L_f \Delta} \clE_n    +\Delta B_f^n(\Delta,\alpha) 
\]
where $B_f^n$ appears in \Cref{t:shortSolidarity}.   
Applying the triangle inequality,
\[
B_f^n(\Delta,\alpha)  = 
			B_f \Delta   +    [1 + \|  \odestate^n_{\tau_{n}}  \|  ]  \epsy_f(\Delta/\alpha)
			\le
			\bigl(B_f \Delta   +  \epsy_f(\Delta/\alpha)  \bigr)   
			+   \bigl(   \|  \odestate_{\tau_{n}}  \|   +  \clE_n\bigr) \epsy_f(\Delta/\alpha)
\]
An application of \Cref{t:BG} once more gives $ \|  \odestate_{\tau_{n}}  \| \le  e^{L_f \tau_n}   \| \theta_0\|$.

Combining these bounds gives
\[
\begin{aligned}
\clE_{n+1}  &\le     \beta_\Delta \clE_n    +  \Delta   \bigl[   \clU^0   +    \clU^1  e^{L_f \tau_n}  \bigr]
   \\
 \  \textit{with} \ \              &\beta_\Delta\eqdef   e^{L_f \Delta} +  \epsy_f(\Delta/\alpha)
            \\
            &\clU^0  \eqdef \bigl(B_f \Delta   +  \epsy_f(\Delta/\alpha)  \bigr)    
            \,, \quad \clU^1 \eqdef   \epsy_f(\Delta/\alpha)    \| \theta_0\|
\end{aligned} 
\]
The final bound is obtained by iteration, and recalling $\clE_0=0$.     

It is simplest to obtain a comparison with the ordinary differential equation,
\[
\ddt X_t  =  \gamma_\Delta X_t   +   \clU^0   +  \clU^1  \exp(L_f t)  \,,\quad X_0 = \clE_0 = 0
\] 
with $\gamma_\Delta = (\beta_\Delta  -1)/\Delta$.
\end{proof}

 \begin{proof}[Proof of \Cref{t:Couple_theta_local}]
 Combining \Cref{t:shortSolidarity} and
 \Cref{t:shortSolidarity2}
 we obtain the bound \Cref{t:Couple_theta_local} for the sampling times $\tau_n$, provided we choose $\Delta = \Delta(\alpha)$ so that 
 \[
 \epsy_0(\alpha) \eqdef  \Delta(\alpha) +  \epsy_f(  \Delta(\alpha) /\alpha   )      \to 0\,,\quad \textit{as $\alpha\downarrow 0$}
 \]
For this we choose $\Delta =\alpha^{\rho}$ with $\rho\in (0,1)$.

 Lipschitz continuity of $f$ and $\barf$ implies that the bound \eqref{e:Couple_theta_local}
 is preserved for all time,  with an increase in $B_0$ by a value of only $O(\Delta)$.  
\end{proof}

 \clearpage

\subsection{Baker's Theorem and Poisson's Equation}
\label{s:BakerPoisson}

Here we explain how to obtain well behaved solutions to Poisson's equation for a function $h\colon\Re^d\to\Re$ that is
is analytic.    We denote the mixed partials by
\[
	h^{(\beta)} (z) 
	=  \frac{\partial^{\beta_1}}{ \partial z_1^{\beta_1}} \cdots  \frac{\partial^{\beta_d}}{ \partial z_d^{\beta_d}}  h\,  (z)\,,\quad  \beta \in \nat^d\, .
\] 
For $\beta \in \nat^d$ and a vector $v\in\Re^d$ we denote $v^\beta = v_1^{\beta_1}\cdots  v_d^{\beta_d}$,   and $|\beta|$ the $\ell_1$ norm:  $|\beta| =\sum |\beta_i|$. 

 A  candidate representation for $\hah$ can be found on obtaining the solution to Poisson's equation for each of the primitive functions appearing in a Taylor series expansion.    Justification is then a challenge without assumptions on the frequencies.   

For any $ \beta\in \intgr^K$ denote $\omega_\beta = \sum \beta_i \omega_i$ and $\phi_\beta = \sum \beta_i \phi_i$.  
We will see that the terms in the Taylor series expansion for $\hah$ resemble those for $h$, but with terms divided by $\omega_\beta$ (subject to the constraint that this is non-zero).    Consequently, a lower bound on  $\omega_\beta$ as a function of $|\beta|$ is required---such bounds are available under Assumption~(A0).

The following results are Theorem B.5 and Corollary B.6 of the Supplemental Material of \cite{laumey22e}. Their proofs
are based on a version of Baker's Theorem  \cite[Thm. 1.8]{bug18}:   there is a constant $C>0$ such that for any $\beta$, provided $\omega_{\beta} \neq 0$,
\[
|\omega_{\beta} |   \ge k_\beta^{-C} \,,  \quad k_\beta = \max\{3, \beta_1,\dots,\beta_d\}
\]

\begin{theorem}
\label[theorem]{t:hah}
Suppose that  $h \colon \Re^K \to \Re$ is analytic.
  Then there are coefficients $\{ c_\beta  :  \beta\in \intgr^K \}$ such that the following representations hold for $h$ and its mean $\barh$:
\begin{subequations}
\begin{align}
 h(\qsaprobe_t) & =  \sum_{ \beta\in \intgr^K} c_\beta  \cos(2\pi[\omega_\beta t + \phi_\beta])   \,,
\qquad
\barh  =  \sum_{ \beta:  \omega_\beta=0} c_\beta  \cos(2\pi  \phi_\beta)
 \label{e:tilh}
\end{align}
If in addition  the frequencies $\{ \omega_i \}$ satisfy \eqref{e:logFreq}, then
there is a solution to Poisson's equation $\hah$ that has zero mean, 
and is analytic on the domain  $\{\Co\setminus \{0\}\}^K$.    
When restricted to $\prstate$ it takes the form,
\begin{equation}
\hah(\Phi_t) =  -   \frac{1}{2 \pi  }   \sum_{ \beta:  \omega_\beta \neq  0}  \frac{1}{ \omega_\beta}  c_\beta   \sin(2\pi[\omega_\beta t + \phi_\beta]) 
\label{e:hahTheoremEquations}
\end{equation}
\end{subequations}
\end{theorem}

\begin{proof}
 Based on  (\ref{e:tilh}, \ref{e:hahTheoremEquations}), the extensions of $\tilh = h-\barh$ and $\hah$ to $z\in \{\Co\setminus \{0\}\}^K$ are easily identified:
 \begin{align*}
 \tilh(G(z)) & =   \sum_{ \beta:  \omega_\beta \neq  0}   c_\beta   \half [z^\beta  +  z^{-\beta} ]   
 \\
 \hah(z) & =   -  j \frac{1}{2 \pi  }   \sum_{ \beta:  \omega_\beta \neq  0}   \frac{1}{ \omega_\beta} c_\beta   \half  [z^\beta  -  z^{-\beta} ]   
\end{align*}
 \end{proof}

   \begin{figure}
	 
	 \centering
		\includegraphics[width=.35\hsize]{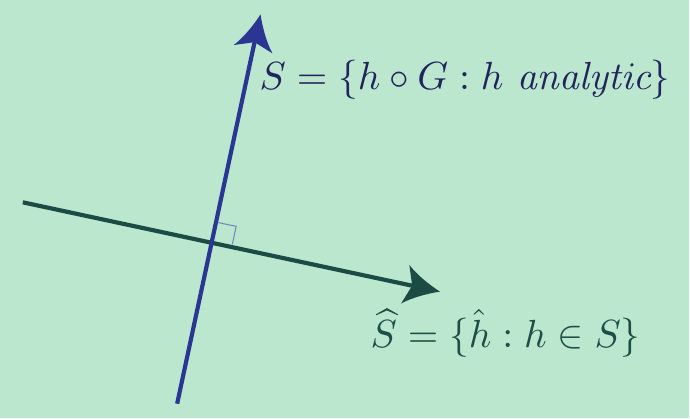}
		\caption{Orthogonality of two function classes: 1.~Analytic functions of the probing signal, and 2.~Corresponding solutions to Poisson's equation.}
	\label{fig:FishOrthogonality}
	\end{figure}  
 

\begin{corollary}
\label[corollary]{t:orthCor}
Suppose that $g$ and $h$ are two real-valued analytic functions on $\Re^K$.  
If  the frequencies $\{ \omega_i \}$ satisfy \eqref{e:logFreq}, 
then  $\Expect[ g(\qsaprobe)  \hah(\Phi)] = 0$,   where $\hah$ is given in \eqref{e:hahTheoremEquations}.
\end{corollary}

The above corollary is vital for ensuring that $\barUpupsilon\equiv \Zero$ when frequencies are chosen with care.  Its geometrical interpretation is illustrated in \Cref{fig:FishOrthogonality}.

%
This subsection is concluded with implications to the \textit{target bias}:
\begin{proposition}
\label[proposition]{t:orthCorBig} 
Consider the QSA ODE \eqref{e:QSAgen_alpha} subject to (A0a) and (A1), but with arbitrary choices of frequencies $\{\omega_i\}$.   
Suppose that an invariant probability measure $ \upvarpi$ exists, and let $\Uppsi = (\ODEstate, \Upphi)  $ denote a $\Uppi$-valued random vector with distribution $ \upvarpi$.   
Then,  for any analytic function $g \colon \Uppi\to \Re$ with zero-mean solution to Poisson's equation $\hag$,
\[
\alpha  \Expect \bigl[  [\Df \hag ] (\Uppsi )     \bigr]  =  \Expect \bigl[  g( \Uppsi )  -  \barg(\ODEstate)        \bigr] 
\]
In particular,  if $f$ is analytic and there is a solution $\haf$ to Poisson's equation,  then the   tracking bias may be expressed
\begin{equation}
\Expect \bigl[ \barf(\ODEstate)        \bigr]  =
\alpha \Expect \bigl[  \Upupsilon (\Uppsi )     \bigr]  
\label{e:newUpupsilon}
\end{equation}
\end{proposition}
 \begin{proof} 
 For (i) we  apply \Cref{t:QSAgenerator}.   
 To establish \eqref{e:newUpupsilon},   recall from below \eqref{e:Dfg} the 
formula   $\Upupsilon  = - \Df g  $,   using $g=\haf$.   Hence from (i),
 \[
 \alpha \Expect \bigl[  \Upupsilon (\Uppsi )     \bigr]  =  - \Expect \bigl[  f( \Uppsi )  -  \barf(\ODEstate)        \bigr] 
\]
The proof is completed on establishing $ \Expect \bigl[  f( \Uppsi )         \bigr] =\Zero$,  which follows from \Cref{t:QSAgenerator}. 
    \end{proof}
   
 \clearpage
 
 \subsection{Stability}
 \label{s:AppStability}
 Here, we establish exponential stability of the mean flow ODE, which along with the results in \Cref{s:Solidarity}, imply semi-exponential stability and $\alpha^0$-ultimate boundedness of the QSA ODE as stated by part (ii) of \Cref{t:ODEstate_is_bdd}. The full proof of this theorem is postponed to the end of \Cref{s:lyap}.

 \begin{proposition}
 	\label[lemma]{t:CriteriaSolidarity}  
 	Either of the criteria (V4) or asymptotic stability of the ODE@$\infty$ imply (V4').  
 \end{proposition}
 
 \begin{proof}
 	The implication (V4) $\Rightarrow$ (V4') follows from integration:   for each $T$,  subject to \eqref{e:ddt_bound_Lyapfun},
 	\[
 	V(\odestate_T) - 
 	V(\odestate_0) \le    -  \delta_0 \int_0^T  V(\odestate_{\tau}) \, d\tau
 	\]
 	This combined with the bounds in \eqref{e:VLip} imply that (V4') also holds,  with $T>0$ sufficiently large.
 	
 	\smallskip

 	We next establish (V4') subject to the assumption that the ODE@$\infty$ is asymptotically stable.   For this we  establish that the function $V$ defined in \eqref{e:BMV} is a Lipschitz continuous solution to \eqref{e:V4T}. 
 	
 	We have already remarked that Lipschitz continuity follows from (QSA2), which requires  Lipschitz continuity of $\barf$.
 	
 	For the remainder of the proof, it is enough to establish that this function solves (V4') for the ODE@$\infty$, by  applying the following bound from \cite{bormey00a}:
 	There is a function $\epsy_0\colon\Re_+\to\Re_+$ for which 
 	$ \epsy_0(r)\to 0$ as $r\to\infty$, and 
 	the following bound holds:
 	with $\theta = \odestate^\infty_0 = \odestate_0$,  and 
 	\[
 	\| \odestate^\infty_T  - \odestate_T  \|    =  \epsy_0(\| \theta\| )   \| \theta\|  
 	\]
 	This is a simple consequence of the representation \eqref{e:ODEinf-r}.

 	To obtain the inequality in \eqref{e:BMV} for $\bfodestate^\infty$ requires that we take $T$ sufficiently large:
 	Choose $T>0$ so that  $\|\odestate_t^\infty \|  \le \half  \|\odestate_t^\infty \|$ for $t\ge T$  (existence follows from Lemma 4.23 of \cite{CSRL}).    
 	We have for any  initial condition $\odestate_0$, 
 	\[
 	V(\odestate^\infty_T)   =
 	\int_0^{T}     \| \odestate_{t+T} ^\infty \| \,dt
 	\le \half 	V(\odestate^\infty_0) 
 	\]
 	It follows that this function $V$ solves (V4') for the ODE@$\infty$.    
 \end{proof}

 \begin{lemma}
 	\label[lemma]{t:littleV4bdd}
 	Suppose that  (A1) and the Lyapunov bound   (V4') hold.   Then, the following semi-exponential stability bound holds for the mean flow:
 	for positive constants $\delta_1$, $b$ and $c$,  
 	\begin{equation}
 		\| \tilodestate_t \|   \le  b \| \tilodestate_0 \|   e^{-\delta_1 t}  \,,
 		\quad
 		\textit{for $ t \le \tau_c  $,  where $\tau_c= \min\{ t : \| \tilodestate_t \|    \le c \}$  
 			and $\tilodestate_t=\odestate_t -\theta^\ocp$.  }
 		\label{e:littleV4bdd}
 	\end{equation}
 \end{lemma}

 \begin{proof}
 	
 	We can choose $\tau =0 $ in \eqref{e:V4T} without loss of generality since the system is autonomous. For $\delta_1,\beta>0$, \eqref{e:V4T} gives
 	\[
 	V(\odestate_{T}) \leq  V(\odestate_{0})  - \delta_1 \|\odestate_0\| \leq e^{ - \beta} V(\odestate_0) \quad \text{if } \|\odestate_\tau\| > c
 	\] 
 	The condition above can be extended to general $t$ by first defining $t \eqdef nT $. 
 	Then,
 	\[
 	V(\odestate_{t}) \leq e^{ - \beta_1 nT} V(\odestate_0), \quad \beta_1 = \beta/T  \quad \text{if } t = nT \leq \tau_c
 	\]
 	This establishes \eqref{e:littleV4bdd} at these discrete time points.   It extends to arbitrary $t\le \tau_c$ via Lipschitz continuity of $V$.   
 \end{proof}

 \begin{proof}[Proof of \Cref{t:qsaprobeExpAS}]
 	The proof is obtained by interpreting 	\Cref{fig:qsaprobeExpAS}.   The time $T_1$ is  the first $t\ge 0$ such that $\odestate_t\in B_1$,  and $T_2$   the first $t\ge T_1$ such that $\odestate_t\in B_2$.   The set $B_1$ is defined in the figure, and $B_2$ is a region of exponential asymptotic stability.

 	\begin{figure}[h]
 		\centering
 		
 		\includegraphics[width=.35\hsize]{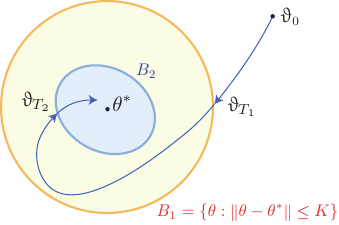}
 		\caption{Roadmap for the proof of exponential asymptotic stability for $	\ddt \odestate_t  = \barf  ( \odestate_t  )
 			$. }%
 		\label{fig:qsaprobeExpAS}
 	\end{figure} 
 	
 	\begin{subequations}
 		\Cref{t:littleV4bdd} tells us that,
 		under (V4'),  the constant $K$ defining
 		$B_1$ can be chosen sufficiently large so that there  are positive constants
 		$b_1$ and $\delta_1$ such that the following bound holds for $\odestate_0\in B_1^c$:    
 		\begin{equation}
 			\| \odestate_{T_1} -  \theta^\ocp \|   \le  b_1 \| \odestate_0-  \theta^\ocp \|    e^{-\delta_1 T_1}
 			\label{e:EAS1}
 		\end{equation}
 		The set $B_1$ is bounded,  so that $T_2 - T_1$ is uniformly bounded over all initial conditions for $\odestate_0$---this is where we apply globally asymptotically stability from (A4).   Consequently, the constant $b_1$ can be increased if necessary to obtain
 		\begin{equation}
 			\| \odestate_{T_2} -  \theta^\ocp \|   \le  b_1 \|  \odestate_{T_1} -  \theta^\ocp \|     
 			\label{e:EAS2}
 		\end{equation}
 		Finally we specify $B_2$:  it is chosen so that it is absorbing, that is $\odestate_t\in B_2$ for all $t\ge T_2$,   and for some positive constants
 		$b_2$ and $\delta_2$,
 		\begin{equation}
 			\| \odestate_t -  \theta^\ocp \|   \le  b_2 \| \odestate_{T_2} -  \theta^\ocp \|    e^{-\delta_2 (t- T_2) } \,,\qquad t\ge T_2.
 			\label{e:EAS3}
 		\end{equation}
 		The proof is completed on combining the three bounds \eqref{e:EASproof}, with $\delta = \min(\delta_1,\delta_2)$.
 		\label{e:EASproof}
 	\end{subequations}
 \end{proof}

 Recall $Y_t \eqdef\ODEstate_t  + \alpha \haf_t$  was introduced in \Cref{t:PMF}.

 \begin{lemma}
 	\label[lemma]{t:Y_Oalpha}
 	Under the assumptions of \Cref{t:Couple_theta}, 
 	there is $\alpha^0>0$ and $k_y<\infty $ such that for each $\alpha \in (0, \alpha^0]$,
 	\[
 	\limsup_{t\to\infty }  \| Y_t - \theta^* \|  <  k_y \alpha  \, .
 	\]
 \end{lemma}
 
 \begin{proof}
 	We assume without loss of generality that $\|\ODEstate_t \| \le \bigb  $ for all $t$;  in view of  \Cref{t:ODEstate_is_bdd}~(ii) the bound $\bigb $ may be chosen independent of $0<\alpha < \alpha^0$.  
 	
 	\Cref{t:qsaprobeExpAS} implies that there is a function  $V\colon\Re^d \to\Re_+$ with Lipschitz gradient satisfying, for   constants $\delta,\delta_0>0$,
 	\begin{equation}
 		\nabla V(x) \cdot \barf(x) \le -\delta_0 \|x - \theta^*\|^2  \le  -\delta  V(x) \,, \qquad x\in\Re^d \,.
 		\label{e:Lyapquad}
 	\end{equation}  
 	The chain rule combined with \eqref{e:prematureMeanflow} then gives, 
 	\[
 	\ddt V(Y_t)  =  \nabla V\, (Y_t) \cdot \ddt Y_t  =   \nabla V\, (Y_t) \cdot \bigl[ 
 	\alpha\bigl[   \barf( Y_t    )   + \alpha  \clB_t \bigr] \bigr]  \,, \qquad   \clB_t = -B_t  \haf_t  +  \Upupsilon_t  
 	\]  
 	Under the assumption that $\|\ODEstate_t \| \le \bigb  $ for all $t$,   we conclude that $ \| \clB_t \| $ is also uniformly bounded, independent of $t$ and $0\le \alpha\le \alpha^0$.   Consequently,  on applying the bound  \eqref{e:Lyapquad},   we conclude that there is a constant $b_v$ such that
 	\[
 	\ddt V(Y_t)   
 	\le
 	- \alpha  \delta  V( Y_t ) + \alpha^2 \nabla V\, (Y_t) \cdot  \clB_t
 	\le
 	- \alpha   \delta  V( Y_t )   + \alpha^2 b_v\,, \qquad t\ge 0 \, .
 	\] 
 	An application of the Bellman-Gr{\"o}nwall inequality completes the proof.
 \end{proof}
 
 \clearpage

\subsection{Filtering footnotes}
\label{s:FiltFoot}
This section contains results leading to the proof of \Cref{t:Couple_theta}. Its proof is postponed to the end of \Cref{s:lyap}.

Consider a stable filter with transfer function denoted $H$ and impulse response $h$.

Let $g$ be any differentiable function with Laplace transform $G$, $g^{(m)}$ be its $m$th derivative, and denote
\[
\hag^{(m)}_t = \int_0^t  h_{t-\tau} g^{(m)}_\tau\, d\tau\,,
\]
omitting the $m$ in the special case $m=0$.  
Let $\| g\|_1$, $\| g\|_\infty$ denote the respective $L_1$ and $L_\infty$-norms of the function $g\colon\Re_+\to \Co$.

Bounds on $\hag^{(i)}$ are obtained in the following,  which are uniform for vanishing $\alpha$ for the choice of filter \eqref{e:GVSOS}.

\begin{proposition}
	\label[proposition]{t:FilterDesign}
	Consider the stable second order transfer function 
	\[
	H(s) =   H^1(s) H^2(s) \,,\quad 
	H^i(s) = \frac{\gamma_i }{ s+\gamma_i}\, ,  \qquad i=1,2,       
	\] 
	with  $\text{Re}(\gamma_i)<0$ for each $i$.  Then,
	\begin{equation}
		\begin{aligned}
			| \hag^{(1)}_t|  & \le | \gamma_1|  \bigl( \| h^2 \|_1  +\| h_2 \|_1 \bigr)   \|g\|_\infty  + o(1)  
			\\
			| \hag^{(2)}_t|   &\le | \gamma_1 \gamma_2|  \bigl( 1  + \| h^1 \|_1  +\| h^2 \|_1  +\| h_2 \|_1 \bigr)   \|g\|_\infty  + o(1)  
		\end{aligned} 
		\label{e:HattenDerivatives}
	\end{equation}
	For the special case \eqref{e:GVSOS} with $0<\zeta\le 1$ fixed,  
	\begin{align}
		\| h^2 \|_1 = \| h^1 \|_1 & \le  \frac{1}{\zeta}   \,,
		\qquad
		\| h\|_1   \le  \frac{1}{\zeta^2}
		\label{e:GVSOSgains}
	\end{align} 
\end{proposition}


\begin{proof}
	For any $m$, the Laplace transform of $ \hag^{(m)}$ is    
	\[
	\clL( \hag^{(m)}) = 
	s^m H(s) G(s)  - H(s) \text{poly}_m(s)
	\]
	In particular, $ \text{poly}_1(s) =  g_0$ and  $ \text{poly}_2(s) = s g_0 + g'_0$.

	Consider a stable first order filter with transfer function $H^1(s) = \beta / (s+\beta)$, where $\text{Re}(\beta)>0$.
	If we pass the derivative of the function $g$ through this filter, then the Laplace transform of the output is  $s H^1(s) G(s)$, with   
	\[
	s H^1(s) =  \beta  [1-   H^1(s)  ]
	\]
	From this we conclude that we obtain attenuation by $\beta$:   
	$
	\hag^{(1)}_t   =  \beta  (g_t - \hag_t )  + o(1)  
	$,
	where the $o(1)$ term is a constant times $e^{-\gamma t}$.  Letting $\|g\|_\infty = \sup_t |g_t|$  gives
	\[
	| \hag^{(1)}_t|   \le   |\beta|   (1+ \| h^1\|_1 ) \|g\|_\infty  + o(1)  
	\]  
	We have $\| h^1\|_1 = 1$ if $\beta$ is real.
	
	This easily leads to the desired bounds for $H = H^1 H^2$ using $\beta=\gamma_i$.    
	We have,
	\[
	\begin{aligned}
		s H(s) &=  {\gamma_1 } \bigl[  1-  H^1(s)   \bigr]    H^2(s)    =  {\gamma_1 } \bigl[    H^2(s)-  H(s)   \bigr]  
		\\
		s^2 H(s) &=  {\gamma_1\gamma_2} \bigl[  1-  H^1(s)   \bigr] \bigl[  1-   H^2(s)   \bigr] 
		={\gamma_1\gamma} \bigl[  1-  H^1(s)    -   H^2(s)  + H(s)  \bigr]  
	\end{aligned} 
	\] 
	which results in the bounds \eqref{e:HattenDerivatives}.

	We turn next to the bounds in \eqref{e:GVSOSgains} for the filter \eqref{e:GVSOS}. 
	The $L_1$ norm is the induced operator norm, with $h$ viewed as a linear operator on $L_\infty$. This implies it is sub-multiplicative.  In the notation above,
	\[
	\begin{aligned} 
		\| h \|_1 &=  \| h^1 *h^2 \|_1  \le \| h^1 \|_1 \| h^2 \|_1
		\\
		\textit{and}\quad
		\| h^1 \|_1 & \le {|\gamma_1|}  \int_0^\infty \exp(- \text{Re}\, (\gamma_1) t )\, dt  
		=   {|\gamma_1|}  \frac{1}{\text{Re}\, (\gamma_1)}  
	\end{aligned} 
	\]
	For these complex poles, $|\gamma_1| =\gamma$   and $ \text{Re}\, (\gamma_1)  =\zeta \gamma$, 
	which gives the desired bounds:
	\[
	\| h^2 \|_1=\| h^1 \|_1  \le  \frac{1}{\zeta}
	\qquad
	\| h \|_1   \le \| h^1 \|_1 \| h^2 \|_1  \le   \gamma^2 \frac{1}{\zeta^2 \gamma^2}  = \frac{1}{\zeta^2}
	\]
\end{proof}

The proof of \Cref{t:Couple_theta} requires a linear approximation of the QSA ODE.
\begin{subequations}		
	\begin{lemma}
		\label[lemma]{t:BiasVariance} 
		Suppose that $\theta^*\in\Re^d$ is the unique solution to $\barf(\theta)=\Zero$.  
		Suppose moreover that Assumptions (A1) and (A2) hold, and denote $\tiltheta = \theta-\theta^\ocp$.
		Then,
		\whamrm{(i)}   There is a function $\clE_A\colon\Re^d\to\Re^d$ satisfying
		\begin{equation}
			\barf(\theta) = A^\ocp \tiltheta  +   \clE_A(\theta)  \,,\qquad \theta\in\Re^d  \, .
			\label{e:MVTtarget}
		\end{equation}
		Consequently, if  $A^\ocp$ is invertible, 
		\begin{equation}
			\ODEstate_t  -\theta^\ocp  =     [A^\ocp ]^{-1}  \bigl[    \barf(\ODEstate_t)  -  \clE_A(\ODEstate_t)  \bigr]  \, .
			\label{e:biasTargetBiasVariancePre}
		\end{equation} 
		The error term  admits the quadratic bound $ \clE_A(\theta) \le L_A  \|\tiltheta\|^2$,  and is also Lipschitz continuous.

		\whamrm{(ii)}  Provided the target bias and variance are finite,   the bias defined in \eqref{e:biasdef}
		admits the bound
		\begin{equation}
			\bias_\ODEstate  \le  \|  [A^\ocp ]^{-1} \|_F \bigl[  \bias_\barftwo     + L_A \,   \sigma_\ODEstate^2 \bigr]  \, .    
			\label{e:biasTargetBiasVariance}
		\end{equation}
		where   the subscript $F$ indicates the Frobenius norm.
		\qed
	\end{lemma}
	
\end{subequations}

\begin{proposition}
	\label[proposition]{t:PMFlin}
	Under the assumptions of \Cref{t:Couple_theta}, we have that  for each initial condition $q_0 = (\theta, z)\in\bivstate$,
	\begin{equation}
		\ddt \ODEstate_t    =  \alpha A^*  \bigl[ \ODEstate_t -\theta^\ocp    -     \alpha   \barY^*
		+    \clW_t   \bigr] +O(\alpha^3)     \,,\qquad \textit{for all   $ t\ge \tau^\alpha_{\lilb}(q_0)$}
		\label{e:BigGlobalODElin}
	\end{equation}
	with $\lilb$ given in part (ii) of \Cref{t:ODEstate_is_bdd}.
	\qed
\end{proposition}

This linearization is useful only if  the matrix   $A^\ocp = \barA (\theta^\ocp)$ is Hurwitz, which is why this assumption appears in (A4).     Subject to this assumption,  the approximation \eqref{e:BigGlobalODElin} strongly suggests that we apply linear filtering techniques to reduce volatility.

\begin{lemma}
	\label[lemma]{t:filterproof}
	Suppose the assumption of \Cref{t:Couple_theta} hold. Suppose in addition that for each $\alpha$, filtered estimates of $\bfODEstate$ are obtained using \eqref{e:GVSOS},  in which
	the value of $\zeta \in (0,1)$ is fixed,  and $\gamma =\eta \alpha$ with $\eta>0$ also fixed. Then, the following attenuation is obtained for $\clW_t$ in \eqref{e:BigGlobalODE},
	\begin{equation}
		| \clW_t^{i \text{\tiny\sf F}} |   = O(\alpha^i) \| \clW^i\|_\infty + o(1)   =   O(\alpha^3) + o(1)\,,\qquad t\ge 0\,,  \ i=1,2
		\label{e:clWattenuation}
	\end{equation}
	where $\clW^{i \text{\tiny\sf F}}$ is the output of the filter when the input is $  \clW^i$,
	and $\| \clW^i\|_\infty $ denotes the $L_\infty $ norm.   
\end{lemma}
\begin{proof}
	The proof of \eqref{e:clWattenuation} follows from \Cref{t:FilterDesign} along with the linearization of the QSA ODE \eqref{e:BigGlobalODElin}.
\end{proof}

\clearpage

\subsection{Coupling and Lyapunov exponents}
\label{s:lyap}
\Cref{s:AppStability} provided results leading to parts (i) and (ii) of \Cref{t:ODEstate_is_bdd}.   To prove part (iii), we consider the scaling $\clS_t = \exp(r t ) \clS^0_t$, with $r>0$
so that $\{\clS_t\}$ solves the time varying linear system
\[
\ddt \clS_t  =   \alpha [Ir+ A(\ODEstate_t, \qsaprobe_t)] \clS_t \,,   \qquad   \clS_0 = I  \ (d\times d) \, .
\]
with $ A( \theta ,  \qsaprobe)$ defined in \eqref{e:barfDer}. 
The joint process $(\ODEstate_t ; \clS_t)$ can be regarded as the solution to a $2d$ QSA ODE, with $2d$-dimensional mean flow,
\begin{equation}
	\ddt \odestate_t = \alpha \barf(\odestate_t) \, , \quad 
	\ddt s_t  = \alpha [ Ir + \barA(\odestate_t) ] s_t
	\label{e:Sensitivityjoint}
\end{equation}

Stability of the first ODE is imposed by (A4). The second ODE is asymptotically stable, provided $r$ is chosen small enough so that there exists a common Lyapunov function for $\{ Ir + \barA(\odestate_t): t\geq 0\}$,
\begin{lemma}
	\label[lemma]{t:GstableLemma}  
	Under (A4), there exists $r>0$, $\epsy_\ell>0$  and a positive definite matrix $P$ such that 
	$[ Ir + \barA(\theta) ]   P   +  P  [ Ir + \barA(\theta) ]  ^\transpose \le -I $ whenever $  \| \theta - \theta^* \|  \le \epsy_\ell$.
	\qed
\end{lemma}

The value of \Cref{t:GstableLemma} is made clear in the next two lemmas.
The first is immediate from \Cref{t:Y_Oalpha}.

\begin{lemma}
	\label[lemma]{t:alpha_ball}
	Suppose the assumptions of \Cref{t:Couple_theta} hold. Then, there exists $\alpha^0>0$ so that there are $k < k^\circ$ such that for all $\alpha \in (0, \alpha^0]$, $\tau^\alpha_{\lilb} (q_0) < \infty$ for all $q_0 \in \bivstate$ with $b = k \alpha$,
	and $\|  \ODEstate_t - \theta^* \| \leq k^\circ \alpha $  for $ t \geq \tau^\alpha_{\lilb} (q_0)$.  
\end{lemma}

\Cref{t:alpha0_2d} follows from \Cref{t:GstableLemma}  combined with part (ii) of \Cref{t:ODEstate_is_bdd}.

\begin{lemma}
	\label[lemma]{t:alpha0_2d}
	Suppose the assumptions of \Cref{t:ODEstate_is_bdd} hold and that $r>0$ is chosen so that the conclusion of \Cref{t:GstableLemma} holds for some
	$\epsy_\ell>0$.
	
	Then,   there exists  $\alpha^0>0$   and constants $\lilb \le \bigb$ such that 
	for each $\alpha \in (0, \alpha^0]$,  $q_0 = (\ODEstate_0,\Phi_0) \in \bivstate$,  and $\clS_0 = I$ in \eqref{e:Sensitivityjoint},
	there is a finite time $\tau^\alpha_{\lilb}(q_0)$ that is continuous in $q_0$ such that
	\begin{equation}
		\| \ODEstate_t - \theta^* \| + \| \clS_t \|_F \leq \bigb \, , \quad t\geq \tau^\alpha_{\lilb}(q_0,I)
		\label{e:Sens_isbdd}
	\end{equation}
\end{lemma}

The bound \eqref{e:Sens_isbdd}  implies that the scaled process $\exp(rt)\| S^0_t\|_F $ is bounded.
The value of $\kappa_0$ used in the lemma and the bound  $\tau^\alpha_{\lilb}(q_0)$ are adopted in the following.

\begin{lemma}
	\label[lemma]{t:coupling_lemma}
	Under the assumptions of \Cref{t:alpha0_2d},  the following holds for   $\alpha \in (0, \alpha^0]$ and  $q_0 = (\ODEstate_0 , \Phi_0) \in \bivstate$,
	\[
	\begin{aligned}
		\| 	\flow_t(\ODEstate_0,\Phi_0) - \flow_t(\theta^*,\Phi_0)\|  
		\le 
		b^\circ  \|  \ODEstate_T  - \theta^*  \|  \exp( - r (t- T )) 
		\,, \qquad t\ge  T 
	\end{aligned}
	\]	
	with $T=T(q_0) = \max \{\tau^\alpha_{\lilb}(q_0)  ,   \tau^\alpha_{\lilb}(\theta^*,\Phi_0) \}$.
\end{lemma}

\begin{proof}
	We  assume without loss of generality that  $T(q_0) =0$.
	
	Let
	$\ODEstate_t^s = (1-s) \ODEstate_t  +  s \theta^*$ for all $t$,  and $\clS_t^{0,s}  = \partial_\theta \flow_t(\ODEstate_0^s,\Phi_0) $.
	The mean value theorem  combined with \eqref{e:Sens_isbdd} yields, for each $\Phi_0 \in \prstate$,
	\[
	\begin{aligned}
		\| 	\flow_t(\ODEstate_0,\Phi_0) - \flow_t(\theta^*,\Phi_0)\|  
		= \Big\|   \int^1_0 \clS_t^{0,s}     [ \ODEstate_0  - \theta^* ]  \,ds   \Big\| 
		\le    b^\circ \exp( - r t ) \|  \ODEstate_0  - \theta^*  \| 
	\end{aligned}
	\]
\end{proof}

In the discussion surrounding  \Cref{fig:CFP}  the notation  $ \flow^-_T(\theta,z) \eqdef \flow_T(\theta,\exp(-W T) z)$ was introduced to indicate the process initialized at time $\ODEstate_{-T} =\theta$.
\Cref{t:coupling_lemma} implies convergence:

\begin{lemma}
	\label[lemma]{t:conv_determ}
	Under the assumptions of \Cref{t:alpha0_2d},   
	$\displaystyle
	\lim_{t \to \infty} \flow^-_t(\theta,z) =\flow_\infty^-(z) $ for $   (\theta,z)  \in \bivstate$, 
	where  convergence is uniform on compact subsets of $\bivstate$, and the limiting function  $\flow_\infty^-: \prstate \to \Re^d $ is continuous.
	\qed
\end{lemma}

\begin{lemma}
	\label[lemma]{t:stat_pross}
	Under the assumptions of \Cref{t:alpha0_2d},    suppose  $\Phi_0 \sim \uppi$ is random.   Then
	$\{ \Phi_t^\infty = e^{Wt} \Phi_0 : -\infty < t <\infty   \}$ is a stationary stochastic process,   and the joint process $\{\Psi_t^\infty =  ( \ODEstate^\infty_t ,   \Phi_t^\infty )=  ( \flow_\infty^-(\Phi_t^\infty) ,   \Phi_t^\infty ): -\infty < t <\infty   \} $ is a stationary realization of the QSA ODE.
	\qed
\end{lemma}

\begin{proof}[Proof of \Cref{t:ODEstate_is_bdd}]
	Part (i) follows from time-scaling, while part (ii) is obtained from \eqref{e:ExpStableAlmost} and \Cref{t:Couple_theta_local}. Part (iii) follows from \Cref{t:coupling_lemma}.
\end{proof}

\begin{proof}[Proof of \Cref{t:varpi}]
	Part~(i) is immediate from \Cref{t:conv_determ,t:stat_pross}.
	An application of \cite[Prop. 6.4.2]{MT} establishes that $\bfPsi$ is an e-process.

	Part (ii) follows from taking expectations of both sides of \eqref{e:EffNoiseDecomp} with respect to $\upvarpi$.

	For (iii), we have that  $\barUpupsilon(\theta) \equiv 0 $ for each $\theta$ under (A0) through an application of \Cref{t:orthCor} with $\hah = \haA$ and $g=f$.
\end{proof}

\begin{proof}[Proof of \Cref{t:Couple_theta}]
	The bound \eqref{e:fixedGain1} follows from \Cref{t:Y_Oalpha}.
	The  linear systems approximation  in   \Cref{t:PMFlin} combined with the bounds in  \Cref{t:filterproof} yield the desired bound \eqref{e:fixedGain1PR}.
\end{proof}

\begin{proof}[Proof of \Cref{t:asympt_bias}]
		The bounds in \eqref{e:Cor_bias}  follow from \Cref{t:BiasVariance} along with   \eqref{e:fixedGain1} and \eqref{e:fixedGain1PR}.
\end{proof}

\clearpage

\subsection{Implications to ESC}
\label{s:ESCtheory}

\subsubsection{Stability}

Recall the following bound for a function $\Obj(\theta):\Re^d \to \Re$ with Lipschitz gradient:
 \begin{equation}
	\Obj(\theta) \leq \Obj(\theta^\prime) + [\theta - \theta^\prime]^\transpose \nabla \Obj(\theta^\prime) + \tfrac{1}{2}L \|\theta - \theta^\prime \|^2,\quad \theta^\prime, \theta \in \Re^d
	\label{e:Lip_Grad}
\end{equation}
where $L$ is the Lipschitz constant associated with $\nabla \Obj$.
This bound and a bit more allows us to construct a solution to (V4) for the qSGD ODEs:

  \begin{proposition}
\label[proposition]{t:Lip_implies_V4}
Suppose that $\Obj$ has a Lipschitz gradient, so that \eqref{e:Lip_Grad} holds.  Suppose moreover that   $\|  \nabla\Obj (\theta)\| \ge \delta\|\theta\| $  for some $\delta>0$ and all $\theta$ satisfying $\|\theta\|\ge \delta^{-1}$.   Suppose that the probing signal is chosen of the form \eqref{e:qSGD_probe} with $K=d$ and $\Sigmaqsa>0$. 
 Then,  $V(\theta) = \sqrt{\Obj(\theta) - \Obj^\opt }$ satisfies assumption (V4) for both  1qSGD and 2qSGD. 
  \end{proposition}

\subsubsection{Vector field approximations}

To obtain bounds that are uniform in the scaling gain $\epsy>0$ it is useful to apply a Taylor series representation very different from what was used in \Cref{t:hah}.

 \begin{lemma}
\label[lemma]{t:ObjObsTaylor}
If the objective function $\Obj$ is analytic, then it and the normalized observation function $\clYn$ admit the representations,
\begin{equation}
\begin{aligned} 
\Obj(\theta + \upepsilon \qsaprobe) & =   \Obj(\theta) +  \sum_{| \beta|\ge 1}   c_\beta  	\Obj^{(\beta)} (\theta)  \upepsilon^{|\beta|}  \qsaprobe^{\beta}
\\
    \clYn( \theta,\qsaprobe)   & =  \frac{1}{\upepsilon}   \Obj(\theta) +  \qsaprobe^\transpose \nabla\, \Obj(\theta)  
    				+  \sum_{| \beta|\ge 2}   c_\beta  	\Obj^{(\beta)} (\theta)  \upepsilon^{|\beta|-1}  \qsaprobe^{\beta}
\end{aligned} 
\label{e:ObjObsTaylor}
\end{equation}
where in each case the sum   is restricted to $\beta\in\nat^d$.
\end{lemma}

For general QSA theory there is no loss of generality in taking $G$ analytic in the definition  \eqref{e:qsaprobeG}.     
In applications to ESC it is convenient to take $G$ linear, since in this case we obtain better control of the terms in  \eqref{e:ObjObsTaylor},  and hence better approximations for $\haf$ and its derivatives.   Before summarizing these results we require some notation.

Let $\haG$ denote the solution to Poisson's equation with forcing function $\qsaprobe_t = G(\Phi_t)$,   and $\haSigma$ the solution to Poisson's equation with forcing function $\qsaprobe_t \qsaprobe_t ^\transpose = G(\Phi_t)G(\Phi_t)^\transpose$.  
In the special case  \eqref{e:qSGD_probe}  we write    $\qsaprobe_t = V \qsaprobe_t^0$ using
\begin{equation}
	\qsaprobe _t^0  =  [  \cos (2\pi [\,  \omega_1 t  +  \phi_1 ] ) ;\,  \dots  ;\,    \cos (2\pi [\,  \omega_K t  +  \phi_K ] )  ]
\label{e:qSGD_probe0}
\end{equation}
 with $V$ the $m\times K$ matrix with columns equal to the $v^i$ appearing in  \eqref{e:qSGD_probe}.

\begin{subequations}

\begin{lemma}
\label[lemma]{t:FirstAndSecondOrderPoisson}
Zero mean solutions to Poisson's equation
 with respective forcing functions  $G(z)$,  $G(z)G(z)^\transpose$, 
may be expressed in the time domain using 
$\haqsaprobe_t = \haG  (\Phi_t) $ and   $\haSigma_t =\haSigma  (\Phi_t) $  with
\begin{align}
\!
\!
\!
\haqsaprobe_t  &=   V    \haqsaprobe^0_t   \,,
\qquad 
&&
  \haqsaprobe^{0i}_t  =  
   \frac{1}{2\pi \omega_i}  \sin (2\pi [\,  \omega_i t  +  \phi_i  ] ) \,,\qquad 1\le i\le d
\label{e:hapsaprobe}
\\
\haSigma_t  &=  V \haSigma^0_t  V^\transpose  \,,
&&
\haSigma^{0i,j}_t  =   \begin{cases}
				\phantom{+}\frac{1}{4\pi} \frac{1}{\omega_i +\omega_j}    \sin (2\pi [\,  (\omega_i + \omega_j)t  +  \phi_i +\phi_j   ] )  &  
\\  
				+\frac{1}{4\pi} \frac{1}{\omega_i -\omega_j}    \sin (2\pi [\,  (\omega_i - \omega_j)t  +  \phi_i -\phi_j    ] )  &   i\neq j
				\\[.5em]
				\phantom{+}
				\frac{1}{8\pi} \frac{1}{\omega_i }    \sin (4\pi [\,  \omega_i t  +  \phi_i   ] )    & i=j
				\end{cases}
\label{e:haSigmaqsaprobe}
\end{align}
\end{lemma}
\end{subequations}

\begin{proof}
The representation \eqref{e:hapsaprobe} is immediate from  $\qsaprobe_t = V \qsaprobe_t^0$ and the definition \eqref{e:PoissonDefn}.

The trigonometric identity $ \cos(2x  )  =  2  \cos( x)  ^2 -1$ and $\Sigmaqsa =  \half VV^\transpose $ gives   
\[
\begin{aligned} 
\qsaprobe_t \qsaprobe_t ^\transpose  	  -   \Sigmaqsa
	& = V   \diag  [  \cos^2 (2\pi [\,  \omega_1 t  +  \phi_1 ] ) \,, \dots \,,   \cos^2 (2\pi [\,  \omega_K t  +  \phi_K ] )  ]   V^\transpose 	  -   \Sigmaqsa
	\\
	& = \half V  \diag [  \cos (2\pi [\,  2\omega_1 t  +  2\phi_1 ] ) \,, \dots \,,   \cos (2\pi [\,  2\omega_K t  +  2\phi_K ] )  ]   V^\transpose 
\end{aligned} 
\]
The representation \eqref{e:haSigmaqsaprobe} follows from integration of each side, and applying the definition \eqref{e:PoissonDefn}.
\end{proof}

\Cref{t:ObjObsTaylor} is presented with $\upepsilon $ an independent variable.  It remains valid and useful when we substitute $\upepsilon = \upepsilon(\theta)$:

\begin{subequations} 
\begin{proposition}
\label[proposition]{t:ESCapproximation}   
If $\Obj$ is analytic,
	and  the frequencies $\{ \omega_i \}$ satisfy \eqref{e:logFreq},
 then we have the following approximations for 1qSGD when $\upepsilon$ is   a function of $\theta$:
\begin{align}
f( \theta,\qsaprobe)   & = - \frac{1}{\upepsilon(\theta)}   \Obj(\theta) \qsaprobe
			- \qsaprobe \qsaprobe^\transpose \nabla\, \Obj(\theta)   +  O(\epsy)   
&&   
\barf( \theta)   = - \Sigmaqsa\nabla\, \Obj(\theta)   +  O(\epsy^2)   
\label{e:ESCapprox_f}
   \\
A( \theta,\qsaprobe)   & =  
 -  \partial_\theta\Bigl( \frac{1}{\upepsilon(\theta)}   \Obj(\theta) \Bigr) \qsaprobe
 - \qsaprobe \qsaprobe^\transpose \nabla^2\, \Obj(\theta)   +  O(\epsy)  
&&
\barA( \theta)    =   - \Sigmaqsa \nabla^2\, \Obj(\theta)   +  O(\epsy^2)   
  \label{e:ESCapprox_A} 
  \\
 \haf( \theta, \Phi)   & = - \frac{1}{\upepsilon(\theta)}   \Obj(\theta) \haG(\Phi)
			- \haSigma(\Phi) \nabla\, \Obj(\theta)   +  O(\epsy)   
&&   
\label{e:ESCapprox_haf} 
\end{align} 
where in each case the ratio $O(\epsy^p) /\epsy^p $ is uniformly bounded over $\epsy>0$, when $\theta$ is restricted to any compact subset of $\Re^d$.    
\end{proposition}
\end{subequations}

\begin{proof}
The expressions for $f$ and $A$ are immediate from 
\Cref{t:ObjObsTaylor},  and the expression for $\haf$ then follows from \Cref{t:hah} combined with  \Cref{t:FirstAndSecondOrderPoisson}.
The error is $O(\epsy^2)$ for the means $\barf$ and $\barA$ because $\Expect[ \qsaprobe^\beta] = 0$ when $|\beta| =3$.
\end{proof}

Expressions for the other terms appearing in \Cref{t:PMF} immediately follow.  We have
 \[
  \hahaf( \theta,\Phi)    = - \frac{1}{\upepsilon(\theta)}   \Obj(\theta) \hahaqsaprobe(\Phi)
			- \hahaSigma(\Phi) \nabla\, \Obj(\theta)   +  O(\epsy)   
\]
in which	the ``double hats'' are defined in analogy with $\hahaf$, and
\[
\Upupsilon ( \theta,\Phi) =  - [\Df \haf]( \theta,\Phi)   = -  
 \Bigl[  \partial_\theta\Bigl( \frac{1}{\upepsilon(\theta)}   \Obj(\theta) \Bigr)   \haG(\Phi)
 		+
			\haSigma(\Phi) \nabla^2\, \Obj(\theta)  \Bigr] \Bigl[ \frac{1}{\upepsilon(\theta)}   \Obj(\theta) \qsaprobe
			+
			 \qsaprobe \qsaprobe^\transpose \nabla\, \Obj(\theta)\Bigr]
											 +  O(\epsy)  
\]
While \Cref{t:orthCor} tells us that $\barUpupsilon    ( \theta)  \eqdef\Expect[\Upupsilon ( \theta,\Phi)  ]= 0$ for every $\theta$ under the assumptions of    \Cref{t:bigQSA+ESC},   the expression above  combined with
 \eqref{e:AllTheNoise0}   
and
 \eqref{e:AllTheNoise1}   
 suggest high volatility.

\subsubsection{Conclusions for 2qSGD}  

The conclusions above simplify greatly in this case:  
Letting $f_1$ denote the QSA vector field for 1qSGD,  and $f_2$   the QSA vector field for 2qSGD,  we have
\[
f_2(\theta,\qsaprobe)  = \half[ f_1(\theta,\qsaprobe) + f_1(\theta, -\qsaprobe) ]
\]
The following corollary to  \Cref{t:ESCapproximation} is immediate:

\begin{subequations} 
\begin{proposition}
\label[proposition]{t:ESCapproximation2}   
If $\Obj$ is analytic and  the frequencies $\{ \omega_i \}$ satisfy \eqref{e:logFreq}, 
we then have the following approximations for 2qSGD:  
\begin{align}
f( \theta,\qsaprobe)   & =  
			- \qsaprobe \qsaprobe^\transpose \nabla\, \Obj(\theta)   +  O(\epsy)   
&&   
\barf( \theta)   = - \Sigmaqsa\nabla\, \Obj(\theta)   +  O(\epsy^2)   
\label{e:ESCapprox_f2}
   \\
A( \theta,\qsaprobe)   & =   
 - \qsaprobe \qsaprobe^\transpose \nabla^2\, \Obj(\theta)   +  O(\epsy)  
&&
\barA( \theta)    =   - \Sigmaqsa \nabla^2\, \Obj(\theta)   +  O(\epsy^2)   
  \label{e:ESCapprox_A2} 
  \\
 \haf( \theta, \Phi)   & = 
			- \haSigma(\Phi) \nabla\, \Obj(\theta)   +  O(\epsy)   
&&     \hahaf( \theta,\Phi)    =  
			- \hahaSigma(\Phi) \nabla\, \Obj(\theta)   +  O(\epsy)   
\\
\Upupsilon ( \theta,\Phi) &=   -   
			\haSigma(\Phi) \nabla^2\, \Obj(\theta) \qsaprobe \qsaprobe^\transpose \nabla\, \Obj(\theta) 
											 +  O(\epsy)   &&
\label{e:ESCapprox_haf2} 
\end{align} 
where in each case the ratio $O(\epsy^p) /\epsy^p $ is uniformly bounded over $\epsy>0$, when $\theta$ is restricted to any compact subset of $\Re^d$.    
\end{proposition}
\end{subequations}

\clearpage


\end{document}

%% file: bookmacrosRL_Journal.tex
\def\urls#1{{\footnotesize\url{#1}}}

\def\mindex#1{\index{#1}}

\def\ocp{*}   

\def\Inv{\Uppi}  
\def\flow{\phi}  
\def\sstate{\textup{\textsf{S}}}
\def\bivstate{\textup{\textsf{Y}}}

\DeclareFontFamily{U}{mathx}{\hyphenchar\font45}
\DeclareFontShape{U}{mathx}{m}{n}{<-> mathx10}{}
\DeclareSymbolFont{mathx}{U}{mathx}{m}{n}
\DeclareMathAccent{\widebar}{0}{mathx}{"73}

\def\barUpupsilon{\widebar{\Upupsilon}}

\def\tilnabla{{\widetilde{\nabla}\!}}
\newcommand{\qsaprobe}{{\scalebox{1.1}{$\upxi$}}}  
\newcommand{\bfqsaprobe}{{\scalebox{1.1}{$\bm{\upxi}$}}}  
\def\Lip{L}  
 
\def\Obj{\Upgamma}  

\def\ODEstate{\Uptheta} 
\def\bfODEstate{\bm{\Uptheta}}
\def\barODEstate{\widebar{\Uptheta}}

\def\ODEstatePR{\ODEstate^{\text{\tiny\sf PR}}}

\def\odestate{\upvartheta}
\def\bfodestate{\bm{\upvartheta}}

\newcommand{\bbblot}{\raise1pt\hbox{\vrule height .4ex width .4ex depth .05ex}}

\long\def\defbox#1{\framebox[.9\hsize][c]{\parbox{.85\hsize}{%
\parindent=0pt
\baselineskip=12pt plus .1pt      
\parskip=6pt plus 1.5pt minus 1pt 
 #1}}}

\long\def\beginbox#1\endbox{\subsection*{}%
\hbox{\hspace{.05\hsize}\defbox{\medskip#1\bigskip}}%
\subsection*{}}

\def\endbox{}

 \def\archival#1{} 

\def\FRAC#1#2#3{\genfrac{}{}{}{#1}{#2}{#3}}

\def\ddt{{\mathchoice{\FRAC{1}{d}{dt}}%
{\FRAC{1}{d}{dt}}%
{\FRAC{3}{d}{dt}}%
{\FRAC{3}{d}{dt}}}}

\def\ddr{{\mathchoice{\FRAC{1}{d}{dr}}%
{\FRAC{1}{d}{dr}}%
{\FRAC{3}{d}{dr}}%
{\FRAC{3}{d}{dr}}}}

\def\ddtp{{\mathchoice{\FRAC{1}{d^{\hbox to 2pt{\rm\tiny +\hss}}}{dt}}%
{\FRAC{1}{d^{\hbox to 2pt{\rm\tiny +\hss}}}{dt}}%
{\FRAC{3}{d^{\hbox to 2pt{\rm\tiny +\hss}}}{dt}}%
{\FRAC{3}{d^{\hbox to 2pt{\rm\tiny +\hss}}}{dt}}}}

\def\ddyp{{\mathchoice{\FRAC{1}{d^{\hbox to 2pt{\rm\tiny +\hss}}}{dy}}%
{\FRAC{1}{d^{\hbox to 2pt{\rm\tiny +\hss}}}{dy}}%
{\FRAC{3}{d^{\hbox to 2pt{\rm\tiny +\hss}}}{dy}}%
{\FRAC{3}{d^{\hbox to 2pt{\rm\tiny +\hss}}}{dy}}}}

\def\half{{\mathchoice{\FRAC{1}{1}{2}}%
{\FRAC{1}{1}{2}}%
{\FRAC{3}{1}{2}}%
{\FRAC{3}{1}{2}}}}

\def\eqdist{\buildrel{\rm dist}\over =}

\def\tr{{\rm tr\, }}

\def\limsup{\mathop{\rm lim{\,}sup}}

\def\state{{\sf X}}

\def\bfmath#1{{\mathchoice{\mbox{\boldmath$#1$}}%
{\mbox{\boldmath$#1$}}%
{\mbox{\boldmath$\scriptstyle#1$}}%
{\mbox{\boldmath$\scriptscriptstyle#1$}}}}

\def\bfPhi{\bfmath{\Phi}}
\def\bfPsi{\bfmath{\Psi}}

\def\bftheta{\bfmath{\theta}}

\def\bfmu{\bfmath{u}}

\def\bfmY{\bfmath{Y}}

\def\bfmhhaY{\bfmath{\hhaY}} 
\def\bfmhhaY{\hbox to 0pt{$\widehat{\bfmY}$\hss}\widehat{\phantom{\raise 1.25pt\hbox{$\bfmY$}}}}

\def\bfmZ{\bfmath{Z}}

\def\haSigma{{\widehat{\Sigma}}}

\def\haf{{\hat f}}
\def\hag{{\hat g}}
\def\hah{{\hat h}}

\def\haA{\widehat A}

\def\tiltheta{{\tilde \theta}}

\def\tilg{\tilde g}

\def\clB{{\cal B}}

\def\clE{{\cal E}}
\def\clF{{\cal F}}

\def\clL{{\cal L}}
\def\clM{{\cal M}}

\def\clS{{\cal S}}
\def\clU{{\cal U}}

\def\clW{{\cal W}}

\def\clY{{\cal Y}}
\def\clZ{{\cal Z}}

\def\clL{{\cal L}}

\def\eqdef{\mathbin{:=}}

\def\Expect{{\sf E}}

\def\lgmath#1{{\mathchoice{\mbox{\large #1}}%
{\mbox{\large #1}}%
{\mbox{\tiny #1}}%
{\mbox{\tiny #1}}}}

\def\Zero{{\mathchoice{\lgmath{\sf 0}}%
{\mbox{\sf 0}}%
{\mbox{\tiny \sf 0}}%
{\mbox{\tiny \sf 0}}}}

 \def\epsy{\varepsilon}

\def\varble{\,\cdot\,}

\def\formtmp#1#2{{\vskip12pt\noindent\fboxsep=0pt\colorbox{#1}{\vbox{\vskip3pt\hbox to \textwidth{\hskip3pt\vbox{\raggedright\noindent\textbf{#2\vphantom{Qy}}}\hfill}\vspace*{3pt}}}\par\vskip2pt%
\noindent\kern0pt}}

\titleformat\subparagraph[runin]
                        {\normalfont\normalsize\bfseries}
                        {mypar}
                        {0pt}
                        {}{}
\titlespacing\subparagraph{0pt}
                       {.1ex minus 0.2ex}
                       {.75em}

\def\barf{{\widebar{f}}}
\def\barg{{\widebar{g}}}

\def\barh{{\overline {h}}}

\def\barA{{\bar{A}}}

\def\barY{{\bar{Y}}}
\def\barZ{{\bar{Z}}}

\def\barSigma{\overline{\Sigma}}

{\end{list}}

\def\ass(#1:#2){(#1\ref{#1:#2})}

\def\ritem#1{
\item[{\sf \ass(\current_model:#1)}]
}

\newenvironment{recall-ass}[1]{%
\begin{description}
\def\current_model{#1}}{
\end{description}
}

\def\sq{\hbox{\rlap{$\sqcap$}$\sqcup$}}
\def\qed{\ifmmode\sq\else{\unskip\nobreak\hfil
\penalty50\hskip1em\null\nobreak\hfil\sq
\parfillskip=0pt\finalhyphendemerits=0\endgraf}\fi}

\newcommand{\blot}{\vrule height 1.1ex width .9ex depth -.1ex }
\def\qedb{\ifmmode\blot\else{\vspace{-.2cm}\unskip\nobreak\hfil
\penalty50\hskip1em\null\nobreak\hfil\blot
\parfillskip=0pt\finalhyphendemerits=0\endgraf}\fi}

\newtheoremstyle{example}{15pt}{20pt}%
     {}
     {}
     {\bfseries}
     {}
     {1pt}
     {\thmname{#1}\thmnumber{ #2.}~\thmnote{\textit{\textbf{#3}}}%
     \\[.15cm]\unskip\nobreak}

\newcounter{rmnum}
\newenvironment{romannum}{\begin{list}{{\upshape (\roman{rmnum})}}{\usecounter{rmnum}
\setlength{\leftmargin}{18pt}
\setlength{\rightmargin}{8pt}
\setlength{\itemindent}{2pt}
}}{\end{list}}

\newcounter{anum}
\newenvironment{alphanum}{\begin{list}{{\upshape (\alph{anum})}}{\usecounter{anum}
\setlength{\leftmargin}{18pt}
\setlength{\rightmargin}{8pt}
\setlength{\itemindent}{2pt}
}}{\end{list}}

\newcommand{\field}[1]{\mathbb{#1}}

\def\Re{\field{R}} 

\def\intgr{\field{Z}}
\def\nat{\field{Z}_+}

\def\Co{\field{C}}

\def\Expect{{\sf E}}

\def\transpose{{\intercal}}

\def\diag{\hbox{\rm diag\thinspace}}

\def\trace{\hbox{\rm trace\,}}  
\def\epsy{\varepsilon}
\def\varble{\,\cdot\,}

\def\haY{\widehat{Y}}

\def\hhaY{\hbox to 0pt{$\haY$\hss}\widehat{\phantom{\raise 1.25pt\hbox{Y}}}}

\def\haA{\widehat A}

\def\haG{{\widehat G}}

\def\haY{\widehat Y}

\def\bfPhi{\bfmath{\Phi}}

\newlength{\dhatheight}

    \newcommand{\doublehat}[1]{%
    \settoheight{\dhatheight}{\ensuremath{\hat{#1}}}%
    \addtolength{\dhatheight}{-0.3ex}%
    \hat{\vphantom{\rule{1pt}{\dhatheight}}%
    \smash{\hat{#1}}}}
    
\def\hahaf{\doublehat{f}}

\def\scerror{Z}

\def\barUpupsilon{\widebar{\Upupsilon}}

\def\tilXi{\widetilde{\Xi}}

\def\prstate{\Upomega}

\def\tilnabla{{\widetilde{\nabla}\!}}
